\documentclass[11pt]{article}
\usepackage[margin=1in]{geometry}
\usepackage[utf8]{inputenc}
\usepackage{amsmath,amssymb,amsthm}
\usepackage[dvipsnames]{xcolor}
\usepackage{hyperref}
\hypersetup{colorlinks,allcolors=[rgb]{0,0.13672,0.95703}}
\usepackage{cleveref}
\usepackage[affil-sl]{authblk}
\usepackage{graphicx}
\usepackage{cite}
\usepackage{diagbox}
\usepackage{cancel}
\usepackage{algorithm}
\usepackage{algpseudocode}
\usepackage{soul} 

\usepackage{mathtools}
\usepackage{relsize}
\usepackage{xspace}
\usepackage{wrapfig}
\usepackage{multirow, booktabs, comment}

\usepackage{pgfplots}  
\pgfplotsset{compat=1.16}

\usepackage{pifont}
\usepackage{lipsum}
\usepackage{amsfonts}
\usepackage{epstopdf}
\usepackage{dsfont}
\usepackage{changepage}

\newcommand\tr{{\operatorname{tr}}}
\newcommand\diag{{\operatorname{diag}}}
\newcommand\E{{\mathbb{E}}}
\newcommand\C{{\mathbb{C}}}

\renewcommand\v[1]{{\boldsymbol{#1}}}

\newcommand{\R}[0]{\mathbb{R}}

\newtheorem{theorem}{Theorem}[section]
\newtheorem{lemma}[theorem]{Lemma}
\newtheorem{corollary}[theorem]{Corollary}
\newtheorem{definition}[theorem]{Definition}

\newtheorem{remark}[theorem]{Remark}

\numberwithin{equation}{section}

\newcommand{\cmark}{\textcolor{green!80!black}{\ding{51}}}%
\newcommand{\xmark}{\textcolor{red}{\ding{55}}}%

\let\origtop\top
\renewcommand\top{{\scriptscriptstyle{\origtop}}} 

\DeclareMathOperator*{\argmin}{\mathop{\mathrm{argmin}}}

\def\diag{\mathrm{diag}}
\def\rank{\mathrm{rank}}
\def\DPP{\mathrm{DPP}}

\def\RDPP{\mathrm{R\textnormal{-}DPP}}

\def\x{{\mathbf x}}
\def\z{{\mathbf z}}
\def\w{{\mathbf w}}
\def\m{{\mathbf m}}
\def\e{{\mathbf e}}
\def\v{{\mathbf v}}
\def\u{{\mathbf u}}
\def\r{{\mathbf r}}
\def\y{{\mathbf y}}

\def\reg{{\lambda}}
\def\Sig{{\mathbf{\Sigma}}}
\def\range{{\mathrm{range}}}

\def\b{{\mathbf b}}
\def\a{{\mathbf a}}

\def\A{{\mathbf A}}
\def\B{{\mathbf B}}

\def\Lb{{\mathbf L}}

\def\H{{\mathbf H}}
\def\I{{\mathbf I}}
\def\Q{{\mathbf Q}}
\def\C{{\mathcal C}}
\def\Bc{{\mathcal B}}
\def\Dc{{\mathcal D}}
\def\Uc{{\mathcal U}}

\def\mPhi{{\mathbf \Phi}}
\def\mPi{{\mathbf \Pi}}
\def\M{{\mathbf M}}

\def\Z{{\mathbf Z}}
\def\D{{\mathbf D}}
\def\SymFHT{{\mathrm{SymFHT}}}
\def\FHT{{\mathrm{FHT}}}

\def\setS{S}

\def\Ec{{\mathcal E}}

\def\V{{\mathbf V}}
\def\P{{\mathbf P}}
\def\E{{\mathds E}}

\def\R{{\mathds R}}
\def\Rb{{\mathbf R}}

\def\U{{\mathbf U}}

\renewcommand{\C}{\mathbf{C}}

\def\tr{\mathrm{tr}}

\def\alg{%
\text{Kaczmarz\raisebox{0.325ex} {\relscale{0.75}++}}\xspace}
\def\algpd{{CD\raisebox{0.325ex} {\relscale{0.75}++}}\xspace}
\def\algpdaccel{{CD\raisebox{0.325ex} {\relscale{0.75}+}}Accel\xspace}

\def\algshort{%
\texttt{K++}\xspace}
\def\algmemo{%
\texttt{K++ w/o Accel}\xspace}
\def\algaccel{%
\texttt{K++ w/o Memo}\xspace}

\title{Randomized Kaczmarz Methods with Beyond-Krylov Convergence\thanks{This work was funded by NSF grant CCF-2338655 (MD and JY), NSF grant DMS-2408912 (DN), and NSF grant DMS-2309685 (ER).}}

\author{
Micha{\l} Derezi\'nski\thanks{University of Michigan (\texttt{derezin@umich.edu})}
\quad
Deanna Needell\thanks{University of California, Los Angeles (\texttt{deanna@math.ucla.edu})}
\quad
Elizaveta Rebrova\thanks{Princeton University (\texttt{elre@princeton.edu})}
\quad
Jiaming Yang\thanks{University of Michigan (\texttt{jiamyang@umich.edu})}}

\begin{document}

\maketitle

\begin{abstract}
Randomized Kaczmarz methods form a family of linear system solvers which converge by repeatedly projecting their iterates onto randomly sampled equations. While effective in some contexts, such as highly over-determined least squares, Kaczmarz methods are traditionally deemed secondary to Krylov subspace methods, since this latter family of solvers can exploit~outliers in the input's singular value distribution to attain fast convergence on ill-conditioned~systems. 

In this paper, we introduce \alg, an accelerated randomized block Kaczmarz algorithm that exploits outlying singular values in the input to attain a fast Krylov-style convergence. Moreover, we show that \alg\ captures large outlying singular values provably faster than popular Krylov methods, for both over- and under-determined systems. We also develop an optimized variant for positive semidefinite systems, called \algpd, demonstrating empirically that it is competitive in arithmetic operations with both CG and GMRES on a collection of benchmark problems. To attain these results, we introduce several novel algorithmic improvements to the Kaczmarz framework, including adaptive momentum acceleration, Tikhonov-regularized projections, and a memoization scheme for reusing information from previously sampled equation~blocks.
\end{abstract}

\newpage
\tableofcontents
\newpage

\section{Introduction}

The Kaczmarz method \cite{Kac37:Angenaeherte-Aufloesung} is an iterative algorithm for solving large linear systems of equations, which has found many applications \cite{natterer2001mathematics,feichtinger1992new,herman1993algebraic} due to its simple and memory-efficient updates that operate on a single equation at a time. Numerous variants of this method have been studied, most notably incorporating randomized equation selection to enable rigorous convergence analysis (Randomized Kaczmarz, \cite{SV09:Randomized-Kaczmarz}), as well as block updates \cite{elfving1980block,needell2014paved} that operate on multiple equations at a time to better balance memory and computations. Notably, Randomized Kaczmarz can be viewed as an instance of Stochastic Gradient Descent (SGD), and this connection has led to weighted sampling schemes for
SGD \cite{needell2014stochastic}.

Kaczmarz(-type) methods have proven effective when the linear system is highly over-determined and the computing environment restricts access to the input data, thus benefiting from their cheap and localized updates. However, outside of these considerations, Krylov subspace methods such as Conjugate Gradient (CG) \cite{hestenes1952methods}, LSQR \cite{paige1982lsqr}, and GMRES \cite{saad1986gmres} typically appear (theoretically) superior to Kaczmarz methods on account of their ability to exploit outliers and clusters in the input's singular value distribution, attaining fast convergence even for some highly ill-conditioned systems (e.g., see Chapter 5 of \cite{nocedal1999numerical}).
In this work, we re-examine this assertion, developing Kaczmarz methods that can similarly exploit outlying singular values in the input to achieve fast convergence, going even beyond what Krylov subspace methods can attain for a natural class of singular value distributions. Crucially, our proposed methods do not require the systems to be very tall or even over-determined to work well.

To achieve this beyond-Krylov convergence, we develop \alg, a randomized block Kaczmarz method that incorporates several novel algorithmic techniques: adaptive acceleration, regularized projections, and block memoization. 
Illustrating our claims, let us focus first on solving a square $n\times n$ linear system $\A\x=\b$ for $\A\in\R^{n\times n}$ and $\b\in\R^n$, but extensions to rectangular under- and over-determined systems are provided in the later sections. Given $1\leq k\leq n$, we show that \alg\ with block size proportional to $k$ solves such a system to within $\epsilon$ error in:
\begin{align}
\text{(\alg, Thm.~\ref{thm:main})}\qquad
\underbrace{\tilde O\big(n^2+nk^2\big)}_{\text{Phase 1}}\ +\ 
\underbrace{\tilde O\big(n^2\bar\kappa_k\log1/\epsilon\big)}_{\text{Phase 2}}\quad\text{operations},\label{eq:intro-kaczmarz}
\end{align}
where $\tilde O$ hides logarithmic factors described in detail alongside the theorem, while $\bar\kappa_k$ is the normalized Demmel condition number of the matrix $\A$ excluding the top-$k$ part of its singular value decomposition (SVD):
\begin{align}\label{eq:kappa_k_def}
\bar\kappa_k:=\bar\kappa(\A-\A_k),\qquad\text{for}\qquad \bar\kappa(\M) := \|\M\|_F\|\M^\dagger\|/\sqrt{\rank(\M)}.
\end{align}
Here, $\A_k=\sum_{i=1}^k\sigma_i\u_i\v_i^\top$ denotes the top-$k$ part of $\A$'s SVD, and $\|\cdot\|_F$ is the matrix Frobenius norm. Demmel condition number $\|\M\|_F\|\M^\dagger\| = \|\M\|_F/\sigma_{\min}^+(\M)$ is often used to describe the convergence rate of Kaczmarz methods, starting from \cite{SV09:Randomized-Kaczmarz}, and when normalized by $\sqrt{\rank(\M)}$, it is always upper-bounded by the classical condition number $\kappa(\M):=\|\M\|\|\M^\dagger\|$.

As suggested by \eqref{eq:intro-kaczmarz}, \alg\ exhibits two phases of convergence: In the first phase, the algorithm implicitly \emph{captures the top-$k$ part of $\A$'s spectrum} via our block memoization scheme that requires only $\tilde O(nk)$ additional memory. Then, in the second phase, it uses this information (along with our adaptive acceleration scheme) to attain a convergence rate that is \emph{independent of the top-$k$ singular values of $\A$}.

To put this result in context, consider a comparable convergence guarantee achievable by a Krylov subspace method (such as LSQR) for solving a dense $n\times n$ ill-conditioned linear system with $k$ large outlying singular values. Such a method is also expected to exhibit two phases of convergence,
where the first phase builds the Krylov subspace that captures the top-$k$ part of $\A$'s SVD, while the second phase leverages it to attain fast convergence, reaching $\epsilon$ error after:
\begin{align}
\text{(Krylov, e.g., \cite{axelsson1986rate})}\qquad \underbrace{O\big(n^2k\big)}_{\text{Phase 1}} \ +\  \underbrace{O\big(n^2\kappa_k\log1/\epsilon\big)}_{\text{Phase 2}}\quad\text{operations}.
\label{eq:intro-krylov}
\end{align}
Here, $\kappa_k\!:=\!\kappa(\A-\A_k)$ is the condition number of $\A$ excluding its top-$k$ singular~values. We note that \eqref{eq:intro-krylov} is obtained in exact arithmetic, and in practice, tends to require re-orthogonalization of the Krylov basis (as we observe empirically in Section~\ref{s:experiments}). 
Also, since our theoretical analysis of \alg\ is limited to dense matrices, we focus on this setting here. Naturally, Krylov methods can attain much lower cost in the sparse setting (as can some variants of Kaczmarz, see below for further considerations).

There are two key differences between the guarantees \eqref{eq:intro-kaczmarz} and \eqref{eq:intro-krylov}. First, the fast convergence of \alg\ relies on $\bar\kappa_k$, which can be substantially smaller than~$\kappa_k$. Second, and more importantly, the first phase of \alg\ takes only $\tilde O(n^2+nk^2)$ time, which can be much more efficient than the $O(n^2k)$ first phase of Krylov. 
This is because building the Krylov subspace for the outlying singular values requires at least $k$ matrix-vector products with the full matrix~$\A$. Lower bounds show that this cost is necessary for any algorithm based solely on matrix-vector products \cite{derezinski2024fine} (including all Krylov methods) when solving systems with $k$ large outlying singular values. On the other hand, in its initial phase, \alg\ leverages direct access to $\A$ by iterating over $\tilde O(n/k)$ blocks of $\tilde O(k)$ rows each, a computational equivalent of only a few matrix-vector products. This benefit of \alg\ is significant when $k$ is sufficiently larger than a logarithmic power of $n$ and sufficiently smaller~than~$n$.

Naturally, the above guarantees do not provide a complete convergence comparison between \alg\ and Krylov subspace methods (e.g., they do not account for \textit{small} outlying singular values). However, they indicate that Kaczmarz methods can work well on ill-conditioned problems and are competitive with Krylov solvers in terms of arithmetic operations even for square linear systems, and not only for highly over-determined ones, as is often suggested. To verify these claims empirically, we develop a practical implementation of \alg, performing an ablation study on ill-conditioned synthetic matrices with large outlying singular values. We also develop another variant of our algorithm which is specifically optimized for positive semidefinite systems (\algpd, Algorithm \ref{alg:bcd}) and enjoys an improved convergence that scales with $\sqrt{\bar\kappa_k}$ instead of $\bar\kappa_k$ (see Theorem \ref{t:cd++}). We test \algpd\ on a collection of benchmark positive definite problems from the machine learning literature \cite{OpenML2013,scikit-learn}, which are known to exhibit large outlying eigenvalues. These experiments confirm our theoretical findings, showing that Kaczmarz methods can be competitive in floating point operations with both CG and GMRES on a range of input matrices.

\paragraph{\textbf{Main contributions.}} 
As part of \alg, we introduce several novel algorithmic techniques to the broader Kaczmarz toolbox, which are crucial both for the convergence analysis and the numerical performance.
\begin{enumerate}
    \item \textit{Adaptive acceleration:} We propose a new way of introducing Nesterov's momentum into Kaczmarz updates, which is both stable with respect to its hyper-parameters and can be adaptively tuned during runtime (Section \ref{s:reformulation}).
    \item \textit{Regularized projections:} We add Tikhonov regularization to the classical Kaczmarz projection steps, showing that it not only makes them better-conditioned, but also reduces the variance coming from randomization, enabling our convergence analysis (Section \ref{s:rate}).
    \item \textit{Block memoization:} Our algorithm saves and reuses small Cholesky factors associated with the sampled equation blocks, thereby speeding up subsequent iterations in the second phase of the convergence, while at the same time maintaining a low memory footprint (Section \ref{s:blocks}).
    \item \textit{Symmetric Hadamard transform:} A key step in our algorithm is to preprocess the linear system with a randomized Hadamard transform. As an auxiliary result, we give a new recursive scheme for applying the Hadamard transform to symmetric matrices which reduces arithmetic operations by half (Section~\ref{s:practical}).
\end{enumerate}

\paragraph{\textbf{Further computational considerations and limitations.}}
While we focus our computational analysis on floating point operations, this metric may not fully capture the true running time, particularly when we wish to exploit sparsity in the input or parallelization in the hardware. Krylov subspace methods benefit from relying primarily on full matrix-vector product operations with $\A$ and can attain much faster performance, particularly for sparse and structured matrices.
Below, we discuss how these considerations may affect the performance of \alg\ (and \algpd).
\begin{enumerate}
    \item \textit{Parallelization:} While Kaczmarz methods often operate on single rows of the matrix at a time, \alg\ is specifically optimized for relying on large blocks of $\tilde O(k)$ rows. Since its dominant operation is typically an $ \tilde O(k) \times n$ matrix-vector product, choosing a large $k$ not only improves the conditioning via $\bar\kappa_k$ but also helps the method take advantage of hardware parallelization.
    \item \textit{Sparsity:} A limitation of our theoretical analysis is that it only applies to solving dense linear systems. This is because randomized Hadamard  preprocessing, which ensures an incoherence property that is needed for our theory, does not preserve the sparsity of the input. In Section~\ref{s:experiments}, we show empirically  that in some cases our methods still perform well without this step, since the input matrix $\A$ may already satisfy the needed incoherence property to begin with. Under these conditions, \alg\ can be implemented to run in time $\tilde O(nk^2+ \mathrm{nnz}(\A)\bar\kappa_k\log1/\epsilon)$, where $\mathrm{nnz}(\A)$ is the number of non-zeros in $\A$. In comparison, the cost of Krylov with orthogonalization to attain the same guarantee is $\tilde O(nT^2 + \mathrm{nnz}(\A) T)$ for $T=O(k+\kappa_k\log1/\epsilon)$. Nevertheless, future work (both theoretical and empirical) could assess the effectiveness of our algorithms on sparse problems.
    \item \textit{Parameter tuning:}
    The theoretically analyzed variant of \alg\ relies on hyper-parameters controlling momentum acceleration and block memoization. In the implemented variants of \alg\ and \algpd, we use adaptive tuning to find these parameters during runtime. However, our algorithms still require correct selection of the block size (which is proportional to~$k$) so that they can take full advantage of our guarantee \eqref{eq:intro-kaczmarz}. This stands in contrast to the Krylov guarantee \eqref{eq:intro-krylov} which holds for all $k$ simultaneously.
\end{enumerate}

\subsection{Overview of the Main Algorithm}\label{s:alg}
In this section, we motivate and derive our \alg\ algorithm, describing how the simultaneous use of fast preprocessing, adaptive acceleration,  regularized projections, and block memoization
enable it to achieve the claimed convergence guarantees. 

Consider solving a consistent linear system with $m$ equations, $\A\x=\b$, where $\A\in\R^{m\times n}$ and $\b\in\R^{m}$. The classical block Kaczmarz method constructs a sequence of iterates $\x_0, \x_1,...$ by repeatedly choosing a subset $S=S(t)\subseteq \{1,...,m\}=:[m]$ of those equations, and then projecting the current iterate $\x_t$ onto the subspace of solutions of those equations, i.e., the under-determined system $\A_{S}\x=\b_S$. This leads to the block Kaczmarz update which can be stated as follows:
\begin{align*}
    \x_{t+1} = \argmin_{\x:\,\A_{S}\x=\b_{S}}\!\!\|\x-\x_t\|^2\ =\ \x_t - \A_S^\dagger(\A_S\x_t-\b_S).
\end{align*}

\paragraph{\textbf{Randomized preprocessing.}} Effective selection of the subset $S$ in each iteration is crucial for obtaining fast convergence of the block Kaczmarz method, and randomization has been suggested as an effective strategy to diversify the selection process. Here, one approach, following the original Randomized Kaczmarz method~\cite{SV09:Randomized-Kaczmarz}, is to use importance sampling that emphasizes equations with large row norms. However, it has proven difficult to characterize the correct importance weights that ensure provably fast convergence of block Kaczmarz. So, we opt for a different strategy: preprocessing the linear system using a Randomized Hadamard Transform \cite{ailon2009fast,tropp2011improved}.
\begin{definition}\label{d:rht}
    An $m\times m$ randomized Hadamard transform 
    (RHT) is a matrix $\Q=\H\D$, where $\H$ is the Hadamard matrix and $\D$ is an $m\times m$ diagonal matrix with random $\pm1/\sqrt m$ entries. Applying $\Q$ to a vector takes $m\log m$ arithmetic operations.
\end{definition}
We note that the Hadamard matrix, similarly to a Discrete Fourier Transform (DFT), is a scaled orthogonal matrix (specifically, $\H^\top\H=m\I_m$) that admits fast matrix-vector multiply  (see Appendix \ref{s:symfht} for a detailed discussion).
In fact, these are the only two properties we use in our analysis, and one could replace Hadamard with a DFT or other fast transforms.

Since $\Q^\top\Q=\I$, transforming the system $\A\x=\b$ into $\Q\A\x=\Q\b$ does not affect the solution of the system nor does it affect the singular value distribution of the input matrix. However, crucially, it ensures that all equations become roughly equally important (this is known as \emph{incoherence}), which allows us to select a representative subset $S$ uniformly at random. Using the fast matrix-vector multiply, this transformation can be done using $mn\log m$ arithmetic operations. However, the resulting matrix $\Q\A$ does not retain the structural properties of $\A$ such as its sparsity pattern. Thus, while important for parts of our theoretical analysis, the RHT may be skipped when the input matrix is sparse and naturally exhibits an incoherence property \cite{needell2014paved}.

\paragraph{\textbf{Regularized projections.}}
Even after preprocessing with the RHT (and especially if this is omitted), a randomly selected sub-matrix $\A_S$ may be poorly conditioned, which adversely affects the performance of block Kaczmarz, especially if one chooses to solve the block system using an iterative method. To guard against this, instead of the true projection step, we consider a \emph{regularized} projection, defined as the following regularized least squares problem:
\begin{align}
    \x_{t+1} &= \argmin_{\x\in\R^n}\Big\{\|\A_S\x - \b_S\|^2 + \reg\|\x-\x_t\|^2\Big\}\nonumber
    \\
    &= \x_t - \A_S^\top(\A_S\A_S^\top+\reg\I)^{-1}
    (\A_S\x_t-\b_S) =: \x_t - \w_t.\label{eq:bk}
\end{align}
Note that by letting $\reg=0$, this formulation recovers standard block Kaczmarz, however a positive $\reg$ ensures that the sub-problem being solved is not too ill-conditioned. This plays a crucial role in the convergence analysis of our method (see Section \ref{s:rate}), and it also improves the stability of solving the sub-problem (see Section \ref{s:blocks}). 

\paragraph{\textbf{Block memoization.}} 
Even with the regularization, the cost of the projection step in each iteration of \alg\ is still a substantial computational bottleneck, as it requires computing or approximately applying the inverse of $\A_S\A_S^\top+\lambda\I$, e.g., via its Cholesky factor, $\Rb=\mathrm{chol}(\A_S\A_S^\top+\lambda\I)$. In Section \ref{s:fast-projection}, we make the projection steps even more efficient by computing an approximation of the Cholesky factor, $\tilde\Rb\approx \Rb$, and combining this with an inner solver.
Finally, to further amortize these costs over the entire convergence run of the algorithm, we store and reuse the Cholesky factors computed in early iterations via what we refer to as \emph{block memoization}. 

To enable block memoization, it is crucial that the algorithm draws its blocks 
 from a small collection $\Bc$ of previously sampled block sets, so that we can reuse a previously computed Cholesky factor $\tilde\Rb[S]$ for a set $S\in\Bc$. However, this comes with a trade-off: we should expect the convergence rate attained by the algorithm to get worse as we restrict the method to a smaller collection of blocks. Fortunately, we show in Section \ref{s:block-memoization} that when using blocks of size $\tilde O(k)$, it suffices to sample a collection $\Bc$ of $\tilde O(\frac mk)$ random blocks, which can then be continually reused while retaining the same convergence guarantees as if we sampled fresh random blocks at every step. This scheme requires only $\tilde O(mk)$ additional memory for storing the Cholesky factors.

\paragraph{\textbf{Adaptive acceleration.}} A key limitation of the Kaczmarz update is that it does not accumulate any information about the trajectory of its convergence, which could be used to accelerate it. This stands in contrast to, for instance, Krylov methods such as CG, as well as momentum-based methods such as accelerated gradient descent (AGD), which use the information from all previous update directions to construct the next step. To address this, we develop a new way of introducing momentum into the block Kaczmarz update through a careful reformulation of Nesterov's momentum that is both theoretically principled and practically effective.

To explain how we accelerate block Kaczmarz using momentum, we first describe the AGD algorithm, following Nesterov \cite{nesterov2013introductory}. In the context of solving a linear system, AGD can be viewed as minimizing the convex quadratic $f(\x) = \frac12\x^\top\A^\top\A\x - \x^\top\A^\top\b$, via the iterative update $\x_{t+1} = \x_t - \w_t + \m_{t+1}$, where $\w_t=\alpha\nabla f(\x_t)$ is the gradient step, and $\m_t$ is the momentum step:
\begin{align}
    \m_{t+1} = \frac{1-\rho}{1+\rho}(\m_t - \w_t),\qquad\rho\in[0,1].\label{eq:momentum}
\end{align}
We note that the above scheme is precisely the AGD scheme (2.2.11) in \cite{nesterov2013introductory}, obtained by substituting their $y_k$ with our $\x_t$, initializing $\m_0=\mathbf{0}$, and setting $\alpha=1/L$, $\rho=\sqrt{\mu/L}$ using their $\mu$ and $L$. Here, parameter $\rho$ 
is also the theoretical convergence rate of AGD: one can show that $\|\x_t-\x^*\|^2 \leq C(1-\rho)^t\|\x_0-\x^*\|^2$ for an appropriate problem-dependent $C>0$ (Theorem 2.2.3 in \cite{nesterov2013introductory}). 

Based on these observations, we replicate the above acceleration procedure for block Kaczmarz, by setting $\w_t$ in the momentum recursion \eqref{eq:momentum} to be the regularized projection step \eqref{eq:bk}, and adjusting $\rho$ to be the target convergence rate of our method. However, this does not fully take into account the stochasticity of the block Kaczmarz update, which operates only on a fraction of the matrix $\A$ at a time. This forces us to curb the momentum step further, by introducing an additional step size $\eta$, which should be proportional to the ratio between the block size and the rank of $\A$:
\begin{align*}
    \x_{t+1} = \x_t - \w_t + \eta\,\m_{t+1},\qquad\eta\in[0,1].
\end{align*}

In Section \ref{s:reformulation}, we show that the above accelerated update is remarkably stable with respect to both $\eta$ and $\rho$. Further, motivated by this analysis, in Section \ref{s:practical} we propose an adaptive scheme that periodically updates $\rho$ at runtime using the current estimate of the convergence rate of the algorithm. We observe that this mechanism exhibits a self-correcting feedback loop that quickly arrives at a near-optimal convergence.

Combining the above ideas, we obtain \alg\ (Algorithm \ref{alg:main}). Building on this, in Section \ref{s:practical} we propose \algpd\ (Algorithm \ref{alg:bcd}), a coordinate descent-type algorithm derived out of \alg, which is optimized for positive definite systems.

\begin{algorithm}[!ht] 
\caption{\alg}
\label{alg:main}
\begin{algorithmic}[1]
\State \textbf{Input: } $\A\in\R^{m\times n}$, $\b\in\R^m$, block size $s$, iterate $\x_0$, parameters $B$, $\reg$, $\rho$, $\eta$;
\State Initialize $\m_0 \leftarrow \mathbf{0}$;
\State Compute $\A \leftarrow \Q\A$ and $\b \leftarrow\Q\b$; \Comment{{\footnotesize Preprocessing with RHT $\Q$.}} 
\State Sample $\Bc\leftarrow\{S_1,S_2,...,S_B\}$ where $S_i\sim {[m]\choose s}$; \Comment{{\footnotesize Prepare $B$ index subsets.}}
  \For{$t = 0, 1, \ldots$}
\State Draw a random $S$ from $\Bc$;  \label{line:blocks}
\State \textbf{if} $\tilde\Rb[S]=\text{null}$ \ \textbf{then} \ $\tilde\Rb[S] \approx \text{chol}(\A_S\A_S^\top+\lambda\I)$;
\Comment{{\footnotesize Save Cholesky factor.}}
\State $\tilde\w_t \approx \A_S^\top(\A_S\A_S^\top+\lambda\I)^{-1}
    (\A_S\x_t-\b_S)$ using $\tilde\Rb[S]$;  \Comment{{\footnotesize Regularized projection.}}\label{line-w}
\State $\m_{t+1} \leftarrow \frac{1-\rho}{1+\rho}(\m_t - \tilde\w_t)$;
\Comment{{\footnotesize Nesterov momentum.}}\label{line:momentum}
\State $\x_{t+1} \leftarrow \x_t - \tilde\w_t + \eta\,\m_{t+1}$;
\Comment{{\footnotesize \alg\ update.}}\label{line:update}
\State Revise convergence rate estimate $\rho$;
\Comment{{\footnotesize Adaptive acceleration.}}
\EndFor \\
\Return $\tilde\x = \x_t$; \Comment{{\footnotesize Solves $\A \x = \b$.}}
\end{algorithmic}
\end{algorithm}

\subsection{Related work}
\label{s:related}
The block Kaczmarz method was first introduced by \cite{elfving1980block}, motivated by applications in image reconstruction \cite{eggermont1981iterative}. Later, a randomized implementation of block Kaczmarz was developed and analyzed by \cite{needell2014paved}, who showed that under some assumptions on the row-norms and spectral norm of the input matrix, after preprocessing it with a randomized Hadamard transform, a sufficiently large random partition of an $n\times n$ linear system into $n/s$ blocks of size $s$ leads to a convergent block Kaczmarz method. However, even if we ignore their assumptions on the input matrix (which are not needed in our work), those convergence guarantees scale with the squared condition number $\kappa^2(\A)$, and thus do not exploit the singular value distribution as shown in \eqref{eq:intro-kaczmarz}.

An alternative block-construction strategy, first proposed by \cite{gower2015randomized}, is to transform the matrix $\A$ via a random \emph{sketching} matrix $\mPi\in \R^{s\times n}$, so that $\mPi\A$ is no longer a subset of equations, but rather a collection of linear combinations of equations. A simple choice is to use a Gaussian matrix or a sparse random sign matrix $\mPi$. These approaches offer a finer control on the quality of sampled blocks, but at the expense of substantially larger cost, since one must compute a new sketch $\mPi\A$ at every iteration of the algorithm. \cite{rebrova2021block} were the first to characterize the convergence rate of Block Kaczmarz with Gaussian sketches, however their convergence also scales with $\kappa^2(\A)$ and does not exploit outlying singular values. In \cite{rebrova2021block}, a sketch memoization idea  was also proposed, in the form of sampling from a set of precomputed sketches, although they require as many as $\tilde O(n^2)$ precomputed sketches to ensure convergence.

Recently, there has been a number of works suggesting that variants of block Kaczmarz can exploit $k$ large outliers in the singular value distribution, by expressing its convergence in terms of $\bar\kappa_k$ or $\kappa_k$. \cite{derezinski2024sharp} were the first to show this, however due to the large cost of Gaussian sketching and lack of acceleration, their method does not have a computational benefit over existing approaches. Then, \cite{derezinski2023solving} showed that a similar convergence guarantee can be obtained by block Kaczmarz with uniformly sampled blocks, after RHT preprocessing. Their algorithm converges in $\tilde O((n^2 + nk^2)\bar\kappa_k^2\log 1/\epsilon)$ operations. Most recently, \cite{derezinski2024fine} obtained $\tilde O((n^2+nk^2)\kappa_k\log1/\epsilon)$ operations by introducing momentum acceleration via a Nesterov-style scheme of \cite{gower2018accelerated} which is closely related to the momentum acceleration we use in \alg. However, in order to provide a theoretical convergence analysis of their algorithm, \cite{derezinski2024fine} have to rely on a sketching-based block-construction strategy, which is much slower and less practical than uniform sampling. We resolve this theoretical limitation, and are able to use uniform block sampling in \alg, by introducing regularized projections which facilitate our acceleration analysis, as discussed in Section \ref{s:nu}.

An alternative variant of block Kaczmarz that exploits large outlying singular values was recently proposed in~\cite{lok2024subspace}. This algorithm first finds a subsystem spanning the leading subspace of the system, and then projects future iterates onto its solution subspace. However, this method requires some external knowledge about the leading subspace to attain a computational advantage over the other approaches. Several other recently developed Kaczmarz methods \cite{alderman2024randomized,epperly2024randomized,patel2023randomized} introduce acceleration and/or adaptivity, but do not provably exploit large outlying singular values.

Our \alg, which requires $\tilde O(nk^2+n^2\bar\kappa_k\log 1/\epsilon)$ operations to converge, improves on all of these recent prior results in two primary ways: 1) It is the only one to achieve the correct condition number dependence, scaling with $\bar\kappa_k$ as opposed to $\bar\kappa_k^2$ or $\kappa_k$; and 2) It is the only block Kaczmarz method to exhibit two distinct phases of convergence, in the sense that the $\tilde O(n^2+nk^2)$ cost of learning the outlying singular values is not incurred for the entire $\tilde O(\bar\kappa_k\log1/\epsilon)$ length of the convergence. This improvement, a result of our block memoization scheme and its analysis, allows \alg\ to always match or improve upon the Krylov convergence guarantee~\eqref{eq:intro-krylov}. 

\subsection{Notation}
We let $[m]:=\{1,...,m\}$, whereas ${[m]\choose s}$ denotes all size $s$ subsets of $[m]$. For a matrix $\A\in\R^{m\times n}$ and subset $S\in{[m]\choose s}$, we use $\A_S\in\R^{s\times n}$ to denote the submatrix of the rows of $\A$ indexed by $S$, and if $m=n$, then $\A_{S,S}\in\R^{s\times s}$ is the principal submatrix indexed by $S$. We use $\|\A\|$, $\|\A\|_F$, and $\A^\dagger$ to denote the spectral/Frobenius norms of matrix $\A$ as well as its Moore-Penrose pseudoinverse, while $\lambda_{\max}(\A)$ and $\lambda_{\min}^+(\A)$ are the largest and smallest positive eigenvalues of an $n\times n$ positive semidefinite (PSD) matrix, denoted $\A\in\mathcal{S}_n^{+}$. For $\A\in\mathcal{S}_n^{+}$ and $\v\in\R^n$, we use $\|\v\|_{\A} = \sqrt{\v^\top\A\v}$, and for symmetric matrices, $\A\preceq\B$ means that $\B-\A\in\mathcal{S}_n^{+}$. For an event $\mathcal{E}$, we use $\neg \mathcal{E}$ to denote its negation. We use $C>0$ to denote an absolute constant, which may change from line to~line. In particular, throughout the paper we assume that $C \geq 64$ whenever requiring that $k \geq C\log (m/\delta)$ given failure probability $\delta$. This assumption is only for the convenience of selecting other constants.

\subsection{Organization}
In Section \ref{s:reformulation} we perform the convergence analysis of our accelerated Kaczmarz update as a function of the randomized block selection, verifying its stability with respect to the momentum parameters $\rho$ and $\eta$. Then, in Section \ref{s:rate} we characterize the convergence rate under uniform block sampling after RHT, in terms of the regularizer $\lambda$. In Section \ref{s:blocks}, we introduce and analyze block memoization, together with a discussion of the overall computational cost of \alg. Finally, Section~\ref{s:practical} describes our specialized \algpd\ algorithm for positive semidefinite linear systems, and Section~\ref{s:experiments} has numerical experiments. We give conclusions in Section \ref{s:conclusions}.

\section{Stable Convergence with Adaptive Acceleration}
\label{s:reformulation}
In this section, we establish how the convergence of \alg\ (Algorithm~\ref{alg:main}) depends on the properties of the randomized block selection scheme. This guarantee does \emph{not} assume that the system was preprocessed with a Randomized Hadamard Transform. Moreover, it applies to all consistent linear systems, regardless of aspect ratio, and most block sampling~schemes. Main theoretical result of this section is the following theorem:

\begin{theorem}[General convergence rate]\label{lem:converge_accelerate}
 Given $\A\in\R^{m\times n}$ and $\b\in\R^m$, let $\x^*$ be the minimum-norm solution to $\A\x=\b$. Also, let $\mathcal{D}$ be a probability distribution over subsets of $[m]$. Given $\lambda\geq 0$, define the random regularized projection:
\begin{align*}
    \P_{\lambda,S} \coloneqq \A_S^\top(\A_S\A_S^\top+\lambda\I)^{\dagger}\A_S,\quad S\sim\mathcal{D},
\end{align*}
and suppose that $\bar\P_\lambda := \E[\P_{\lambda,S}]$ has the same null space as $\A$. Define the following:
\begin{align}
\mu :=  \mu(\A,\Dc,\reg) &= \lambda_{\min}^+\big(\bar\P_\lambda\big), \nonumber\\
\nu := \nu(\A,\Dc,\reg) &= \lambda_{\max}\Big(\E\big[(\bar\P_\lambda^{\dagger/2}\P_{\lambda,S}\bar\P_\lambda^{\dagger/2})^2\big]\Big), \nonumber \\
\text{and}\quad\bar\rho := \bar\rho(\A,\Dc,\lambda) &= \sqrt{\frac{\mu}{\nu}}.
\label{eq:rho-mu-nu}
\end{align}
Let 
$\rho\in[0,c\bar\rho]$ and $\eta\in [\frac c{\nu},\frac1{2\nu}]$ for some $c\in(0,1/2]$, and suppose that a sequence $\x_t$ is updated as in lines~\ref{line:blocks}-\ref{line:update} of Algorithm~\ref{alg:main}, allowing the regularized projection step (line~\ref{line-w}) to be computed inexactly, with $\tilde \w_t$ so that $\|\tilde\w_t-\w_t\|\leq \frac{\rho^2}{8\eta} \|\x_t-\x^*\|$. Then,
\begin{align*}
\E\Big[\left\|\x_{t} - \x^*\right\|^2\Big] \leq 8\big(1 - \rho/2\big)^t \cdot \|\x_0-\x^*\|^2.
\end{align*}
\end{theorem}

\begin{remark}[Stability]
Assuming exact projection steps, it is possible to get convergence rate $8(1-\sqrt{\mu/\nu})^t$  by setting $\rho = \bar\rho$ and $\eta = \frac{1/\nu-\bar\rho}{1-\bar\rho}$, replicating the convergence rate achieved by an acceleration scheme of \cite{gower2018accelerated}. 

However, a key feature of our algorithm captured by Theorem~\ref{lem:converge_accelerate} is that our proposed acceleration (momentum) scheme for block Kaczmarz is remarkably stable with respect to both the choice of parameters $\rho$ and $\eta$, as well as the precision of solving the regularized projection step.  Specifically, as one cannot hope to find the parameters at runtime, our result shows that it suffices to use an over-estimate of $\nu$ and an under-estimate of $\bar\rho$, where the estimation accuracy is captured by a factor $c$.

We highlight this stability as a key feature of our scheme, since it motivates our adaptive acceleration tuning for parameter $\rho$, described later in Section~\ref{s:practical}. Further, in Section~\ref{s:rate}, we show that after preprocessing with the RHT, $\nu$ can be bounded by $\tilde O(\frac ms)$, where $s$ is the block size, which suggests a simple problem-agnostic estimate for $\eta$ as well. Finally, we also show similar stability guarantees with respect to the regularization parameter $\lambda$ in Section \ref{s:rate}.
\end{remark}

Our next goal is to prove Theorem~\ref{lem:converge_accelerate}. We obtain this result through a careful reformulation of the \alg\ update, replacing the momentum vector with two auxiliary iterate sequences, which allows us to lean on existing Lyapunov-style convergence analysis for accelerated methods \cite{gower2018accelerated}. 
Instead of maintaining the momentum vector $\m_t$, we initialize $\v_0 = \y_0$ and maintain iterates $\x_t, \y_t, \v_t$ based on the following update rules:
\begin{align}\label{alg:main_analyze}
\begin{cases}
\x_t = \alpha\v_t + (1-\alpha)\y_t \\
\text{Compute } \tilde\w_t\approx\w_t  \\
\y_{t+1} = \x_t-\tilde\w_t \\
\v_{t+1} = \beta\v_t+(1-\beta)\x_t - \gamma\tilde\w_t
\end{cases}
\end{align}
These iterates are essentially in the form considered earlier in \cite{gower2018accelerated, derezinski2024fine}, and we can carry out the analysis of Nesterov's acceleration similarly, with the key difference that we use regularized projections in the Kaczmarz update, whereas prior works use exact projections. This results in the following estimate for the convergence rate in a convenient metric $\Delta_t$ as defined below and under a particular parameter choice.
\begin{lemma}[Based on {\cite[Lem.~2, Thm.~3] {gower2018accelerated}}]\label{lem:converge_accelerate_form2}
In the setting of Theorem~\ref{lem:converge_accelerate}, observe that $\mu := \lambda_{\min}^+(\bar\P_\lambda)$ and $\nu := \lambda_{\max}(\E[(\bar\P_\lambda^{\dagger/2}\P_{\lambda,S}\bar\P_\lambda^{\dagger/2})^2])$ satisfy $1\leq \nu\leq 1/\mu$. Moreover, suppose that $\tilde\mu\leq \mu$ and $\tilde\nu\geq \nu$, and let $\x_t,\y_t, \v_t$ be defined as in \eqref{alg:main_analyze} with  $\beta = 1 - \sqrt{\tilde\mu / \tilde\nu}$, $\gamma = 1 / \sqrt{\tilde\mu\tilde\nu}$ and $\alpha = 1 / (1+\gamma\tilde\nu)$. Also, let
\begin{align*}\Delta_t = \|\v_t-\x^*\|_{\bar\P_{\reg}^\dagger}^2 + \frac{1}{\tilde\mu}\|\y_t-\x^*\|^2.
\end{align*}
If $\|\tilde\w_t-\w_t\| \leq \frac{\tilde{\mu}}4\|\x_t - \x^*\|$, then we have
\begin{align*}
\E\left[\Delta_{t+1} \right] \leq \left(1 - \frac{1}{2}\sqrt{\frac{\tilde\mu}{\tilde\nu}}\right) \cdot \E\left[\Delta_t \right].
\end{align*}
\end{lemma}
The proof of Lemma~\ref{lem:converge_accelerate_form2} is generally similar to the proofs of \cite[Theorem 3]{gower2018accelerated} and \cite[Lemma 23]{derezinski2024fine} and it is deferred to Appendix~\ref{s:inexact}. 
Next, we prove the equivalence of the \alg\ update and \eqref{alg:main_analyze}.
\begin{lemma}[Equivalence of two algorithm formulations]\label{lem:equivalence}
In the notations of Theorem~\ref{lem:converge_accelerate}, let $\rho \in [0, c \bar\rho]$ and $\eta \in [\frac c\nu,\frac1{2\nu}]$  for some $c\in(0,1/2]$. Setting the parameters $\alpha, \beta, \gamma$ as in the statement of Lemma~\ref{lem:converge_accelerate_form2} above with
\begin{align*}
\tilde\nu  = \frac{1}{\rho + \eta(1 -\rho)}\quad\text{and}\quad \tilde\mu
=\rho^2\tilde\nu,
\end{align*}
we will get that (a) $\tilde\mu\leq \mu$ and $\tilde\nu\geq \nu$; and (b) the iterates $\x_t$ obtained from the updates \eqref{alg:main_analyze}, initialized with $\v_0=\y_0=\x_0$ and letting $\m_0 = \mathbf{0}$, satisfy 
\begin{align*}
\m_{t+1} = \frac{1-\rho}{1+\rho}\big(\m_t - \tilde\w_t\big)
\quad\text{and}\quad \x_{t+1} = \x_t - \tilde\w_t +\eta\m_{t+1}.
\end{align*}
That is, the iteration \eqref{alg:main_analyze} will satisfy the assumptions of Lemma~\ref{lem:converge_accelerate_form2} and it will be equivalent to the  lines~\ref{line:blocks}-\ref{line:update} of Algorithm~\ref{alg:main}.
\end{lemma}

The proof of this lemma is a direct verification, in particular, checking that $\m_t := \frac{\beta(1-\alpha)}{\gamma - 1}(\v_t-\y_t)$ satisfies the above recursion. It is deferred to Appendix~\ref{s:appendix_proof_equivalence}.

\begin{proof}[Proof of Theorem~\ref{lem:converge_accelerate}]
By Lemma~\ref{lem:equivalence}, the update from Algorithm~\ref{alg:main} is equivalent to the process according to the updates \eqref{alg:main_analyze}, with $\alpha = 1 / (1+\sqrt{\tilde\nu/\tilde{\mu}}) = \frac{\rho}{1+\rho}$. From Lemma~\ref{lem:converge_accelerate_form2}, we have the following two convergence results in terms of iterates $\y_t$ and $\v_t$ respectively:
\begin{align}\label{eq:converge_xt_vt}
\begin{cases}
\E\|\y_t-\x^*\|^2\le \tilde\mu\E[\Delta_t]\leq (1-\rho/2)^t\tilde\mu\Delta_0, \\
\E\|\v_t-\x^*\|^2\leq \E[\Delta_t]\leq (1-\rho/2)^t\Delta_0,
\end{cases}
\end{align}
where in the second inequality we use  that $\bar\P_{\reg}^\dagger \succeq \I$, thus $\|\v_t - \x^*\| \leq \|\v_t - \x^*\|_{\bar\P_{\reg}^\dagger}$. 

Our goal is to reformulate the convergence result in terms of the sequence $\x_t$. Since $\x_t = \alpha \v_t + (1-\alpha)\y_t$, we have the following:
\begin{align*}
\E\|\x_t-\x^*\|^2 
= & ~ \E\|\alpha(\v_t-\x^*) + (1-\alpha)(\y_t-\x^*)\|^2 \\
\leq & ~ 2\alpha^2\E\|\v_t-\x^*\|^2 + 2(1-\alpha)^2\E\|\y_t-\x^*\|^2 \\
\leq & ~ 2(1-\rho/2)^t\Big(\alpha^2\Delta_0 + (1-\alpha)^2\tilde\mu\Delta_0\Big)\\
\leq & ~ 2(1-\rho/2)^t(\alpha^2/\tilde\mu + 1)\Big(\tilde\mu\|\y_0-\x^*\|_{\bar\P_{\reg}^\dagger}^2+\|\y_0-\x^*\|^2\Big) \\
\leq & ~ 4(1+1/\tilde\nu)(1-\rho/2)^t\|\y_0-\x^*\|^2,
\end{align*}
where the third step follows from \eqref{eq:converge_xt_vt}, the fourth step follows since $\v_0 = \y_0$, the last step follows since $\|\bar\P_{\reg}^\dagger\| \leq 1 / \tilde\mu$ and $\alpha = \rho / (1+\rho) < \rho = \sqrt{\tilde\mu / \tilde\nu}$. Finally, since $\tilde\nu \ge \nu \geq 1$ (also by Lemma~\ref{lem:equivalence}) and $\x_0 = \alpha \v_0 + (1-\alpha)\y_0 = \y_0$, we conclude that
\begin{align*}
\E\|\x_t-\x^*\|^2 \leq 4(1+1/\tilde\nu)(1-\rho/2)^t\|\y_0-\x^*\|^2 \leq 8 \left(1 - \rho/2\right)^t \cdot \|\x_0 - \x^*\|^2.
\end{align*}
\end{proof}

\section{Sharp Convergence Rate via Regularized Projections}\label{s:rate}

In the previous section, we showed that, with an appropriate choice of the parameters $\eta$ and $\rho$, the convergence rate of Algorithm \ref{alg:main} depends on the theoretical quantity $\bar\rho = \sqrt{\mu/\nu}$, where $\mu$ and $\nu$ are notions of expectation and variance for the random regularized projection $\P_{\reg,S}$. In this section, we show that after preprocessing with the randomized Hadamard transform, it is possible to give a sharp characterization of both of these quantities in terms of the regularization amount $\reg$ and the spectrum of the input matrix $\A$. 
First, in Section \ref{s:mu}, we lower bound the expectation of the regularized projection (with respect to positive semidefinite ordering), which allows us to lower bound the parameter $\mu$.  Then we bound the variance of this projection and thus the term $\nu$ in Section \ref{s:nu}. The use of regularization is crucial for bounding $\nu$, and also for the analysis of block memoization in Section~\ref{s:blocks}.

\subsection{Expectation of the Regularized Projection}
\label{s:mu}
First, we lower bound the parameter $\mu =\lambda_{\min}^+(\E[\P_{\reg,S}])$
for some $\reg \geq 0$. 
Even better, we give a more general result, lower-bounding the entire expectation of the matrix in positive semidefinite ordering, which will be necessary later for the analysis of the variance term $\nu$ in Section~\ref{s:nu}.  

The following result shows that after applying the randomized Hadamard transform, the expectation of the regularized projection matrix $\P_{\reg,S}$ based on a  random sample of the rows of $\A$ is lower-bounded by an analogously defined regularized projection of the full matrix $\A$, with appropriately adjusted regularizer (denoted as $\bar\lambda$). This result can be viewed as a natural extension of some recent prior works \cite{derezinski2023solving,derezinski2024fine}, which showed such guarantees for classical random projections (i.e., not regularized). While introducing regularization naturally shrinks the random matrix $\P_{\reg,S}$, and thus it must also decrease the lower bound, we show that there is a level of regularization below which the overall expectation bound does not get significantly affected. Thus, we can reap the benefits of regularization (e.g., when bounding $\nu$ later on) without sacrificing any of the effectiveness of the projection.

\begin{theorem}\label{l:mu-reg}
Suppose $\A\in\R^{m \times n}$ 
is transformed by RHT, i.e., $\bar\A = \Q\A$. Let $\sigma_1\geq\sigma_2\geq ...$ be $\A$'s singular values. Given $\delta\in(0,1)$ and $C\log(m/\delta) \leq k < \rank(\A)$, let $\bar{\lambda} = \frac{1}{k}\sum_{i>k} \sigma_i^2$. Let $\setS\sim \Uc(m,s)$ be a uniformly random subset of $[m]$ with size $s \geq Ck \log (m\bar\kappa_k)$. Then, for any $0 \leq \reg \leq \frac{k}{m}\bar{\lambda}$, with probability $1-\delta$ the transformed matrix $\bar\A$ satisfies:
\begin{align*}
\E_{S\sim\Uc(m,s)}\Big[\bar\A_{\setS}^\top\big(\bar\A_{\setS}\bar\A_{\setS}^\top + \reg\I\big)^{\dagger}\bar\A_{\setS}\Big] ~\succeq~ \frac12\A^\top(\A \A^\top + \bar{\lambda}\I)^{-1}\A.
\end{align*}
\end{theorem}
\begin{remark}
Theorem \ref{l:mu-reg} implies that after RHT preprocessing, the expected regularized projection has the same null space as $\A$, and moreover, for any $\lambda\in[0,\frac km\bar\lambda]$:
\begin{align*}
\mu\Big(\bar\A,\Uc(m,s),\lambda\Big) \geq \frac{\sigma_{\min}^+(\A)^2/2}{\sigma_{\min}^+(\A)^2+\bar\lambda}
= \frac{1/2}{1+\frac{r-k}k\bar\kappa_k^2}
\geq \frac{k}{2r\bar{\kappa}_k^2},
\end{align*}
where $r$ is the rank of $\A$, while $\bar{\kappa}_k$ and $\mu$ are defined in \eqref{eq:kappa_k_def} and \eqref{eq:rho-mu-nu}, respectively. 
\end{remark}
\begin{remark}
    The main part of our analysis that requires RHT is Lemma \ref{l:dpp-reduction} below, from \cite{derezinski2024fine}. RHT preprocessing ensures that the input matrix $\A$ satisfies a deterministic incoherence condition which makes it amenable to uniform sub-sampling. For example, a sufficient (but not necessary) incoherence condition is that the matrix $\U$ of left singular vectors of $\A$ has all entries bounded as follows: $\max_{i,j} |u_{i,j}|=  O(1/\sqrt m)$. If this property holds for the input matrix $\A$, then Theorem \ref{l:mu-reg} applies without RHT.
\end{remark}
A similar guarantee to this one was given by \cite{derezinski2023solving}, which was later refined by \cite{derezinski2024fine}, however both of these prior results apply only to the case where $\lambda=0$. Remarkably, the right-hand side in the inequality above is identical to the corresponding Lemma 10 of \cite{derezinski2024fine}, which intuitively implies that introducing some regularization into the random projection step of block Kaczmarz does not substantially alter its expectation.

To prove the above result we build on a technique that has been developed in the aforementioned prior works. The strategy is to first show the lower bound for a non-uniform subset distribution called a determinantal point process, where one can leverage additional properties to compute the expectation of a random projection. Then, one can show that a sufficiently large uniform sample contains a sample from the determinantal point process, which implies the desired lower bound.
\begin{definition}\label{d:dpp}
    Given a PSD matrix $\Lb\in\mathcal{S}_m^{+}$, a determinantal point process $\setS\sim\DPP(\Lb)$ is a distribution over all sets $S\subseteq[m]$ such that $\Pr(\setS)\propto \det(\Lb_{\setS,\setS})$.
\end{definition}
\begin{lemma}[\cite{derezinski2021determinantal}]\label{l:dpp-size}
    The expected size of $S\sim \DPP(\Lb)$ is $\E[|S|] = \tr(\Lb(\Lb+\I)^{-1})$.
\end{lemma}
In our proof we will use the following black-box reduction from uniform sampling to a DPP, which first appeared in the proof of Lemma 4.3 in \cite{derezinski2023solving}. The version below is based on the proof of Lemma 10 in \cite{derezinski2024fine}. 
\begin{lemma}[\cite{derezinski2024fine}]\label{l:dpp-reduction}
    Consider a PSD matrix $\Lb\in\mathcal{S}_m^{+}$, an $m\times m$ RHT matrix~$\Q$, and $\delta>0$ such that set  $\setS_{\DPP}\sim\DPP(\Q\Lb\Q^\top)$ satisfies $k:=\E[|\setS_{\DPP}|]\geq C\log(m/\delta)$. Then, conditioned on an RHT property that holds with probability $1-\delta$, a uniformly random set $\setS\sim\Uc(m,s)$ of size $s\geq Ck\log(k/\delta')$ can be coupled with $\setS_{\DPP}$ into a joint random variable $(S,S_{\DPP})$ such that $\setS_{\DPP}\subseteq\setS$  with probability $1-\delta'$.
\end{lemma}

To conclude the expected projection result from the above black-box reduction, \cite{derezinski2024fine} used a classical Cauchy-Binet-type determinantal summation formula (e.g., see Lemma 5 in \cite{derezinski2020improved}), which shows that a DPP-sampled set $\setS_{\DPP}\sim\DPP(\frac1{\bar\lambda}\A\A^\top)$ satisfies $\E[\P_{0,S_\DPP}] = \A^\top(\A\A^\top+\bar\lambda\I)^{-1}\A$, where $\P_{0,S_\DPP}=\A_{S_\DPP}^\top(\A_{S_\DPP}\A_{S_\DPP}^\top)^\dagger\A_{S_\DPP}$ is the standard projection arising in block Kaczmarz. Then, one can convert from $S_{\DPP}$ to $S$ by observing via Lemma \ref{l:dpp-reduction} that $\P_{0,S}\succeq \P_{0,S_{\DPP}}$ with probability $1-\delta'$.

However, since we are bounding the expectation of a \emph{regularized} projection matrix $\P_{\reg,\setS}$, the classical Cauchy-Binet-type formula cannot be directly applied, and we no longer have a simple closed form expression for the expected regularized projection under DPP sampling. To address this, we show in Lemma~\ref{l:mu-reg-exact} that if we sample according to a different DPP defined by matrix $\frac{m}{\bar{\lambda}(m-k)}\bar{\A}\bar{\A}^\top + \frac{k}{m-k}\I$, then we can sufficiently bound the corresponding regularized projection by using the concept of ``Regularized DPPs" originating from \cite{derezinski2020bayesian} (for details see Appendix~\ref{s:appendix_proof_mu}). By combining the above discussion, we formally give the proof of Theorem~\ref{l:mu-reg}. 

\begin{proof}[Proof of Theorem~\ref{l:mu-reg}]
Let $\A = \U\Sig\V^\top$ be its singular value decomposition and denote $\bar{\U} = \Q\U$. Define matrix $\Lb \coloneqq \frac{m}{\bar{\lambda}(m-k)} \A\A^\top + \frac{k}{m-k}\I$, and notice that\footnote{For convenience, in the case of $m > n$, we simply define $\sigma_{n+1} = \cdots = \sigma_m = 0$.}
\begin{align*}
\Q\Lb\Q^\top = & ~
\frac{m}{\bar{\lambda}(m-k)} \bar{\A}\bar{\A}^\top + \frac{k}{m-k}\I \\
= & ~ \bar{\U}\,\diag\Big(\frac{m\sigma_1^2}{\bar{\lambda}(m-k)} + \frac{k}{m-k},\ldots, \frac{m\sigma_m^2}{\bar{\lambda}(m-k)} + \frac{k}{m-k}\Big) \bar{\U}^\top
\end{align*}
is the eigendecomposition of matrix $\Q\Lb\Q^\top$. By setting $\bar{\lambda} = \frac{1}{k}\sum_{i>k}\sigma_i^2$ and denoting $\setS_{\DPP} \sim \DPP(\frac{m}{\bar{\lambda}(m-k)}\bar{\A}\bar{\A}^\top + \frac{k}{m-k}\I)$, we can bound the expected sample size of $\setS_{\DPP}$ using Lemma \ref{l:dpp-size} as follows: 
\begin{align*}
\E[|\setS_{\DPP}|] = \sum_{i=1}^m \frac{\frac{m}{\bar{\lambda}(m-k)}\sigma_i^2 + \frac{k}{m-k}}{\frac{m}{\bar{\lambda}(m-k)}\sigma_i^2 + \frac{k}{m-k} + 1} = \sum_{i=1}^m \frac{m\sigma_i^2 + \bar{\lambda} k}{m\sigma_i^2 + \bar{\lambda} m} = k + \sum_{i=1}^m \frac{(m-k)\sigma_i^2}{m\sigma_i^2 + \bar{\lambda}m} \geq k,
\end{align*}
and a similar calculation shows $\E[|S_{\DPP}|]\leq 3k$.
Given $\delta,\delta' >0$ and $k \geq C \log (m/\delta)$, let $\setS\sim\Uc(m,s)$ be a uniformly random set with $s \geq 3C k \log (3k/\delta')$. By applying Lemma~\ref{l:dpp-reduction} to matrix $\Lb$, conditioned on an RHT property that holds with probability $1-\delta$, we have $\setS_{\DPP} \subseteq \setS$ holds with probability $1-\delta'$.
With the above analysis, we move on to the expectation of the ``regularized'' projection matrix. Using that $0 \leq \reg \leq \frac{k}{m}\bar{\lambda}$, the following holds:
\begin{align}\label{eq:p_lambda}
\P_{\reg,\setS} = & ~ \A_{\setS}^\top(\A_{\setS}\A_{\setS}^\top + \reg\I)^{\dagger}\A_{\setS}\nonumber \\
\succeq & ~ \A_{\setS}^\top \left(\A_{\setS}\A_{\setS}^\top + \frac{k\bar{\lambda}}{m}\I \right)^{-1}\!\!\A_{\setS} = \I - \frac{k\bar{\lambda}}{m} \cdot \left(\A_{\setS}^\top\A_{\setS} + \frac{k\bar{\lambda}}{m}\I \right)^{-1}.
\end{align}
To bound the right hand side of \eqref{eq:p_lambda} we use the following lemma, which bounds the corresponding term when sampling according to this specific DPP. The proof of Lemma~\ref{l:mu-reg-exact} is based on the concept of Regularized DPP proposed by \cite{derezinski2020bayesian}, and we defer it to Appendix~\ref{s:appendix_proof_mu}.

\begin{lemma}\label{l:mu-reg-exact}
Given $\A\in\R^{m\times n}$ and $ k < \rank(\A)$, let  $\bar{\lambda} = \frac{1}{k}\sum_{i>k}\sigma_i^2(\A)$. Then, the random set $\setS_{\DPP} \sim\DPP(\frac{m}{\bar{\lambda}(m-k)}\A\A^\top + \frac{k}{m-k}\I)$ satisfies
\begin{align}\label{eq:mu-reg-exact}
\E\left[\left(\I + \frac{m}{k \bar{\lambda}}\A_{\setS_{\DPP}}^\top\A_{\setS_{\DPP}} \right)^{-1}\right] \preceq \bar{\lambda} \left(\A^\top\A + \bar{\lambda}\I\right)^{-1}.
\end{align}
\end{lemma}
Conditioned on the event $\mathcal{E} := [\setS_{\DPP} \subseteq \setS]$ (which holds with probability $1-\delta'$), we have $\A_{\setS_{\DPP}}^\top \A_{\setS_{\DPP}} \preceq \A_{\setS}^\top \A_{\setS}$.
Combining this with \eqref{eq:p_lambda} and Lemma~\ref{l:mu-reg-exact}, we have the following holds:
\begin{align*}
\E[\P_{\reg, \setS}] = & ~ \E[\P_{\reg, \setS} \mid \mathcal{E}]\Pr\{\mathcal{E}\} + \E[\P_{\reg, \setS} \mid \neg \mathcal{E}]\Pr\{\neg \mathcal{E}\} \\
\succeq & ~ \E[\P_{\reg, \setS} \mid \mathcal{E}]\Pr\{\mathcal{E}\} \\
\succeq & ~ \Pr\{\mathcal{E}\} \cdot\I - \frac{k\bar{\lambda}}{m} \cdot \E\left[\left(\A_{\setS}^\top \A_{\setS} + \frac{k\bar{\lambda}}{m}\I \right)^{-1} \biggl\vert \mathcal{E} \right] \cdot \Pr\{\mathcal{E}\} \\
\succeq & ~ \Pr\{\mathcal{E}\} \cdot \I - \E\left[\left(\frac{m}{k\bar{\lambda}}\A_{\setS_{\DPP}}^\top\A_{\setS_{\DPP}} + \I \right)^{-1} \biggl\vert \mathcal{E} \right] \cdot \Pr\{\mathcal{E}\} \\
\succeq & ~ (1-\delta')\cdot \I - \bar{\lambda}(\A^\top\A + \bar{\lambda}\I)^{-1} 
=  ~ \A^\top(\A\A^\top + \bar{\lambda}\I)^{-1} \A - \delta'\I.
\end{align*}
Notice that the spectrum of matrix $\A^\top(\A\A^\top + \bar{\lambda}\I)^{-1} \A$ can be expressed as $\{\frac{\sigma_i^2}{\sigma_i^2 + \bar{\lambda}}\}_i$, thus with the choice of $\bar{\lambda} = \frac{1}{k}\sum_{i>k}\sigma_i^2$ we have
\begin{align*}
\A^\top(\A\A^\top + \bar{\lambda}\I)^{-1} \A \succeq
\frac{(\sigma_{\min}^+)^2}{(\sigma_{\min}^+)^2 + \bar{\lambda}} \I = \frac{1}{1 + \frac{r-k}{k}\bar{\kappa}_k^2} \I \succeq \frac{k}{r \bar{\kappa}_k^2} \I \succeq \frac{k}{m \bar{\kappa}_k^2} \I.
\end{align*}
By choosing $\delta' = \frac{k}{2r\bar{\kappa}_k^2}$ we have $\delta'\I \preceq \frac{1}{2}\A^\top(\A\A^\top + \bar{\lambda}\I)^{-1} \A$, which gives $\E[\P_{\reg, \setS}] \succeq  \frac{1}{2} \A^\top(\A\A^\top + \bar{\lambda}\I)^{-1} \A$,
and the sample size needs to satisfy $s \geq O(k\log (k/\delta')) =  O(k \log (r\bar{\kappa}_k))$. We also conclude that $\mu(\bar\A,\Uc(m,s),\lambda) = \lambda_{\min}^+(\E[\P_{\reg,S}]) \geq \frac{k}{2r \bar{\kappa}_k^2}$. 
\end{proof}

\subsection{Variance of the Regularized Projection}
\label{s:nu}
We now turn to bounding the term 
$\nu = \lambda_{\max}(\E[(\bar\P_{\reg}^{\dagger/2}\P_{\reg,S}\bar\P_{\reg}^{\dagger/2})^2])$, where $\bar\P_{\reg}=\E[\P_{\reg,S}]$, which intuitively describes a notion of variance for the regularized projection $\P_{\reg,S}$. This quantity was first introduced by \cite{gower2018accelerated} in the case of $\reg = 0$. They showed that $\frac rs \leq \nu\leq \frac1\mu$ for any matrix $\A$ of rank $r$ and random blocks $S$ of size $s$, which unfortunately does not provide any acceleration guarantee. Recently, \cite{derezinski2024fine} gave an improved upper bound, but it came with trade-offs: Their bound, $\nu= \tilde O(\frac rs\bar\kappa_{k:O(k\log k)}^2)$, where $\bar\kappa_{k,l}^2:=\frac1{l-k}\sum_{i=k+1}^l\sigma_i^2(\A)$, requires replacing block sampling with a much more expensive sketching approach, due to their reliance on sophisticated tools from random matrix theory, and yet, it is still affected by a problem-dependent condition number $\bar\kappa_{k,l}$.

We use regularized projections to entirely avoid these trade-offs: Not only are we able to use block sampling (as opposed to expensive sketching), but also our proof is surprisingly elementary, and with the right choice of $\reg$, we get a bound of $\nu = \tilde O(\frac rs)$, without any problem-dependent condition number factors. 

\begin{theorem}\label{l:nu-reg}
Given matrix $\A\in\R^{m\times n}$, 
parameters $\bar{\lambda} \geq \reg > 0$, and a probability distribution $\Dc$ over subsets of $[m]$, suppose that the corresponding regularized projection matrix $\P_{\reg,S} \coloneqq \A_{\setS}^\top(\A_{\setS}\A_{\setS}^\top + \reg\I)^{-1}\A_{\setS}$ satisfies:
\begin{align*}
    \bar\P_\lambda:=\E_{S\sim\Dc}[\P_{\lambda,S}] \succeq c\A^\top\A(\A^\top\A+\bar\lambda\I)^{-1},
\end{align*}
for some $c\in (0,1]$. Then, it follows that:
\begin{align*}
\lambda_{\max}\Big(\E\big[(\bar\P_{\reg}^{\dagger/2}\P_{\reg,S}\bar\P_{\reg}^{\dagger/2})^2\big]\Big)\leq \frac{2\bar\lambda}{c\reg}.
\end{align*}
If we further assume that $\A_S^\top\A_S\preceq\alpha\A^\top\A$ with probability $1-\delta$ for some $\alpha\in[0,1]$ and $\delta \in [0, \alpha / \|\bar\P_\lambda^\dagger\|]$, then we can obtain the following potentially sharper bound:
\begin{align*}
\lambda_{\max}\Big(\E\big[(\bar\P_{\reg}^{\dagger/2}\P_{\reg,S}\bar\P_{\reg}^{\dagger/2})^2\big]\Big)
\leq \frac{2}{c}\bigg(1 + \alpha\frac{\bar\lambda}{\lambda}\bigg).
\end{align*}
\end{theorem}
Note that the second part of Theorem~\ref{l:nu-reg} is (up to a constant) a generalization of the first part: when we set $\alpha = 1$, then naturally we have that $\A_S^\top \A_S \preceq \A^\top\A$ holds with probability~$1$ (that is, $\delta = 0$), and the bound on the right hand side becomes $\frac{2}{c}(1+\frac{\bar\lambda}{\lambda}) \leq \frac{4\bar\lambda}{c\lambda}$. However, a simple matrix concentration argument, given in the following lemma, shows that in the over-determined case where $m \gg r=\rank(\A)$, we can give a sharper bound of $\alpha= O(\frac{r}{m} \log(n/\delta))$.
\begin{lemma}\label{l:nu_misc}
Suppose matrix $\A\in\R^{m \times n}$ with rank $r$ is transformed by RHT. Let $S \sim \Uc(m,s)$ be a uniformly random subset of $[m]$ with size $s \leq r$. Conditioned on an event that happens with probability $1-\delta$ and only depends on RHT, with probability $1-\delta'$ we have the following bound:
\begin{align*}
\A_S^\top\A_S \preceq \frac{(4r + 32\log(m/\delta)) \cdot \log(n/\delta')}{m} \cdot \A^\top\A.
\end{align*}
\end{lemma}
By combining Theorem~\ref{l:nu-reg} and Lemma~\ref{l:nu_misc}, we have the following corollary.
\begin{corollary}
\label{cor:nu}
Suppose matrix $\A\in \R^{m \times n}$ with $\rank(\A) = r$ is transformed by RHT. Given $\delta\in(0,1)$ and $C\log(m/\delta) \leq k < r$, let $\bar{\lambda} = \frac{1}{k}\sum_{i>k} \sigma_i^2$. Let $S \sim \Uc(m,s)$ be a uniformly random subset of $[m]$ with size $s\in [Ck \log (m\bar\kappa_k), r]$. Then for any $0< \lambda \leq \frac{k}{m} \bar{\lambda}$, conditioned on an event that happens with probability $1-\delta$ and only depends on RHT, we have 
\begin{align*}
\lambda_{\max}\Big(\E\big[(\bar\P_{\reg}^{\dagger/2}\P_{\reg,S}\bar\P_{\reg}^{\dagger/2})^2\big]\Big)
\leq \frac{4\bar\lambda}{\lambda} \cdot \min\left\{1, \frac{4r}{m} \log(mn \bar{\kappa}_k) \right\}.
\end{align*}
By further choosing $\lambda = \frac{k}{m}\bar{\lambda}$, we have
\begin{align*}
\nu(\bar\A,\Uc(m,s),\lambda)\leq \min\left\{\frac{4m}{k}, \frac{16r}{k} \log(mn\bar{\kappa}_k) \right\}.
\end{align*}
\end{corollary}
We are now ready to present the proof of Theorem \ref{l:nu-reg}.

\begin{proof}[Proof of Theorem~\ref{l:nu-reg}]
By using the assumption on $\bar{\P}_{\reg} = \E[\P_{\reg, S}]$, we can bound the pseudoinverse of this matrix as follows:
\begin{align}\label{eq:nu-tech}
\bar{\P}_{\reg}^{\dagger} \preceq \frac{1}{c}\left(\A^\top\A + \bar{\lambda}\I \right) \left(\A^\top\A \right)^{\dagger} \preceq \frac1c\Big(\I + \bar{\lambda}\left(\A^\top\A\right)^{\dagger}\Big),
\end{align}
which gives
\begin{align}\label{eq:nu_bound}
\nu = & ~ \left\|\E[\bar{\P}_{\reg}^{\dagger/2} \P_{\reg,S} \bar{\P}_{\reg}^{\dagger} \P_{\reg,S}\bar{\P}_{\reg}^{\dagger/2}] \right\| = \left\|\bar{\P}_{\reg}^{\dagger/2}\E[ \P_{\reg,S} \bar{\P}_{\reg}^{\dagger} \P_{\reg,S}]\bar{\P}_{\reg}^{\dagger/2} \right\| \nonumber \\
\overset{\eqref{eq:nu-tech}}{\leq} & ~ \frac{1}{c} \left\|\bar{\P}_{\reg}^{\dagger/2}\E[\P_{\reg,S}^2 +  \bar{\lambda}\P_{\reg,S}(\A^\top\A)^{\dagger} \P_{\reg,S}]\bar{\P}_{\reg}^{\dagger/2} \right\| \nonumber \\
\leq & ~ \frac{1}{c} \left\|\bar{\P}_{\reg}^{\dagger/2}\left(\E[\P_{\reg,S}] + \bar{\lambda} \E[\P_{\reg,S}(\A^\top\A)^{\dagger} \P_{\reg,S}] \right)\bar{\P}_{\reg}^{\dagger/2} \right\| \nonumber \\
\leq & ~ \frac{1}{c} + \frac{\bar{\lambda}}{c} \left\|\bar{\P}_{\reg}^{\dagger/2}\E[\P_{\reg,S}(\A^\top\A)^{\dagger} \P_{\reg,S}] \bar{\P}_{\reg}^{\dagger/2} \right\|.
\end{align}
By expanding $\P_{\reg,S}$, we can express the middle term in \eqref{eq:nu_bound} as follows:
\begin{align*}
\P_{\reg,S} (\A^\top\A)^\dagger \P_{\reg,S} = & ~ \A_{\setS}^\top ( \A_{\setS} \A_{\setS}^\top + \reg\I)^{-1} \cdot\A_{\setS}(\A^\top\A)^\dagger \A_{\setS}^\top\cdot (\A_{\setS} \A_{\setS}^\top + \reg\I)^{-1} \A_{\setS}.
\end{align*}
Denote $\mathcal{E}$ as the event that $\A_S^\top\A_S \preceq \alpha \A^\top\A$, by assumption we have $\Pr\{\mathcal{E}\} = 1-\delta$. Conditioned on $\mathcal{E}$, we have $\|\A_S(\A^\top\A)^\dagger \A_S^\top\| = \|(\A^\top\A)^{\dagger/2} \A_S^\top\A_S (\A^\top\A)^{\dagger/2}\| \leq \alpha$, which gives
\begin{align*}
& ~ \E[\P_{\reg,S} (\A^\top\A)^\dagger \P_{\reg,S}] \\
= & ~ (1-\delta) \cdot \E[\P_{\reg,S} (\A^\top\A)^\dagger \P_{\reg,S} \mid \mathcal{E}] + \delta \cdot \E[\P_{\reg,S} (\A^\top\A)^\dagger \P_{\reg,S} \mid \neg \mathcal{E}] \\
\preceq & ~ \alpha (1-\delta)
\cdot \E[\A_{\setS}^\top ( \A_{\setS} \A_{\setS}^\top + \reg\I)^{-2} \A_{\setS} \mid \mathcal{E}] + \delta \cdot \E[\A_{\setS}^\top ( \A_{\setS} \A_{\setS}^\top + \reg\I)^{-2} \A_{\setS} \mid \neg \mathcal{E}] \\
= & ~ \alpha \cdot \E[\A_{\setS}^\top ( \A_{\setS} \A_{\setS}^\top + \reg\I)^{-2} \A_{\setS}] + \left(1- \alpha \right)\delta\cdot\E[\A_{\setS}^\top ( \A_{\setS} \A_{\setS}^\top + \reg\I)^{-2} \A_{\setS} \mid \neg \mathcal{E}]\\
\preceq & ~ \frac{\alpha}{\lambda} \cdot \E[\P_{\reg,\setS}] + \frac{(1-\alpha)\delta}{\lambda}\cdot \I.
\end{align*}
Here, we use that $\E[\A_{\setS}^\top ( \A_{\setS} \A_{\setS}^\top + \reg\I)^{-2} \A_{\setS}] \preceq \frac{1}{\lambda} \E[\A_{\setS}^\top ( \A_{\setS} \A_{\setS}^\top + \reg\I)^{-1} \A_{\setS}] = \frac{1}{\lambda} \E[\P_{\lambda,S}]$ in the last step. If $\alpha=1$, then $\A_S^\top\A_S \preceq \A^\top\A$ always holds, which gives $\delta = 0$. Then, $\E[\P_{\reg,S} (\A^\top\A)^\dagger \P_{\reg,S}] \preceq \frac{1}{\lambda}\E[\P_{\lambda,S}]$, and by applying this result to \eqref{eq:nu_bound} we have
\begin{align*}
\nu \leq \frac{1}{c} + \frac{\bar{\lambda}}{c} \cdot \frac{1}{\lambda} \leq \frac{2\bar\lambda}{c\lambda}.
\end{align*}
If $0 \leq \alpha < 1$, then by applying $\E[\P_{\reg,S} (\A^\top\A)^\dagger \P_{\reg,S}] \preceq \frac{\alpha}{\lambda} \cdot \E[\P_{\reg,\setS}] + \frac{(1-\alpha)\delta}{\lambda}\cdot \I$ to \eqref{eq:nu_bound} we have
\begin{align*}
\nu \leq \frac{1}{c} + \frac{\bar\lambda}{c} \cdot \left\|\frac{\alpha}{\lambda} \I + \frac{(1-\alpha) \delta}{\lambda} \bar{\P}_{\lambda}^\dagger \right\| \leq \frac{1}{c} + \frac{\bar\lambda}{c\lambda} \left(\alpha + \delta \|\bar{\P}_{\lambda}^\dagger\| \right) = \frac1c\bigg(1 + \frac{\bar\lambda}{\lambda}\Big(\alpha+ \delta\|\bar\P_\lambda^\dagger\|\Big)\bigg).
\end{align*}
Using the assumption that $\delta \leq \alpha / \|\bar{\P}_{\lambda}^\dagger\|$ we conclude the proof.
\end{proof}

Finally, we conclude this section by showing how Corollary \ref{cor:nu} follows by combining Theorems \ref{l:mu-reg} and \ref{l:nu-reg} with Lemma \ref{l:nu_misc}.

\begin{proof}[Proof of Corollary~\ref{cor:nu}]
Under the assumptions, by Theorem~\ref{l:mu-reg} we have
\begin{align}\label{eq:assumption}
\bar{\P}_{\reg} = \E_{S\sim\Dc}[\P_{\lambda,S}] \succeq \frac{1}{2}\A^\top\A(\A^\top\A+\bar\lambda\I)^{-1}
\end{align}
holds conditioned on an event that only depends on RHT. This gives that $\mu \coloneqq \lambda_{\min}^+(\bar{\P}_{\reg}) \geq \frac{k}{2r\bar{\kappa}_k^2}$. Since we assume that $r> C\log(m/\delta)$, according to Lemma~\ref{l:nu_misc} we have $\A_S^\top\A_S \preceq \frac{c'r \log(n/\delta')}{m}\A^\top\A$ holds for $c' = 4 + \frac{32}{C}$ with probability $1-\delta'$. By applying this result and \eqref{eq:assumption} to Theorem~\ref{l:nu-reg} with choice $\delta' = \frac{r}{m\|\bar{\P}_{\lambda}^\dagger\|} = \frac{r}{m}\mu \geq \frac{k}{2m\bar{\kappa}_k^2}$, we have
\begin{align*}
& ~ \lambda_{\max}\Big(\E\big[(\bar\P_{\reg}^{\dagger/2}\P_{\reg,S}\bar\P_{\reg}^{\dagger/2})^2\big]\Big) \\
\leq & ~ 2 + \frac{2\bar\lambda}{\lambda}\Big(\frac{c'r \log(n/\delta')}{m} + \frac{r}{m}\Big) \leq 2+\frac{2\bar\lambda}{\lambda} \frac{(c'+1)r\log(n/\delta')}{m} \\
\leq & ~ \frac{2\bar\lambda}{\lambda} \left(\frac{k}{m} + \frac{(c'+1)r\log(mn\bar{\kappa}_k)}{m}\right) \leq \frac{2\bar\lambda}{\lambda}\cdot \frac{(c'+2)r \log(mn\bar{\kappa}_k)}{m} \\
\leq & ~ \frac{4\bar\lambda}{\lambda}\cdot \frac{(3+\frac{16}{C})r \log(mn\bar{\kappa}_k)}{m} \leq \frac{\bar\lambda}{\lambda}\cdot \frac{16r \log(mn\bar{\kappa}_k)}{m}
\end{align*}
where the last step follows by taking $C\geq 16$. Since $\lambda_{\max}(\E[(\bar\P_{\reg}^{\dagger/2}\P_{\reg,S}\bar\P_{\reg}^{\dagger/2})^2]) \leq \frac{4\bar\lambda}{\lambda}$ also holds, we finish the proof. By further specifying $\lambda = \frac{k}{m}\bar{\lambda}$, we conclude that
\begin{align*}
\nu(\bar\A,\Uc(m,s),\lambda)\leq \min\left\{\frac{4m}{k}, \frac{16r}{k} \log(mn\bar{\kappa}_k) \right\}.
\end{align*}
\end{proof}

\section{Optimized Computations via Block Memoization}
\label{s:blocks}

The overall computational cost of \alg\ consists of the cost of applying the RHT plus the cost of performing its iterations.
Next, in Section \ref{s:fast-projection}, we discuss the computational cost of computing the regularized projections, which dominate the overall computations in an iteration, but fortunately can be done inexactly. Then, in Section \ref{s:block-memoization}, we analyze our proposed block memoization, which reduces the number of Cholesky computations required for the regularized projections.  Finally, we put everything together in Section~\ref{s:overall_comp}, and summarize the overall computational costs in Theorem~\ref{thm:main}.

\subsection{Computing the Projection Step}\label{s:fast-projection}
The dominant computational cost in each of the iterations is computing the regularized projection step $\w_t$, which can be formulated as standard under-determined least squares with Tikhonov regularization:
\begin{align*}
    \w_t = \argmin_{\w\in\R^n}\Big\{\|\A_S\w-\r_t\|^2+\lambda\|\w\|^2\Big\},\quad\text{where}\quad \r_t = \A_S\x_t-\b_S.
\end{align*}
This step can be computed directly using $O(ns^2)$ arithmetic operations for a block of size $s$, which may be acceptable for small $s$, but becomes prohibitive for large block sizes.
However, since our convergence analysis allows computing $\w_t$ inexactly, we can also use a preconditioned iterative solver such as CG or LSQR. Here, we propose a randomized preconditioning strategy based on sketching, where one constructs a small sketch $\hat\A=\A_S\mPi^\top\in\R^{s\times \tau}$ for a sketching matrix $\mPi\in\R^{\tau\times n}$, and then use this sketch to construct a preconditioner. Given the extensive literature on randomized sketching (e.g., see \cite{drineas2016randnla,martinsson2020randomized,derezinski2024recent}), there are several different preconditioner constructions one can use, such as Blendenpik \cite{avron2010blendenpik} and LSRN \cite{meng2014lsrn}. Of particular relevance here are approaches that exploit the presence of regularization $\lambda$ to improve the quality of the preconditioner. Here, we will describe the Cholesky-based preconditioner of \cite{meier2022randomized}, due to its simplicity and numerical stability:
\begin{algorithmic}[1]
    \State Compute $\hat\A=\A_S\mPi^\top$, where $\mPi\in\R^{\tau\times n}$ is a random sketching matrix;
    \State Compute $\Rb = \mathrm{chol}(\hat\A\hat\A^\top+\lambda\I)$, where $\mathrm{chol}()$ is the Cholesky factorization.
\end{algorithmic}
Armed with this Cholesky preconditioner $\Rb$, we can now compute $\w_t$ as part of the min-length solution to the following system using an iterative method such as LSQR:
\begin{align}
\begin{bmatrix}\w_t\\\v_t\end{bmatrix}
    =\argmin_{\w\in\R^{n},\v\in\R^s}
    \Big\|\Rb^{-\top}\big[\A_S\ \sqrt\lambda\I\big]\begin{bmatrix}\w\\\v\end{bmatrix}-\Rb^{-\top}\r_t\Big\|^2.\label{eq:lsqr}
\end{align}
The quality of the preconditioning is determined primarily by the choice of sketch size $\tau$. In particular, to ensure that the system \eqref{eq:lsqr} has condition number $O(1)$, it suffices to use sketch size $\tau$ proportional to the so-called $\lambda$-effective dimension $d_\lambda(\A_S) = \sum_{i=1}^s \frac{\sigma_i^2(\A_S)}{\sigma_i^2(\A_S)+\lambda}$. Note that $d_\lambda(\A_S)\leq s$ for any $\lambda\geq 0$, and moreover, larger $\lambda$ yields smaller $d_\lambda(\A_S)$, which means that introducing regularization makes it easier to precondition the projection step. We illustrate this in the case when $\mPi$ is the Subsampled Randomized Hadamard Transform (SRHT, \cite{ailon2009fast,tropp2011improved}), although similar guarantees can be obtained, e.g., for sparse sketching matrices \cite{clarkson2013low,chenakkod2023optimal}.
\begin{lemma}\label{l:inner}
    For $\A_S\in\R^{s\times n}$ and $\lambda\geq 0$, if $\mPi=\sqrt{\frac n\tau}\I_T\Q$ where $\Q$ is the RHT and $T$ is a uniformly random set of size $\tau\geq C (d_\lambda(\A_S)+\log(n/\delta))\log(d_\lambda(\A_S)/\delta)$, then with probability $1-\delta$ we have $\kappa(\Rb^{-\top}\big[\A_S\ \sqrt\lambda\I\big]) \leq 2$, and after $O(\log (1/\epsilon))$ iterations of LSQR on \eqref{eq:lsqr}, we get $\tilde\w$ such that $\|\tilde\w-\w_t\|\leq\epsilon\|\x_t - \x^*\|$.
\end{lemma}
\begin{proof}
The bound $\kappa(\Rb^{-\top}\big[\A_S\ \sqrt\lambda\I\big]) \leq 2$ follows from standard analysis of sketching, e.g., see Theorem 3.5 in \cite{meier2022randomized} and the associated discussion. LSQR initialized with zeros after $O(\log 1/\epsilon)$ iterations returns vectors $\tilde\w$ and $\tilde\v$ such that:
\begin{align*}
    \bigg\|\begin{bmatrix}\tilde\w\\\tilde\v\end{bmatrix} - \begin{bmatrix}\w_t\\\v_t\end{bmatrix}\bigg\|\leq 
\epsilon\sqrt{\kappa(\M)}\cdot\bigg\|\begin{bmatrix}\w_t\\\v_t\end{bmatrix}\bigg\|,
    \quad\text{for}\quad 
    \M = \begin{bmatrix}\A_S^\top\\\sqrt\lambda\I\end{bmatrix}
    \Rb^{-1}\Rb^{-\top}\big[\A_S\ \sqrt\lambda\I\big].
\end{align*}
From the condition number bound we have that $\kappa(\M)\leq 4$, so we can now recover the error bound for $\tilde\w$ as follows:
\begin{align*}
\|\tilde\w-\w_t\|\leq \sqrt{\|\tilde\w-\w_t\|^2+\|\tilde\v - \v_t\|^2}\leq 2\epsilon\sqrt{\|\w_t\|^2+\|\v_t\|^2}.
\end{align*}
Since $\w_t = \A_S^\top(\A_S\A_S^\top+\lambda\I)^{-1}\r_t$, $\v_t = \sqrt\lambda(\A_S\A_S^\top+\lambda\I)^{-1}\r_t$, and $\r_t = \A_S (\x_t - \x^*)$,
we can bound $\|\w_t\|^2$ and $\|\v_t\|^2$ separately as
\begin{align*}
\|\w_t\|^2 &=  \|\A_S^\top(\A_S\A_S^\top+\lambda\I)^{-1} \A_S(\x_t - \x^*)\|^2 \leq \|\x_t - \x^*\|^2
\end{align*}
and
\begin{align*}
\|\v_t\|^2 &=  \lambda (\x_t - \x^*)^\top \A_S^\top (\A_S\A_S^\top+\lambda\I)^{-2}\A_S(\x_t - \x^*) \\
&\leq(\x_t - \x^*)^\top \A_S^\top (\A_S\A_S^\top+\lambda\I)^{-1}\A_S(\x_t - \x^*) \leq \|\x_t - \x^*\|^2.
\end{align*}
Thus $\|\tilde\w - \w_t\| \leq 2\sqrt{2}\epsilon \|\x_t - \x^*\|$. Adjusting $\epsilon$ appropriately concludes the proof.
\end{proof}
Constructing the preconditioner $\Rb$ takes $O(T_{\mathrm{sketch}} + \tau s^2 + s^3)$ operations, where $T_{\mathrm{sketch}}$ represents the cost of the matrix product $\A_S\mPi^\top$. For example, when $\mPi$ is the SRHT, as in Lemma \ref{l:inner}, then $T_{\mathrm{sketch}} = O(ns\log n)$.
Thus, setting $\tau = O(s\log s)$, the overall cost of solving the projection step to within $\epsilon$ relative accuracy takes no more than $O(ns\log n + s^3\log s)$ operations for constructing $\Rb$, followed by $O(ns\log(1/\epsilon))$ operations for running LSQR.

We note that more elaborate randomized preconditioning schemes exist for regularized least squares, such as the SVD-based preconditioner of \cite{meier2022randomized}, and the KRR-based preconditioners of \cite{avron2017faster,frangella2023randomized}, which can be computed with $O(T_{\mathrm{sketch}} + \tau^2s)$ operations. These approaches may be preferable when using $\tau\ll s$. However, given that the matrix $\A_S$ is itself random, it may be difficult to find the optimal value of $\tau$ in each step, which is why we recommend the simple choice of $\tau = \tilde O(s)$.

\subsection{Block Memoization}
\label{s:block-memoization}

When the block size $s$ is larger than $O(\sqrt n)$, then the $O(s^3)$ cost of computing the Cholesky factor $\Rb$ dominates the remaining $\tilde O(ns)$ operations required for performing the projection step. This raises the question of whether we can reuse the $\Rb$ computed in one step for any future steps. Naturally, we could do that if we encounter the same block set $S$ again in a subsequent iteration, but when sampling among all ${m\choose s}$ sets, this is very unlikely. To address this, we propose sampling among a small collection of blocks, $\Bc\subseteq {[m]\choose s}$. That way, we only have to compute $|\Bc|$ Cholesky factors, which can then be reused to speed up later iterations of the algorithm. This strategy, which we call \emph{block memoization}, enables using \alg\ effectively with even larger block sizes.

The crucial challenge with \emph{block memoization} is to ensure that the reduced amount of randomness in the block sampling scheme does not adversely affect the convergence rate. This challenge has been encountered by prior works which have considered sampling from a small collection of blocks, including the classical variant of block Kaczmarz \cite{elfving1980block} where the rows of $\A$ are partitioned into $m/s$ blocks of size $s$. However, despite efforts \cite{needell2014paved}, sharp convergence analysis for a partition-based block Kaczmarz has proven elusive. 

We demonstrate that, once again, introducing regularized projections resolves this crucial challenge: We show that Algorithm \ref{alg:main} using a collection $\Bc$ consisting of $O(\frac mk\log n)$ uniformly random blocks of size $s=\tilde O(k)$ achieves nearly the same (up to factor 2) convergence rate as if it was sampling from all ${[m] \choose s}$ blocks. Importantly, this is more blocks than the $m/s$ that would be obtained by simply partitioning the rows, but only by a logarithmic factor. This over-sampling factor appears necessary for fast convergence even in practice.

\begin{theorem}[Block memoization]\label{thm:block_memo}
Consider matrix $\A\in\R^{m\times n}$, 
parameters $\bar\lambda\geq \lambda>0$, and a probability distribution $\Dc$ over subsets of $[m]$, such that the regularized projection matrix $\P_{\reg,\setS} \coloneqq \A_{\setS}^\top(\A_{\setS}\A_{\setS}^\top + \reg\I)^{-1}\A_{\setS}$ satisfies:
\begin{align*}
    \bar\P_{\reg} \coloneqq \E_{S \sim \mathcal{D}}[\P_{\reg,\setS}] \succeq c\A^\top\A(\A^\top\A+\bar\lambda\I)^{-1}
\end{align*}
for some $c \in (0, 1]$. Assume that $\A_S^\top\A_S\preceq\alpha\A^\top\A$ holds with probability $1- \delta'$ for some $\alpha\in[0,1]$ and $\delta' \in [0,\alpha / \|\bar{\P}_{\lambda}^\dagger\|]$. Let $\mathcal{B}=\{S_i\}_{i=1}^B$ be a collection of $B$ independent samples from $\mathcal{D}$. If $B \geq \frac{40}{c}(1 + \alpha\frac{\bar\lambda}{\lambda})\cdot\log(2n/\delta)$ for some $\delta \in (0, 1)$, then with probability $1-B\delta'-\delta$, the collection $\Bc$ satisfies:
\begin{align*}
\frac{1}{B}\sum_{j=1}^B \P_{\reg, S_j} \succeq \frac{1}{2}\bar\P_{\reg}.
\end{align*}
\end{theorem}
The proof of Theorem \ref{thm:block_memo} follows along similar lines as the proof of Theorem \ref{l:nu-reg}, since it primarily requires bounding the variance of each term $\P_{\reg, S_j}$, so that we can apply a matrix concentration inequality (details in Supplement \ref{s:block_memo-proof}). 

The following corollary combines Theorem \ref{thm:block_memo} with Theorem \ref{l:mu-reg}, while also incorporating Lemma \ref{l:nu_misc} to account for over-determined systems. 
\begin{corollary}\label{cor:block_memo}
Suppose matrix $\A\in \R^{m \times n}$ with $\rank(\A) = r$ is transformed by RHT. Given $\delta\in(0,r/m \bar{\kappa}_k^2)$ and $C\log(m/\delta) \leq k < r$, let $\bar{\lambda} = \frac{1}{k}\sum_{i>k} \sigma_i^2$ and let $\Dc$ be uniform over size $s$ subsets of $[m]$ where $s\in [Ck \log (m\bar\kappa_k) , r]$. Let $\mathcal{B} = \{S_i\}_{i=1}^B$ be a collection of $B$ independent samples from $\Dc$. Then, for any $0 < \lambda \leq \frac{k}{m} \bar{\lambda}$, $\alpha = \min\{1,\frac{9r}{m }\log(mn\bar{\kappa}_k)\}$, and $B \geq 80(1+\alpha\frac{\bar\lambda}{\lambda} )\cdot\log(2n/\delta)$, with probability $1-\delta$ we have (i) $\frac{1}{B}\sum_{j=1}^B \P_{\reg, S_j} \succeq \frac{1}{2}\bar\P_{\reg}$, and (ii) $\A_{S_i}^\top \A_{S_i} \preceq \alpha \A^\top\A$ for any $i\in [B]$. 

By further choosing $\lambda = \frac{k}{m}\bar\lambda$, the required number of blocks is:
\begin{align*}
B\geq 80\left( \frac{\min\{m,9r \log(mn\bar{\kappa}_k)\}}{k} + 1\right)\cdot\log(2n/\delta). 
\end{align*}
\end{corollary}

\subsection{Overall Computational Analysis}\label{s:overall_comp}

In this section, we illustrate how all of the above results can be put together to achieve low computational cost for \alg, while ensuring fast convergence towards the optimum.

We start by combining Theorem \ref{thm:block_memo} with the analysis of $\mu$ and $\nu$ in Theorems \ref{l:mu-reg} and~\ref{l:nu-reg}, as well as the acceleration analysis from Theorem \ref{lem:converge_accelerate}, in order to recover the convergence guarantee for \alg (Algorithm \ref{alg:main}).

\begin{corollary}\label{c:convergence}
Given $\A\in\R^{m \times n}$ with $\rank(\A)=r$ such that $\A\x^*=\b$, let $\delta\in(0,r/m\bar{\kappa}_k^2)$ and $C\log(m/\delta) \leq k < r$.
There are $\lambda,\rho,\eta$ such that if Algorithm~\ref{alg:main} solves the regularized projection step (line \ref{line-w}) via Lemma \ref{l:inner} with $\epsilon\leq\frac{\rho^2}{8\eta}$ and samples $B\geq C\frac rk\log(mn\bar\kappa_k)\log(n/\delta)$ blocks, then conditioned on a $1-\delta$ probability event depending only on the RHT $\Q$ and block set $\Bc$, we have:
\begin{align*}
\E\|\x_t-\x^*\|^2\leq 8\left(1-\frac{k}{24 \tilde{r}}\bar\kappa_k\right)^t \cdot\|\x_0-\x^*\|^2 \quad\text{where}\quad \tilde{r} = \min\Big\{m, 3r \sqrt{\log(mn \bar{\kappa}_k)}\Big\} .
\end{align*}
\end{corollary}
\begin{proof}
 Choose $\lambda = \frac{1}{m}\sum_{i>k} \sigma_i^2(\A)$ and let $\Uc(\Bc)$ denote the uniform distribution over the block sets. Also, let $\bar\A=\Q\A$ denote the input matrix after RHT preprocessing. Conditioned on the $1-\delta$ probability event defined in Corollary \ref{cor:block_memo}, by applying Theorem~\ref{l:mu-reg} with $\Dc=\Uc(\Bc)$ we obtain the following:
\begin{align*}
\mu\big(\bar\A,\Uc(\Bc),\lambda\big)\geq \frac{k}{4r\bar\kappa_k^2}.
\end{align*}
Moreover, according to Corollary \ref{cor:block_memo}, for any $i \in [B]$, conditioned on the same event we have $\A_{S_i}^\top\A_{S_i} \preceq \min\{1,\frac{9r\log(mn\bar{\kappa}_k)}{m}\} \A^\top\A$ for all $i\in[B]$. Thus, by Theorem \ref{l:nu-reg},
 \begin{align*}
 \nu\big(\bar\A,\Uc(\Bc),\lambda\big)\leq \min\left\{\frac{8m}{k}, \frac{80r}{k} \log(mn\bar{\kappa}_k) \right\}.
 \end{align*}
 Combining the above bounds on $\mu$ and $\nu$, we obtain the convergence rate as
\begin{align*}
\bar\rho\big(\bar\A,\Uc(\Bc),\lambda\big) = \sqrt{\frac{\mu\big(\bar\A,\Uc(\Bc),\lambda\big)}{\nu\big(\bar\A,\Uc(\Bc),\lambda\big)}} \geq \frac{k}{6\tilde r\bar\kappa_k}.
\end{align*}
We conclude the proof via Theorem \ref{lem:converge_accelerate} by choosing $\rho=\bar\rho/2$ and $\eta = \frac1{2\nu}$.
\end{proof}

We are now ready to provide the overall computational analysis of our algorithm. While we allow optimal selection of the algorithmic hyper-parameters in this statement, note that all of our intermediate results show that the algorithm is robust to different choices of parameters $\lambda$, $\rho$, $\eta$, and $B$, and in the following section, we discuss how to efficiently select these parameters during runtime.

\begin{theorem}[Computational analysis]\label{thm:main}
Given $\A\in\R^{m \times n}$ with $\rank(\A)=r$ and $\b \in \R^m$, let $\x^*$ be the minimum-norm solution of  $\A \x = \b$. For $\delta\in(0,1)$ and $C\log(m/\delta) \leq k < r$, \alg\ (Alg.~\ref{alg:main} via Corr.~\ref{c:convergence}) with $s = \lceil Ck \log(m\bar\kappa_k)\rceil$, $\lambda =\frac1m\sum_{i>k}\sigma_i^2(\A)$, $\rho = \frac{k}{24r\bar{\kappa}_k\log^{1/2}(mn\bar{\kappa}_k)}, B=\lceil \frac{Cr}{k}\log(mn\bar{\kappa}_k)\log(2n/\delta) \rceil$, and $\eta=\frac k{160r\log(mn\bar\kappa_k)}$,  $\x_0=\mathbf{0}_n$,
after $t = \tilde O(\frac rk\bar{\kappa}_k \log 1/\epsilon\delta)$ iterations, with probability $1-\delta$ has
\begin{align*}
\|\x_t - \x^*\|\leq \epsilon \|\x^*\|\quad\text{using} \quad\tilde{O}\left(mn+ rk^2 + nr\bar{\kappa}_k \log 1/\epsilon\delta\right) \quad\text{operations.}
\end{align*}
\end{theorem}
\begin{proof}
Corollary \ref{c:convergence} implies that it takes $t=\tilde{O}(\frac rk\bar{\kappa}_k\sqrt{\log(mn\bar\kappa_k)} \log 1/\epsilon)$ iterations to converge $\epsilon$-close in expectation. We convert this to a guarantee that holds with $1-\delta$ probability by replacing $\epsilon$ with $\epsilon\delta$ and then applying Markov's inequality.

It remains to bound the costs associated with running the algorithm. 
\begin{enumerate}
    \item First, applying the RHT takes $O(mn\log m)$ operations. \item  Then, computing all of the Cholesky factors takes $O(B(ns\log n + s^3\log s)) = \tilde{O}(nr+rk^2)$.
    \item Finally, each iteration of the algorithm requires solving \eqref{eq:lsqr} with LSQR so that $\|\tilde\w_t-\w_t\|\leq \epsilon'\|\x_t-\x^*\|$ for $\epsilon'\leq \frac{k}{Cr\bar\kappa_k^2}$, which for a given block $S$ takes $O(ns\log(1/\epsilon'))$ operations.
    Thus, the cost of each iteration is $\tilde{O}(nk)$. Multiplying by $t$, we get $\tilde{O}(nr\bar{\kappa}_k\log1/\epsilon\delta)$ operations.
\end{enumerate}
Adding all of these costs together, we recover the claim.    
\end{proof}

\section{Improved Algorithm for Positive Semidefinite Systems}
\label{s:practical}
In this section, we propose a specialized implementation of \alg\ as a coordinate descent-type solver (\algpd, Algorithm \ref{alg:bcd}), which is optimized for square positive semidefinite linear systems. This algorithm not only gets an improved convergence rate compared to \alg\ for PSD matrices, but also admits a simplified block memoization scheme that avoids an inner LSQR solver. Along the way, we describe a novel adaptive scheme for tuning the acceleration parameters (relevant for both \alg\ and \algpd), as well as a fast implementation of the randomized Hadamard transform for symmetric matrices.
\subsection{Coordinate Descent}
 When the matrix $\A$ is PSD, then \alg\ admits a specialized formulation as a block coordinate descent method, similarly as can be done for the classical randomized Kaczmarz algorithm, see e.g. \cite{HefnyRows15,petra2015randomized}. By applying our \alg result to $\A^{1/2}$ we have the following convergence guarantee.
\begin{theorem}\label{t:cd++}
Suppose a PSD matrix $\A\in\mathcal{S}_n^{+}$ and $\b\in\R^n$ are transformed by RHT so that $\bar{\A} = \Q\A\Q^\top$ and $\bar\b=\Q\b$, and let $\x^*$ be the minimum-norm solution to $\A\x=\b$. Given $\delta \in (0,1)$ and $C\log(n/\delta) \leq k < \rank(\A)$, let $\Bc$ consist of $B$ uniformly random sets from ${[n]\choose s}$ and $\Uc(\Bc)$ be the uniform distribution over $\Bc$. If we run the following update (\algpd, see lines 10-13 of Algorithm~\ref{alg:bcd}):
\begin{align}
\begin{cases}
\w_t\leftarrow \I_{S}^\top(\bar\A_{S,S}+\lambda\I)^{-1}
(\bar\A_S\x_t-\bar\b_S)\text{ for } S\sim \Uc(\Bc)\\
\m_{t+1} \leftarrow \frac{1-\rho}{1+\rho}\big(\m_t - \w_t\big) \\
\x_{t+1} \leftarrow \x_t - \w_t + \eta \m_{t+1}
\end{cases}\label{eq:cd-update}
\end{align}
initialized with $\x_0\in\R^n$, block size $s = \lceil Ck \log(n\bar\kappa_k(\A^{1/2}))\rceil$, $\lambda =\frac1n\sum_{i>k}\lambda_i(\A)$, and the number of blocks $B=\lceil C\frac nk\log(n/\delta)\rceil$, then we have
\begin{align*}
\E\left\|\Q^\top\x_t - \x^*\right\|_{\A}^2 \leq 8\Big(1 - \frac{k}{24n\bar\kappa_k(\A^{1/2})}\Big)^t \cdot \|\Q^\top\x_0-\x^*\|_{\A}^2.
\end{align*}
\end{theorem}

\begin{proof}
We start by describing the reduction from \algpd\ to \alg. For the sake of notation, suppose that $\x_t = \mathcal A_t(\A,\b,\x_0)$ denotes $t$ updates from lines \ref{line:blocks}-\ref{line:update} of Algorithm \ref{alg:main} for solving a general linear system $\A\x=\b$ (with exact projections).

Given matrix $\A\in\mathcal{S}_n^{+}$, let $\mPhi \in \R^{n \times n}$ be such that $\A = \mPhi \mPhi^\top$. Since we denote $\bar\A = \Q\A\Q^\top$, we can first rewrite the linear system as $\bar{\A}\bar{\x} = \Q\b$ with $\x = \Q^\top\bar{\x}$, which is further equivalent to $\Q\mPhi\z = \Q\b$ with $\z = \mPhi^\top\Q^\top\bar{\x}$. Thus, the preprocessing step of \algpd on $\A\x=\b$ is equivalent to the preprocessing step of \alg on $\mPhi\z=\b$.
Denoting $\bar{\mPhi} = \Q\mPhi$ and $\bar{\b} = \Q\b$, let $\bar\mPhi^\top\x_t = \mathcal A_t(\bar{\mPhi},\bar{\b},\bar\mPhi^\top\x_0)$ be the implicit \alg\ iterates, denoted as $\z_t = \bar\mPhi^\top\x_t$. The projection step for these iterates is
\begin{align}\label{eq:implicit}
\bar{\mPhi}_S^\top(\bar{\mPhi}_S\bar{\mPhi}_S^\top+\reg\I)^{-1}(\bar{\mPhi}_S\z_t-\bar{\b}_S)=\bar\mPhi^\top\I_{S}^\top(\bar{\A}_{S,S}+\reg\I)^{-1}(\bar{\A}_S\x_t-\bar{\b}_S)
\end{align}
which leads to the coordinate descent update \eqref{eq:cd-update}. Notice that the iterates $\x_t$ of \algpd\ are actually solving the system $\bar\A\bar\x=\bar\b$, and thus converging to its solution, $\bar\x^*$, and not to $\x^*$. According to Corollary~\ref{c:convergence}, letting $\z^*$ denote the solution of the implicit system $\bar\mPhi\z=\bar\b$, we have the following:
\begin{align*}
\E\left\|\Q^\top\x_t - \x^*\right\|_{\A}^2 
&=  \E\left\|\bar\mPhi^\top(\x_t-\bar\x^*)\right\|^2=
\E\left\|\z_t - \z^*\right\|^2\\
&\leq 8\Big(1 - \frac{k}{24n\bar\kappa_k(\bar\mPhi)}\Big)^t \cdot \|\z_0-\z^*\|^2 
\\
&= 8\Big(1 - \frac{k}{24n\bar\kappa_k(\bar\mPhi)}\Big)^t \cdot \|\Q^\top\x_0-\x^*\|_{\A}^2.
\end{align*}
Choosing $\mPhi=\A^{1/2}$, we recover the claim. Note that the application of $\Q^\top$ (line \ref{l:stopping} in Algorithm \ref{alg:bcd}) transforms the iterate $\x_t$ back to solving the system $\A\x=\b$.
\end{proof}
Thus, the convergence rate of \algpd\ is determined by the convergence of the implicit \alg\ algorithm. 
Crucially, this convergence is now governed by \begin{align*}
\bar\kappa_k(\A^{1/2})\leq\sqrt{\bar\kappa_k(\A)},
\end{align*}
which improves over simply running \alg\ on $\A$ by at least the square root of $\bar\kappa_k(\A)$. Finally,  similar to Corollary~\ref{c:convergence}, the convergence rate in Theorem~\ref{t:cd++} can be improved when the matrix $\A$ is rank deficient, by replacing $n$ with $2r\sqrt{\log(n\bar\kappa_k)}$.

\subsection{Simplified Block Memoization}\label{s:block_memoization}
One of the main costs in the above block coordinate descent update is applying the inverse matrix $(\A_{S,S}+\lambda\I)^{-1}$ to a vector. Similarly to what we did in Section~\ref{s:blocks}, we can amortize this cost by sampling a collection of blocks, and pre-computing the Cholesky factors of $\A_{S,S}+\lambda\I$, so that all subsequent applications of $(\A_{S,S}+\lambda\I)^{-1}$ can be done using $O(s^2)$ operations, where $s$ is the block size.  Note that, there is no need for sketching or using LSQR in the \algpd\ version of this scheme, because we can compute the Cholesky factor exactly in $O(s^3)$ time, unlike in \alg, where this would take $O(ns^2)$ time (Section \ref{s:fast-projection}).

The key question is how to choose the number of blocks to pre-compute, in order to ensure effective convergence of the method. Our theory suggests that $O(\frac nk\log n)$ blocks is enough with high probability when block size is $s=\tilde O(k)$, but the constant/logarithmic factors matter significantly, since if we choose too few blocks up front, we may not end up with a convergent method, whereas too many blocks leads to significant unnecessary computational overhead. To address this, we propose an online block selection scheme, where the algorithm adds new blocks during the course of its convergence, but gradually shifts towards reusing the previously collected blocks.

Specifically, we initialize our block list $\Bc$ as empty, and then in iteration $t$:
\begin{enumerate}
    \item Sample a Bernoulli variable $b$ with success probability $\min\{\,1,\ \frac1t\cdot \frac ns\log n\,\}$. 
    \item If $b=1$, then sample new random block $S$ from ${[n]\choose s}$, store the Cholesky factor $\Rb[S]=\mathrm{chol}(\A_{S,S}+\lambda\I)\in\R^{s\times s}$, and add $S$ to the list $\Bc$.
    \item If $b=0$, then sample block $S$ from $\Bc$, and reuse the saved Cholesky $\Rb[S]$.
\end{enumerate} 
 This block selection strategy (which we also adapt for \alg\ in the supplement) implies that the first $B=\frac ns\log n$ blocks will be sampled uniformly at random from all size $s$ index sets and added to the block list $\Bc$. After that, in $T$ iterations the scheme will collect on average an additional $\sum_{t=B}^T\frac Bt\approx  B\log(\frac TB)$ blocks. Since the factor $\log(\frac TB)$ grows as the iterations progress, this implies that we are guaranteed to reach the number of blocks that is needed by our theory to imply convergence. However, since the factor grows slowly, we will not have to do too much unnecessary Cholesky factorizations before converging to a desired accuracy.
Note that each Cholesky factor takes $O(s^2)$ memory, so if we store $O(\frac ns\log n)$ factors throughout the convergence, then this uses only $O(ns\log n)$ additional memory.

\subsection{Symmetric Randomized Hadamard Transform}
To ensure that uniformly sampled blocks yield a fast convergence rate for our algorithm, we must preprocess the linear system.
Specifically, in the PSD case, where we assume that $\A=\mPhi\mPhi^\top$, we need to apply the randomized Hadamard transform $\Q$ to $\mPhi$ and $\b$, so that our theoretical analysis can be applied to the implicit block Kaczmarz algorithm based on \eqref{eq:implicit}. In the context of \algpd, with access to $\A$ and not $\mPhi$, this corresponds to transforming the original system into:
\begin{align*}
    \Q\A\Q^\top\bar\x=\Q\b,\qquad \x=\Q^\top\bar\x,
\end{align*}
where applying $\Q$ on both sides of $\A$ is crucial to maintaining the PSD structure of the system. 
Recall that the transform can be defined as $\Q = \H\D$, where $\D = \frac1{\sqrt n}\diag(d_1,...,d_n)$ and $d_i$ are independent random $\pm 1$ signs (Rademacher variables). 
So, the dominant cost of this preprocessing step is applying the Hadamard transform $\H$ on both sides of the matrix $\D\A\D$. The classical recursive algorithm for doing this, which we refer to as the Fast Hadamard Transform (FHT, see Appendix \ref{s:symfht}), takes $mn\log n$ operations to compute $\FHT(\M)=\H\M$ for an $n\times m$ matrix $\M$. Thus, the cost of computing $\Q\A\Q^\top=\FHT(\FHT(\D\A\D)^\top)$ is roughly $2n^2\log n$ operations. 
This suggests that, in order to maintain the positive definite structure for \algpd, we must double the preprocessing cost compared to \alg.

We show that this trade-off can be entirely avoided: By exploiting the symmetric structure of $\A$, we perform the two FHTs simultaneously at the cost of one FHT applied to a general matrix. We achieve this using a specialized recursive algorithm, which we call SymFHT (Algorithm \ref{alg:symfht}). Here, for simplicity we write the recursion assuming that the input matrix is at least $2\times 2$. Naturally, the base case of the recursion is a $1\times 1$ input matrix, in which case we let $\SymFHT(\A)=\A$.

\begin{algorithm}[!ht]
\caption{Symmetric Fast Hadamard Transform (SymFHT)}
\label{alg:symfht}
\begin{algorithmic}[1]
\Function{SymFHT}{$\A$} \Comment{Input: Symmetric matrix $\A=\begin{bmatrix}
    \A_{11} & \A_{12} \\
    \A_{12}^\top & \A_{22}
\end{bmatrix}$.}
\State Compute $\B_{11} \leftarrow\SymFHT(\A_{11})$\Comment{Recursive call.}
\State Compute $\B_{22} \leftarrow\SymFHT(\A_{22})$\Comment{Recursive call.}
\State Compute $\B_{12} \leftarrow \FHT(\FHT(\A_{12}^\top)^\top)$
\State Compute 
$\begin{bmatrix}
    \C_{11} & \C_{12}\\
    \C_{21} & \C_{22}
\end{bmatrix} \leftarrow 
\begin{bmatrix}
    \B_{11}+\B_{12}^\top & \B_{11} - \B_{12}\\
    \B_{12}+\B_{22} & \B_{12}^\top-\B_{22}
\end{bmatrix}$\\
\State \textbf{return} $\begin{bmatrix}
    \C_{11}+\C_{21} & \C_{12}+\C_{22} \\
    \C_{12}^\top+\C_{22}^\top & \C_{12}-\C_{22}
\end{bmatrix}$\Comment{Computes $\H\A\H$.}
\EndFunction
\end{algorithmic}
\end{algorithm}

Note that, given an $n\times n$ matrix $\A$ broken down into four $n/2\times n/2$ blocks, the function SymFHT performs two recursive calls corresponding to the two diagonal blocks $\A_{11}$ and $\A_{22}$, since both of these blocks are symmetric. The two off-diagonal blocks $\A_{12}$ and $\A_{12}^\top$ are not symmetric, so we must revert back to applying the classical FHT twice. But crucially, since the off-diagonal blocks are identical up to a transpose, we only have to transform one of them, which gives our recursion its computational gain, as shown in the following result. See Appendix~\ref{s:symfht} for proof.
\begin{theorem}\label{t:symFHT}
    Given an $n\times n$ symmetric matrix $\A$, where $n$ is a power of $2$, Algorithm~\ref{alg:symfht} returns $\H\A\H$ after at most $n^2(2.5+\log n)$ arithmetic operations.
\end{theorem}

\subsection{Error Estimation and Adaptive Tuning}
The remaining challenge with making our algorithms practical is effectively tracking the progress of the convergence, without spending significant additional computational cost. This progress tracking is important both for designing an effective stopping criterion for the algorithm, as well as to tune the acceleration parameters $\rho$ and $\eta$.

\paragraph{\textbf{Stopping criterion.}} Solvers such as CG commonly use a stopping criterion based on the relative residual error, i.e., $\|\A\x_t-\b\|/\|\b\|\leq \epsilon$ for some target value of $\epsilon$. However, unlike in CG where the residual vector $\A\x_t-\b$ is computed as part of the method, in a Kaczmarz-type solver this vector is never explicitly computed, so using this stopping criterion directly would substantially add to the overall cost.

Instead, we propose to estimate the residual error by reusing the computations from the Kaczmarz updates. Specifically, in each update we compute the vector $\r_t=\A_{S_t}\x_t-\b_{S_t}$, which can be viewed as a sub-sample of the coordinates of $\bar\r_t=\A\x_t-\b$. Thus, we can use $\frac ns\|\r_t\|^2$ as a nearly-unbiased estimate of $\|\bar\r_t\|^2$ (the bias comes due to our block memoization scheme reusing previously sampled subsets; this bias is insignificant in practice). We propose to use a running average of these estimates:
\begin{align}
    \Ec_{t,p} = \frac1{p}\sum_{i=t-p}^{t}\frac ns\|\A_{S_i}\x_i-\b_{S_i}\|^2 \approx \|\A\x_t-\b\|^2.\label{eq:residual}
\end{align}
This leads to our stopping criterion, $\Ec_{t,p} \leq \epsilon^2\|\b\|^2$, where we let $p= n/s$, so that the estimate is most likely based on nearly all of the rows of $\A$ (line \ref{l:stopping} in Algorithm \ref{alg:bcd}).

\paragraph{\textbf{Acceleration Tuning.}} 
We next discuss how we can use runtime information to adaptively tune the acceleration parameters $\rho$ and $\eta$.
Recall from Theorem \ref{lem:converge_accelerate} that the momentum vector recursion $\m_{t+1} = \frac{1-\rho}{1+\rho}(\m_t-\w_t)$ is tied to its guaranteed expected convergence rate $\E\frac{\|\x_t-\x^*\|^2}{\|\x_0-\x^*\|^2}\sim(1-\rho/2)^t$ via the parameter $\rho$. A natural strategy is thus to use the algorithm's ongoing rate of convergence as a proxy for $\rho$, by comparing the residual norm estimates \eqref{eq:residual} at two different iterates.

Specifically, we propose to compute the ratio $r_{i,p} = \Ec_{t_i,p}/\Ec_{t_i-p,p} \approx \frac{\|\A\x_{t_i}-\b\|^2}{\|\A\x_{t_i-p}-\b\|^2}$ at certain checkpoints $t_i$ during the run, and use them to recover a value of $\rho$ that will be used in the momentum recursion. One possibility would be to simply let $\hat\rho_i = 1 - r_{i,p}^{1/p}$, which would give the most up-to-date convergence rate of the algorithm over the last $p$ iterations until iteration $t_i$. However, this results in a feedback loop, since the momentum recursion affects the convergence of the algorithm and vice versa, leading to the convergence rate oscillating up and down.
To achieve stable convergence, we maintain a weighted average $\hat r_i$ of the ratios $r_{i,p}$, and use $\hat\rho_i=1-\hat r_i^{1/p}$ (line \ref{l:rho-proxy}). Concretely, we follow the parameter-free weighted averaging scheme proposed by \cite{na2023hessian}, which ``forgets" older estimates quickly, while converging to a stable estimate:
\begin{align}
\hat r_i = \frac{a_{t_i-1}}{a_{t_i}}\hat r_{i-1} + \Big(1 - \frac{a_{t_i-1}}{a_{t_i}}\Big) r_{i,p},
\quad\text{for}\quad a_{t_i}= (i+1)^{\log(i+1)}.\label{eq:weighted}
\end{align}

Finally, we choose the momentum step size parameter $\eta$ as indicated by our theory. From Theorem \ref{lem:converge_accelerate}, we need $\eta=\Theta(\frac1\nu)$ where $\nu$ is the variance parameter of the regularized projection. Then, using Theorem \ref{l:nu-reg}, we have that $\nu=\tilde O(\frac ns)$, where $s$ is the block size. This suggests $\eta\sim\frac sn$, and we simply set it to $\frac s{2n}$.

Combining the above ideas, we obtain \algpd, given in Algorithm \ref{alg:bcd}. Also, Algorithm \ref{alg:kzpp} in the supplement  provides the complete pseudocode for \alg, including the above adaptive tuning scheme as well as block memoization with preconditioned LSQR from Section \ref{s:blocks}.

\begin{algorithm}[!ht]
\caption{\algpd: Coordinate descent solver for positive semidefinite systems}
\label{alg:bcd}
\begin{algorithmic}[1]
\State \textbf{Input: } $\A\in\mathcal{S}_n^{+}$, $\b \in \R^n$, block size $s$, iterate $\x_0$, regularization $\reg$, tolerance $\epsilon$;
\State Sample $\D\leftarrow \frac1{\sqrt n}\diag(d_1,...,d_n)$\text{ for }$d_i\sim\mathrm{Rademacher}$;
\State $\A\leftarrow \SymFHT(\D\A\D)$, $\b\leftarrow \FHT(\D\b)$;
\Comment{{\footnotesize See Algorithm \ref{alg:symfht}.}}
\State Initialize $\m_0 \leftarrow \mathbf{0}$, $\rho\leftarrow 0$, $\eta\leftarrow \frac s{2n}$, $\Bc\leftarrow \emptyset$, $\zeta\leftarrow \lceil n/s\rceil$, $\Ec_0,\Ec_1\leftarrow 0$;
\For{$t=0,1,...$}
\State \textbf{if} Bernoulli$\big(\min\{\,1,\ \frac1t\cdot \frac ns\log n\,\}\big)$
\textbf{then} \label{l:sampling}
\State \quad\ $\Bc \leftarrow \Bc \cup\{S\}$ \text{for} $S\sim {[n]\choose s}$; \Comment{{\footnotesize Sample new subset.}}
\State \quad\ $\Rb[S] = \textrm{chol}(\A_{S,S}+\lambda\I)$; \Comment{{\footnotesize Save Cholesky factor.}}
\State \textbf{else} $S\sim \Bc$; \textbf{end if} 
\State $\r_t \leftarrow \A_S\x_t-\b_S$;\Comment{{\footnotesize Use for error estimation.}}
\State $\w_t\leftarrow \I_{S}^\top(\A_{S,S}+\lambda\I)^{-1}
\r_t$ using $\Rb[S]$;\Comment{{\footnotesize Coordinate descent.}}\label{l:wcd}
\State $\m_{t+1} \leftarrow \frac{1-\rho}{1+\rho}\big(\m_t - \w_t\big)$;\Comment{{\footnotesize Adaptive momentum.}}\label{l:momentum}
\State $\x_{t+1} \leftarrow \x_t - \w_t + \eta\,\m_{t+1}$;
\State \textbf{if} $t<\zeta \mod 2\zeta$ \textbf{then} $\Ec_0 \leftarrow \Ec_0+\|\r_t\|^2$; \textbf{else} $\Ec_1 \leftarrow \Ec_1+\|\r_t\|^2$; \label{l:update_start}
\If{$t = 2\zeta-1 \mod 2\zeta$}
\State \textbf{if} $\Ec_1\leq \epsilon^2\|\b\|^2$ \textbf{then} \textbf{return $\D\cdot\FHT(\x_{t+1})$}; \Comment{{\footnotesize Stopping criterion.}}\label{l:stopping}
\State $r\leftarrow r a_t + (\Ec_1/\Ec_0) (1-a_t)$; \Comment{{\footnotesize Weighted average \eqref{eq:weighted}.}}
\State $\rho\leftarrow 1 - r^{1/\zeta}$;
\Comment{{\footnotesize Convergence rate estimate.}}\label{l:rho-proxy}
\State $\Ec_0 \leftarrow\Ec_1\leftarrow 0$;
\EndIf \label{l:update_end}
\EndFor
\end{algorithmic}
\end{algorithm}

\section{Numerical Experiments}
\label{s:experiments}

In this section, we present numerical experiments that support our theory. First, we demonstrate that the convergence analysis carried out in Sections \ref{s:reformulation}-\ref{s:blocks} accurately predicts the performance of \alg on a collection of synthetic linear system tasks, by evaluating the effect of adaptive acceleration, block memoization and inexact projections on the convergence rate. Then, focusing on a real-world positive definite system task, we evaluate \algpd\ against popular Krylov solvers, CG \cite{hestenes1952methods} and GMRES \cite{saad1986gmres}, showing that our algorithmic framework achieves better computational cost for certain classes of linear systems that arise naturally in applications such as machine learning, as suggested by our theory. We note that due to the limitations in our computational setup, we evaluate our algorithms on moderately sized matrices, which showcases our theoretical contributions, and do not report running times or test against direct~solvers.

\subsection{Experimental Setup}
We set up our experiments to evaluate how well the solvers exploit large outlying singular values to achieve fast convergence for ill-conditioned systems. To that end, we consider two families of linear systems which naturally exhibit such spectral structure (see Appendix \ref{s:further-experiments} for formal definitions).

\paragraph{\textbf{Synthetic Low-Rank Matrices.}} To gain precise control on the eigenvalue distribution of the system, we first consider a collection of synthetic benchmark matrices with a bell-shaped spectrum, constructed via the \verb~make_low_rank_matrix~ function in Scikit-learn \cite{scikit-learn}. We control the number of large outlying singular values via the parameter \verb~effective_rank~, choosing among four values (25, 50, 100 and 200).

\paragraph{\textbf{Kernel Matrices from Machine Learning.}}
We also consider four real-world benchmark datasets, available through Scikit-learn \cite{scikit-learn} and OpenML \cite{OpenML2013}: Abalone, California\_housing, Covtype, and Phoneme. We transform these datasets using a kernel function to produce a PSD kernel matrix. For the kernel function, we consider two types of radial basis functions (Gaussian and Laplacian), each with two different values of width $\gamma\in\{0.1,0.01\}$. These are known to produce highly ill-conditioned matrices with fast spectral decays, leading to many large outlying eigenvalues \cite{RasmussenWilliams06}.
Each resulting PSD matrix is truncated to dimensions $4096\times4096$. Then, following standard applications in Kernel Ridge Regression \cite{alaoui2015fast}, we augment the matrix with a regularization term $\phi\I$, where $\phi=0.001$.

For each resulting test matrix $\A$, we solve the linear system $\A\x=\b$, where $\b$ is a standard Gaussian vector, and we evaluate the estimates via the normalized residual:
\begin{align}
\epsilon=\|\A\x-\b\|/\|\b\|.\label{eq:residual-error}
\end{align}

\subsection{Verifying Our Convergence Analysis}
To verify our convergence analysis, we implement four variants of \alg and plot the per-iteration convergence on ill-conditioned synthetic low-rank matrices (details in Appendix \ref{s:further-experiments}). Note that for these experiments we are testing \alg on rectangular $4096 \times 1024$ linear systems (constructed via the \texttt{make\_low\_rank\_matrix} function in Scikit-learn \cite{scikit-learn}).

\begin{wraptable}{r}{0.52\textwidth}
\vspace{-1mm}
\centering\begin{tabular}{|c|c|c|}
\hline
Solver Name & Accelerate & Memoize \\
\hline
\texttt{Kaczmarz} & \xmark & \xmark \\
\algmemo & \xmark & \cmark \\
\algaccel & \cmark & \xmark \\
\texttt{Full K++} & \cmark & \cmark \\
\hline
\end{tabular}
\caption{Evaluated variants of \alg. Analogous names apply in Section \ref{s:experiments_krylov} with \textnormal{\texttt{CD++}} instead~of~\textnormal{\algshort}.}
\vspace{-5mm}
\label{tab:kacz-variants}
\end{wraptable}
\textit{\textbf{Inner solver.}} First, we test the influence of computing the inner projection steps inexactly via sketch-and-precondition LSQR (Section~\ref{s:fast-projection}). As shown in Figure~\ref{fig:lsqr_steps}, even very few steps of LSQR suffice to attain the fast linear convergence of \alg. With only $8$ steps of LSQR, there is barely any difference between \alg and the alternative using exact projections. This is supported by our theory: as discussed in Lemma~\ref{l:inner}, by using a preconditioner constructed via SRHT with sketch size proportional to the block size, the condition number of the arising sub-problem can be reduced to $\leq 2$. In our experiments, we observed that using sketch size $\tau$ equal twice the block size is sufficient.

Next, to isolate the effect of its individual
components (such as block size, acceleration, and memoization), we
consider four variants of \alg\ that toggle acceleration and
memoization on/off, as illustrated in Table \ref{tab:kacz-variants}. In Figure
\ref{fig:accel-sample}, we show  convergence plots for synthetic
matrices with effective rank $50$ and $100$, each with block sizes $100$ and $200$ (the plots for matrices with effective rank 25 and 200 are in Appendix \ref{s:experiment_acc}). In
particular, no memoization means that  we sample a new block set $S$
uniformly at random from ${[m]\choose s}$ at every step (and compute
its Cholesky factor), while no acceleration means setting the momentum
step size $\eta$ to $0$.

\begin{figure}
    \centering
    \includegraphics[width=\linewidth]{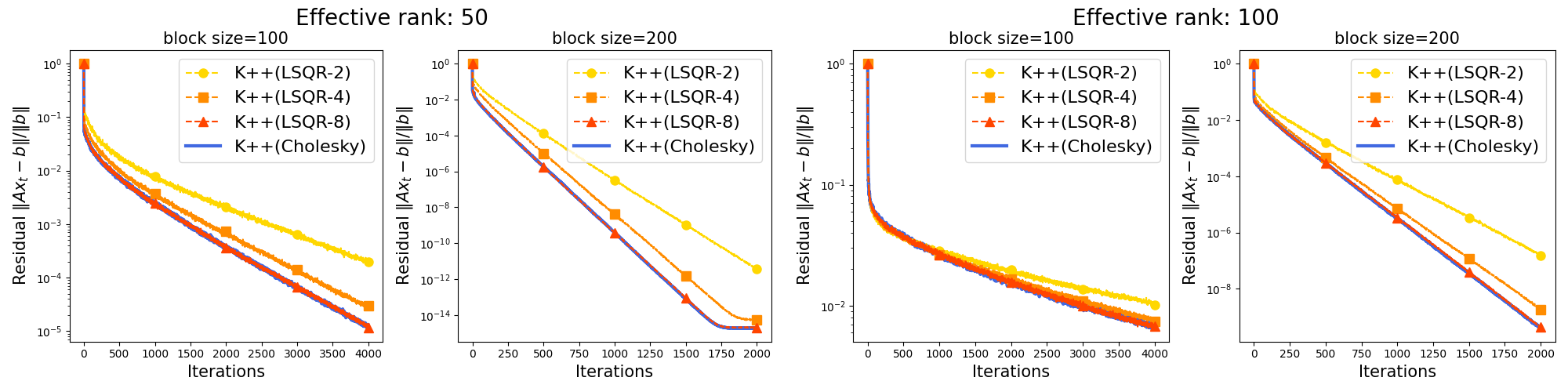}
    \vspace{-8mm}
    \caption{Convergence plots for \alg varying the number of steps of preconditioned LSQR in each regularized projection step, with two block sizes (100, 200) on synthetic test matrices with effective rank 50 (left) and 100 (right). \textnormal{\texttt{K++(LSQR-X)}} is \alg using \textnormal{\texttt{X}} steps of LSQR, while \textnormal{\texttt{K++(Cholesky)}} means \alg with exact inner solver using Cholesky decomposition.}
    \label{fig:lsqr_steps}
\end{figure}

\begin{figure}
    \centering
    \includegraphics[width=\linewidth]{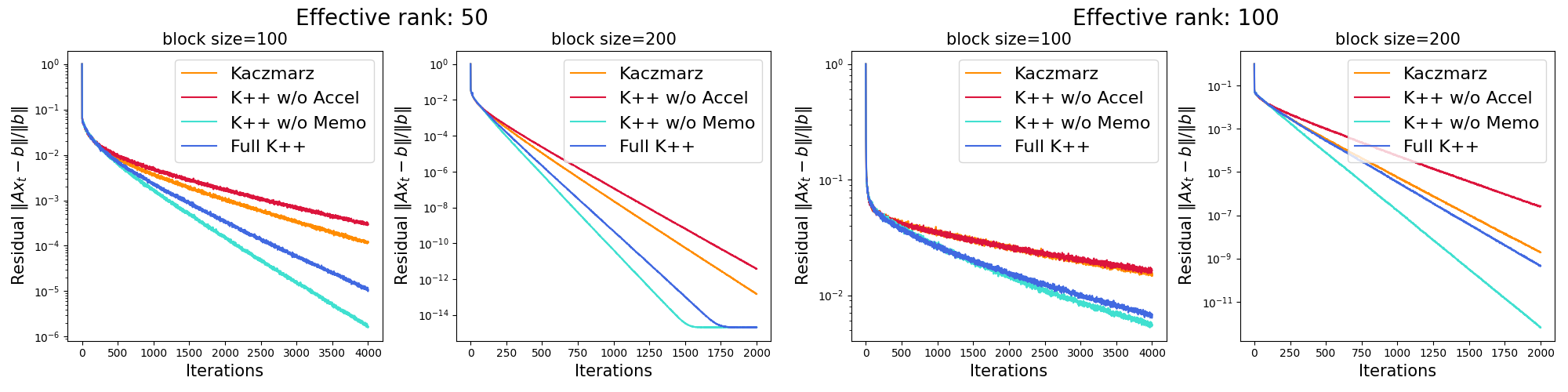}
    \vspace{-8mm}
    \caption{Convergence plots for different variants of \alg (see Table \ref{tab:kacz-variants}) using two block sizes (100, 200), on synthetic test matrices with effective rank 50 (left) and 100 (right). Note that the iteration range in each plot is scaled to keep \emph{(iterations}\,$\times$\,\emph{block-size)} consistent.}
    \label{fig:accel-sample}
\end{figure}

\paragraph{\textbf{Block size.}}
First, observe that as we increase the block size, all of the methods achieve faster per-iteration convergence rate, as expected from our theory. Looking closely, we can see that this rate improves more than just proportionally to the increase in block size, which means that larger blocks lead to greater efficiency in terms of how many rows of $\A$ need to be processed by the algorithm to reach certain accuracy. This corresponds to the $\bar\kappa_k$ condition number in our theory, which decreases as the block size increases, indicating that the methods do exploit large outlying singular values. Also, as we increase the effective rank of the matrix (i.e., the number of large singular values, $k$ in our theory), we need larger block sizes to attain fast convergence.

\paragraph{\textbf{Adaptive acceleration.}} Comparing \texttt{Full K++} and \algmemo, we see that our adaptive acceleration significantly improves the convergence rate, particularly for the smaller block size 100. Note that if the block size is chosen to be much larger than the effective rank, then the benefit of acceleration will necessarily get diminished, because in this case the condition number $\bar\kappa_k$ is just a small constant.

\paragraph{\textbf{Block memoization.}} Comparing \texttt{Full K++} and \algaccel, we see that introducing our block memoization scheme (i.e., sampling from a small collection of random blocks) only slightly reduces the per-iteration convergence rate compared to the method that samples a new random block at every step. This is explained by our analysis in Theorem \ref{thm:block_memo}. The benefits of block memoization become apparent once we compare the computational cost of both procedures, as discussed below (Figure~\ref{fig:flops-sample}).

\paragraph{\textbf{Regularized projections.}} In Appendix \ref{s:experiments-regularization}, we also tested the effect of different choices of the projection regularizer $\lambda$ on our methods. In all our experiments, the convergence remained largely unchanged for any $\lambda\in[0,0.01]$ (our theory requires $\lambda>0$). This suggests that regularizing the Kaczmarz projection can be done without sacrificing convergence, and so, we recommend using a small but positive~$\lambda$ to ensure stable computation of the Cholesky factors (we used $\lambda=10^{-8}$ as a default).

\begin{figure}
    \centering
\hspace{-5mm}
\resizebox{\linewidth}{!}{
\begin{tikzpicture}
\begin{axis}[
    ybar,
    width=17cm,  
    height=5cm,
    bar width=26.85pt,
    enlarge x limits=0.13,
    ymin=0, ymax=8.5,
    legend style={at={(0.15, 1.19)}, anchor=north west, legend columns=4, reverse legend,font=\small},
    symbolic x coords={Gaussian$\ \gamma$=0.1, Gaussian$\ \gamma$=0.01, Laplacian$\ \gamma$=0.1, Laplacian$\ \gamma$=0.01},
    xtick=data, 
    ylabel={FLOPs ($\times 10^{9}$)},
    ylabel style={font=\footnotesize},
    xticklabel style={font=\footnotesize},
    yticklabel style={font=\footnotesize},
]

\addplot+[
    bar shift=+0.48cm,
    fill=cyan!20,
    draw=cyan!20
]
coordinates {(Gaussian$\ \gamma$=0.1, 3.26) (Gaussian$\ \gamma$=0.01, 2.11) (Laplacian$\ \gamma$=0.1, 8.13) (Laplacian$\ \gamma$=0.01, 3.09)};

\addplot+[
    bar shift=+0.48cm,
    fill=cyan!60,
    draw=cyan!60
]
coordinates {(Gaussian$\ \gamma$=0.1, 0.464) (Gaussian$\ \gamma$=0.01, 0.297) (Laplacian$\ \gamma$=0.1, 2.22) (Laplacian$\ \gamma$=0.01, 0.240)};

\addplot+[bar shift=-0.48cm,
    fill=red!20,
    draw=red!20
]
coordinates {(Gaussian$\ \gamma$=0.1, 5.17) (Gaussian$\ \gamma$=0.01, 1.75) (Laplacian$\ \gamma$=0.1, 8.14) (Laplacian$\ \gamma$=0.01, 4.11)};

\addplot+[
    bar shift=-0.48cm,
    fill=red!60,
    draw=red!60
]
coordinates {(Gaussian$\ \gamma$=0.1, 1.47) (Gaussian$\ \gamma$=0.01, 0.85) (Laplacian$\ \gamma$=0.1, 2.72) (Laplacian$\ \gamma$=0.01, 0.541)};

\legend{\algpd (1e-8), \algpd (1e-4), GMRES (1e-8), GMRES (1e-4)}
\end{axis}
\end{tikzpicture}
}
\includegraphics[width=\linewidth]{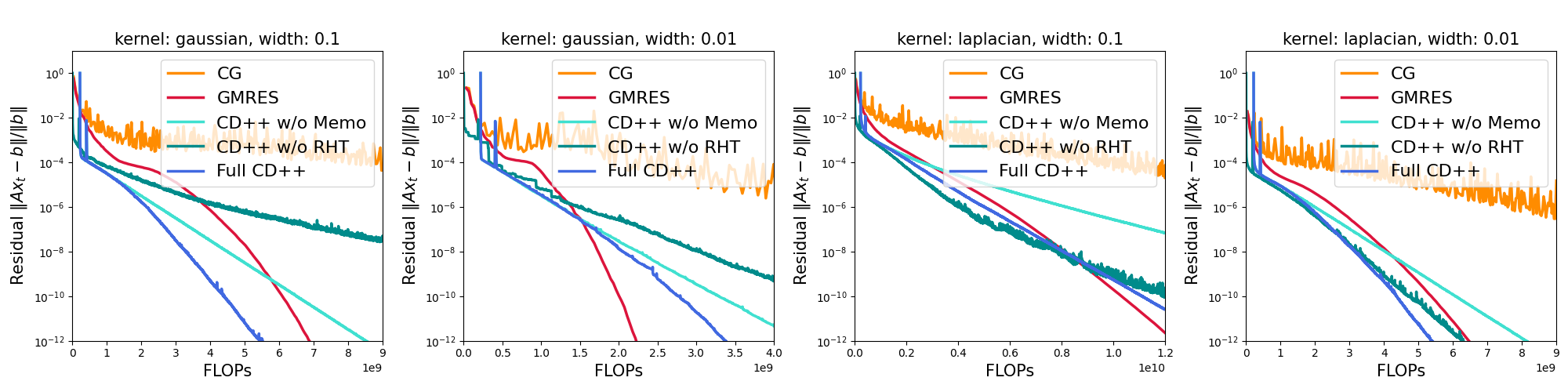}
\vspace{-8mm}    
    \caption{Computational cost comparison, measuring floating point operations (FLOPs) needed to reach a given error threshold on four kernel matrices constructed from the Abalone dataset. Above, we show total FLOPs for GMRES and \algpd\ to reach one of two error thresholds, $\epsilon\in\{10^{-4},10^{-8}\}$. Below, we show convergence-vs-FLOPs plots, including CG and \algpdaccel\ as additional baselines. We also include \algpd w/o RHT to test the influence of RHT on the convergence.}
    \label{fig:flops-sample}
\end{figure}

\subsection{Comparison with Krylov Subspace Methods}\label{s:experiments_krylov}
Next, we evaluate the computational cost of our methods, comparing them to Krylov solvers. Specifically, we count floating point operations (FLOPs) needed to reach a given error threshold for the positive definite linear systems arising from benchmark kernel matrices in machine learning. Here, we show the results for the Abalone dataset, while results for the remaining test matrices are in Appendix \ref{s:experiment_flops}.
In Figure \ref{fig:flops-sample} (bottom) we show the convergence plots of the methods, with FLOPs instead of iterations on the x-axis (this includes the cost of RHT pre-processing, when appropriate). Since the tasks are positive definite, we focus on evaluating different variants of \algpd. 

First, we compare \texttt{Full CD++} and \texttt{CD++ w/o Memo} (i.e., Algorithm~\ref{alg:bcd} with and without block memoization, using block size 200). We see that, even though block memoization slightly reduces the per-iteration convergence (Figure \ref{fig:accel-sample}), it more than makes up for this in the computations. In particular, \texttt{Full CD++} exhibits a phase transition where it accelerates past \texttt{CD++ w/o Memo} once enough blocks have been memoized and it no longer pays for the block Cholesky factorizations.

We also evaluate the effect of RHT preprocessing on the overall computational cost of the method by comparing \texttt{Full CD++} and \texttt{CD++ w/o RHT}. The preprocessing cost itself is relatively negligible, as can be seen by how much the beginning of the convergence curve of \texttt{Full CD++} is shifted away from 0 (by around 0.2e9 FLOPS). Furthermore, RHT provides a substantial gain in overall performance for roughly half of the test matrices, and this can be seen consistently across the remaining plots in the supplement. Importantly, even without RHT preprocessing, \algpd\ exhibits good convergence for most of the problems in our testing pool. These results suggest the possibility that \texttt{CD++ w/o RHT} may also be effective for sparse systems, where RHT preprocessing is not desirable. However, further empirical evidence on large-scale sparse systems could be interesting future work to address this possibility. 

Finally, we compare the computational cost of \algpd\ with two Krylov solvers\footnote{We use the SciPy implementation of CG \cite{virtanen2020scipy} and the PyAMG implementation of GMRES \cite{pyamg2023}, which were chosen based on how amenable they are to our implementation of FLOPs counting.}, CG and GMRES. First, we observe that CG struggles to converge on all of the tested matrices. This suggests that the systems are indeed ill-conditioned, and CG cannot overcome this effectively due to numerical stability issues. On the other hand, GMRES avoids these issues by maintaining an explicit Krylov basis, and thus after a number of initial iterations, it exploits the large outlying eigenvalues to achieve fast convergence. 

Thus, in most cases, both \algpd\ and GMRES exhibit two distinct phases of convergence, as suggested by the theory. However, in the majority of our test cases, the second phase of \algpd\ starts sooner than for GMRES, matching our complexity analysis from equations \eqref{eq:intro-kaczmarz} and \eqref{eq:intro-krylov}. This gives \algpd\ a computational advantage over GMRES, particularly in the low-to-moderate precision regime (say, $\epsilon>10^{-6}$). In the high precision regime (say, $\epsilon<10^{-6}$), \algpd\ maintains its fast convergence, while GMRES, in some cases, further accelerates by exploiting smaller isolated eigenvalues. 

These overall trends are reflected in Figure \ref{fig:flops-sample} (top), showing the total FLOP counts of GMRES and \algpd\ with two error thresholds, $\epsilon\in\{10^{-4},10^{-8}\}$. We see that  for $\epsilon=10^{-4}$, \algpd\ is consistently more efficient than GMRES, while for $\epsilon=10^{-8}$, GMRES overtakes \algpd in some cases. Across all tested matrices (see Appendix~\ref{s:experiment_flops} for details), we observed that for $\epsilon=10^{-4}$, \algpd\ performed better than GMRES in 18 out of 20 cases, while for $\epsilon=10^{-8}$, it did so in 14 out of 20 cases. Interestingly, \texttt{CD++ w/o RHT} outperformed GMRES on all 20 test matrices for $\epsilon=10^{-4}$, while for $\epsilon=10^{-8}$, it did so in 9 out of 20 cases. This suggests that RHT preprocessing provides a significant computational benefit mainly in the high-precision regime, and that it might be preferable to skip this step when moderate precision is sufficient.

\section{Conclusions}
\label{s:conclusions}

We developed new Kaczmarz methods, called \alg\ and \algpd, which exploit large outlying singular values to attain fast convergence in solving many ill-conditioned linear systems. We demonstrated both theoretically and empirically that our methods outperform Krylov solvers such as CG and GMRES on certain classes of problems that arise naturally, for instance, in the machine learning literature. Along the way, we introduced and analyzed several novel algorithmic techniques to the Kaczmarz framework, including adaptive acceleration, regularized projections and block memoization.

A potential future direction is to develop Kaczmarz-type methods that can exploit not just large but also small isolated singular values to achieve fast convergence, as motivated by applications in partial differential equations, among others. Another natural question is whether it is possible to adapt Krylov subspace methods to take advantage of row-sampling techniques in order to improve their computational guarantees beyond what is possible through only matrix-vector products. Finally, large-scale implementation and experiments evaluating our methods in terms of wall-clock time, as well as for sparse problems, is an exciting next step.

\subsubsection*{Acknowledgments} Thanks to Daniel LeJeune for helpful discussions regarding accelerated sketch-and-project, and for sharing the programming environment. Also, thanks to Sachin Garg for helpful discussions on block memoization. We also thank the editor and reviewers for their useful feedback that greatly improved the manuscript.

\bibliographystyle{alpha}
\bibliography{simax.bib}

\appendix

\section{Acceleration Analysis: Proofs of Lemmas \ref{lem:converge_accelerate_form2} and \ref{lem:equivalence}}

\subsection{Proof of Lemma~\ref{lem:converge_accelerate_form2}}\label{s:inexact}
\begin{proof}[Proof of Lemma~\ref{lem:converge_accelerate_form2}]
First, we show that $1\leq \nu\leq 1/\mu$. This follows since:
\begin{align*}    
1 &=
\|(\bar\P_\lambda^{\dagger/2}\E[\P_{\lambda,S}]\bar\P_\lambda^{\dagger/2})^2]\|
\leq \|\E[(\bar\P_\lambda^{\dagger/2}\P_{\lambda,S}\bar\P_\lambda^{\dagger/2})^2]\|=\nu,\\
\nu&\leq \|\bar\P_\lambda^\dagger\|\|\bar\P_\lambda^{\dagger/2}\E[\P_{\lambda,S}^2]\bar\P_\lambda^{\dagger/2}\| \leq\|\bar\P_\lambda^\dagger\|= 1/\mu,
\end{align*}
where we used Jensen's inequality and that $\P_{\lambda,S}\preceq \I$.

Our proof of the convergence guarantee follows closely the steps of \cite{derezinski2024fine}, who introduced inexact updates into the argument of \cite{gower2018accelerated}.
Denote $\e_t \coloneqq \tilde{\w}_t - \w_t$ as the approximation error, and denote $\v_{t+1}^* \coloneqq \v_{t+1} + \gamma \tilde{\w}_t - \gamma\w_t = \v_{t+1} + \gamma \e_t$ as the exact update. Let $\r_t \coloneqq \|\v_t - \x^*\|_{\bar{\P}_{\reg}^\dagger}$ and $\r_t^* \coloneqq \|\v_t^* - \x^*\|_{\bar{\P}_{\reg}^\dagger}$ where $\bar{\P}_{\reg} = \E[\P_{\reg, S}]$. We have the following: 
\begin{align}\label{eq:error_1}
\E[\r_{t+1}^2] = & ~ \E[\|\v_{t+1} - \x^*\|_{\bar{\P}_{\reg}^\dagger}^2] = \E[\|\v_{t+1}^* - \x^* - \gamma\e_t\|_{\bar{\P}_{\reg}^\dagger}^2] \nonumber \\
\leq & ~ (1+\phi)\E[(\r_{t+1}^*)^2] + (1+\frac{1}{\phi})\E[\|\gamma\e_t\|_{\bar{\P}_{\reg}^\dagger}^2],
\end{align}
where we use the fact that $\phi a^2 + \frac{1}{\phi} b^2 \geq 2ab$ for $\phi > 0$, and we will specify $\phi$ later. According to Appendix A.3 in \cite{gower2018accelerated}, we can decompose $(\r_{t+1}^*)^2$ into three parts:
\begin{align*}
(\r_{t+1}^*)^2 = & ~ \underbrace{\left\|\beta \v_t + (1-\beta) \x_t - \x^*\right\|_{\bar{\P}_{\reg}^{\dagger}}^2}_I + \gamma^2 \underbrace{\left\|\P_{\reg,S}\left(\x_t - \x^*\right)\right\|_{\bar{\P}_{\reg}^{\dagger}}^2}_{II} \\
& ~ - 2\gamma \underbrace{\left\langle\beta\left(\v_t - \x^*\right) + (1 - \beta)\left(\x_t - \x^*\right), \bar{\P}_{\reg}^{\dagger} \P_{\reg,S}\left(\x_t - \x^*\right)\right\rangle}_{III}
\end{align*}
We upper bound the three terms separately. For the first term, following \cite{gower2018accelerated} and with the use of a parallelogram identity, we have
\begin{align*}
I = \|\beta(\v_t - \x^*) + (1-\beta) (\x_t - \x^*)\|_{\bar{\P}_{\reg}^{\dagger}}^2 \leq \beta \r_t^2 + (1-\beta) \|\x_t - \x^*\|_{\bar{\P}_{\reg}^{\dagger}}^2.
\end{align*}
For the second term, since $\nu \leq \tilde{\nu}$, we have
\begin{align*}
\E[II \mid \x_t] = & ~ \E\left[\left\|\P_{\reg,S} \left(\x_t - \x^*\right)\right\|_{\bar{\P}_{\reg}^{\dagger}}^2 \mid \x_t \right] = \left\langle\E\left[\P_{\reg,S} \bar{\P}_{\reg}^\dagger \P_{\reg,S}\right] \left(\x_t - \x^*\right), \left(\x_t - \x^*\right)\right\rangle \\
\leq & ~ \nu \cdot \langle \bar{\P}_{\reg} (\x_t - \x^*), \x_t - \x^* \rangle \leq \tilde{\nu} \cdot \|\x_t - \x^*\|_{\bar{\P}_{\reg}}^2
\end{align*}
where in the third step we use the definition of $\nu$. For the third term we have
\begin{align*}
\E[III \mid \x_t, \v_t, \y_t] = & ~ \left\langle \beta \v_t + (1-\beta) \x_t - \x^*, \bar{\P}_{\reg}^{\dagger} \bar{\P}_{\reg} \left(\x_t - \x^*\right)\right\rangle \\
= & ~ \left\langle \beta \v_t + (1-\beta) \x_t - \x^*, \x_t - \x^*\right\rangle \\
= & ~ \left\langle \x_t - \x^* + \frac{\beta(1-\alpha)}{\alpha} (\x_t - \y_t), \x_t - \x^*\right\rangle \\
= & ~ \|\x_t - \x^*\|^2 - \frac{\beta(1-\alpha)}{2\alpha} \left(\|\y_t - \x^*\|^2 - \|\x_t - \y_t\|^2 - \|\x_t - \x^*\|^2\right)
\end{align*}
where the second step follows from the assumption that $\x_t - \x^* \in \range(\bar{\P}_{\reg})$ and also the property of pesudoinverse that $\bar{\P}_{\reg}^\dagger \bar{\P}_{\reg} \w = \w$ for any $\w \in \range(\bar{\P}_{\reg})$, the third step follows from the first equation of \eqref{alg:main_analyze} and the last step follows from the parallelogram identity. Moreover, by using the fact that $\P_{\reg,S}^2 \preceq \P_{\reg,S}$ we have
\begin{align*}
\E\left[\left\|\y_{t+1} - \x^*\right\|^2 \mid \x_t\right] = & ~ \E\left[\left\|\left(\I - \P_{\reg,S}\right)\left(\x_t - \x^*\right) - \e_t\right\|^2 \mid \x_t\right] \\
\leq & ~ (1+\phi) \left\langle(\I - \bar{\P}_{\reg})\left(\x_t - \x^*\right), \x_t - \x^*\right\rangle + \left(1+\frac{1}{\phi}\right)\E[\|\e_t\|^2] \\
= & ~ (1+\phi)\left(\left\|\x_t - \x^*\right\|^2 - \left\|\x_t - \x^*\right\|_{\bar{\P}_{\reg}}^2 \right) + \left(1+\frac{1}{\phi}\right)\E[\|\e_t\|^2].
\end{align*}
By combining the above four bounds, we have the following bound for $\E[(\r_{t+1}^*)^2]$:
\begin{align}\label{eq:error_2}
& ~\E[(\r_{t+1}^*)^2] \nonumber \\
= & ~ I + \gamma^2 \E[II \mid \x_t] - 2\gamma \E[III \mid \x_t, \v_t, \y_t] \nonumber \\
\leq & ~ \beta \r_t^2 + (1-\beta) \|\x_t - \x^*\|_{\bar{\P}_{\reg}^{\dagger}}^2 + \gamma^2 \tilde{\nu} \|\x_t - \x^*\|_{\bar{\P}_{\reg}}^2 \nonumber \\
& ~ - 2\gamma\left(\|\x_t - \x^*\|^2 - \frac{\beta(1-\alpha)}{2\alpha} \left(\|\y_t - \x^*\|^2 - \|\x_t - \y_t\|^2 - \|\x_t - \x^*\|^2\right)\right) \nonumber \\
\leq & ~ \beta \r_t^2 + \frac{1-\beta}{\tilde{\mu}} \|\x_t - \x^*\|^2 \nonumber \\
& ~ + \gamma^2 \tilde{\nu} \left(\|\x_t - \x^*\| - \frac{1}{1+\phi}\E\left[\left\|\y_{t+1} - \x^*\right\|^2 \mid \x_t\right] + \frac{1+1/\phi}{1+\phi} \E[\|\e_t\|^2]\right) \nonumber \\
& ~ - 2\gamma\left(\|\x_t - \x^*\|^2 - \frac{\beta(1-\alpha)}{2\alpha} \left(\|\y_t - \x^*\|^2 - \|\x_t - \x^*\|^2\right)\right)
\end{align}
where in the last step we use the fact that $\|\bar{\P}_{\reg}^\dagger\| \leq 1/\mu$ and $\mu \geq \tilde{\mu}$.
For the bound on $\e_t = \tilde{\w}_t - \w_t$, supposing that $\|\e_t\|\leq\epsilon_1\|\x_t - \x^*\|$, we have
\begin{align}
\E[\|\e_t\|^2] \leq \epsilon_1^2 \cdot \E[\|\x_t - \x^*\|^2] = \epsilon_1^2 \cdot \|\x_t - \x^*\|^2\label{eq:error_3}
\end{align}
and also
\begin{align}
\E[\|\e_t\|_{\bar{\P}_{\reg}^\dagger}^2] \leq \|\bar{\P}_{\reg}^\dagger\| \cdot \E[\|\e_t\|^2] \leq \frac{\epsilon_1^2}{\tilde{\mu}} \cdot \|\x_t - \x^*\|^2.
\label{eq:error_4}
\end{align}

Finally, by combining \eqref{eq:error_1}, \eqref{eq:error_2}, \eqref{eq:error_3} and \eqref{eq:error_4} we have 
\begin{align*}
& ~ \E[\r_{t+1}^2 + \gamma^2 \tilde{\nu} \|\y_{t+1} - \x^*\|^2] \\
\leq & ~ (1+\phi)\E[(\r_{t+1}^*)^2] + (1+\frac{1}{\phi})\E[\|\gamma\e_t\|_{\bar{\P}_{\reg}^\dagger}^2] + \gamma^2 \tilde{\nu} \E[\|\y_{t+1} - \x^*\|^2] \\
\leq & ~ (1+\phi)\beta \r_t^2 + \beta(1+\phi)\underbrace{\frac{\gamma(1-\alpha)}{\alpha}}_{P_1} \|\y_t - \x^*\|^2 \\
& ~ + (1+\phi) \underbrace{\left(\frac{1-\beta}{\tilde{\mu}} + \gamma^2 \tilde{\nu} - 2\gamma - \frac{\gamma\beta(1-\alpha)}{\alpha} + \frac{\gamma^2 \epsilon_1^2 }{\phi}(\tilde{\nu} + \frac{1}{\tilde{\mu}})\right)}_{P_2}\|\x_t - \x^*\|^2.
\end{align*}
By choosing $\alpha = \frac{1}{1+\gamma \tilde{\nu}}$ and $\gamma = \frac{1}{\sqrt{\tilde{\nu}\tilde{\mu}}}$ we have $P_1 = \gamma^2 \tilde{\nu} = \frac{1}{\tilde{\mu}}$. By choosing $\beta = 1 - \sqrt{\frac{\tilde{\mu}}{\tilde{\nu}}}$ we have $\beta(1+\phi) \leq 1 - \frac{1}{2}\sqrt{\frac{\tilde{\mu}}{\tilde{\nu}}}$ holds for $\phi = \frac{\sqrt{\tilde{\mu} / \tilde{\nu}}}{2(1 - \sqrt{\tilde{\mu} / \tilde{\nu}})}$. By further setting $\epsilon_1 \leq \frac{\tilde{\mu}}{4} \leq \frac{\tilde{\mu}}{\sqrt{8(\tilde{\mu}\tilde{\nu}+1)}}$ we also achieve $P_2 \leq 0$. Finally, denote $\Delta_t = \|\v_t-\x^*\|_{\bar{\P}_{\reg}^\dagger}^2 + \frac{1}{\tilde\mu}\|\y_t-\x^*\|^2 = \r_t + \frac{1}{\tilde\mu}\|\y_t-\x^*\|^2$, we conclude that
\begin{align*}
\E[\Delta_{t+1}] \leq \left(1 - \frac{1}{2} \sqrt{\frac{\tilde{\mu}}{\tilde{\nu}}}\right) \cdot \E\left[\r_t^2 + \frac{1}{\tilde{\mu}} \|\y_t - \x^*\|^2\right] = \left(1 - \frac{1}{2} \sqrt{\frac{\tilde{\mu}}{\tilde{\nu}}}\right) \cdot \E[\Delta_t].
\end{align*}
\end{proof}
\subsection{Proof of Lemma~\ref{lem:equivalence}}
\label{s:appendix_proof_equivalence}
\begin{proof}[Proof of Lemma~\ref{lem:equivalence}]
Recall that $1\leq \nu\leq 1/\mu$, which implies $\bar\rho=\sqrt{\mu/\nu}\leq 1/\nu$. Moveover, $\frac{c}{\nu} \le \eta \le \frac{1}{\nu} \le 1$ and $\rho \leq c\bar\rho\leq
  c/\nu\leq \eta$.

(a) Then, 
\begin{align*}
\nu \le \frac{1}{2 \eta} \le \frac{1}{\rho + \eta(1 - \rho)} = \tilde\nu
\end{align*}
and
\begin{align*}\tilde\mu = \frac{c^2\bar\rho^2}{\rho + \eta(1 - \rho)} \le \frac{c^2\bar\rho^2}{\eta} = \frac{c}{\eta} c\bar\rho^2 \le \nu c\bar\rho^2 = \mu c \le \mu.
\end{align*}

(b) Note that the choice of $\tilde\mu$ and $\tilde\nu$ and the statement of Lemma~\ref{lem:converge_accelerate_form2} implies
\begin{align}\label{eq:para_rho}
\beta = 1-\rho, \quad \gamma = 1 + \frac{\eta}{\rho} - \eta, \quad\text{and}\quad \alpha = \frac{\rho}{1+\rho}.    
\end{align}
Let $\m_t := \frac{\beta(1-\alpha)}{\gamma - 1}(\v_t-\y_t)$ for $t \ge 0$. Let's check that then it satisfies the recurrence relation for $\m_t$. Indeed, from the first equation of \eqref{alg:main_analyze}, 
$$\x_t-\v_t = (1-\alpha)(\y_t-\v_t),$$ and then, from the last two equations of \eqref{alg:main_analyze},
\begin{align*}
\v_{t+1}-\y_{t+1}
    &=(\beta\v_t+(1-\beta)\x_t - \gamma \tilde\w_t) - (\x_t-\tilde\w_t)
    \\
    &=\beta(\v_t-\x_t) - (\gamma-1)\tilde\w_t\\
    &=\beta(1-\alpha)(\v_t-\y_t) - (\gamma-1)\tilde\w_t\\
    &=(\gamma-1)(\m_t - \tilde\w_t).
\end{align*}
By combining the above result with \eqref{eq:para_rho}, we have
\begin{align*}
\m_{t+1} = \frac{\beta(1-\alpha)}{\gamma-1}(\v_{t+1} - \y_{t+1}) = \beta(1-\alpha)(\m_t - \tilde\w_t) = \frac{1-\rho}{1+\rho}(\m_t - \tilde\w_t).
\end{align*}
For the update of $\x_{t+1}$, from \eqref{alg:main_analyze}, the definition of $\m_t$, and \eqref{eq:para_rho}, we have
\begin{align*}
\x_{t+1} = \y_{t+1} + \alpha(\v_{t+1}-\y_{t+1}) = (\x_t - \tilde\w_t) + \frac{\alpha(\gamma-1)}{\beta(1-\alpha)}\m_{t+1} = \x_t - \tilde\w_t + \eta \m_{t+1}.
\end{align*}
This concludes the proof of Lemma~\ref{lem:equivalence}.
\end{proof}

\section{Regularized DPPs: Proof of Lemma~\ref{l:mu-reg-exact}}
\label{s:appendix_proof_mu}
\begin{proof}[Proof of Lemma~\ref{l:mu-reg-exact}]
To bound the term $\E[(\I + \frac{m}{k\bar{\lambda}}\A_{\setS_{\DPP}}^\top\A_{\setS_{\DPP}} )^{-1}]$ for $\setS_{\DPP} \sim \DPP(\frac{m}{\bar{\lambda}(m-k)}\A\A^\top + \frac{k}{m-k}\I)$, we first introduce the following notion of Regularized DPP (R-DPP). We then show that our DPP sample is equivalent to an R-DPP (in distribution) in Lemma~\ref{lem:equivalence_dpp}.
\begin{definition}[Regularized DPP, Definition 2 of \cite{derezinski2020bayesian}]
Given matrix $\A\in\R^{m \times n}$, let $\a_i^\top$ denote its $i$th row. For a sequence $p = (p_1, \ldots, p_m)\in[0,1]^m$ and $\lambda>0$, define $\RDPP_p(\A,\lambda)$ as a distribution over $S\subseteq [m]$ such that
\begin{align*}
\Pr\{\setS\} = \frac{\det(\A_{\setS}^\top \A_{\setS} + \lambda\I)}{\det(\sum_i p_i \a_i\a_i^\top + \lambda\I)} \cdot \prod_{i\in S} p_i \cdot \prod_{i\notin S} (1-p_i).
\end{align*}
\end{definition}
\begin{lemma}[Lemma 7 in \cite{derezinski2020bayesian}]\label{lem:equivalence_dpp}
Given $\A\in\R^{m \times n}, \lambda > 0$ and $p \in [0,1)^m$, denote $\D_p \coloneqq \diag(p)$ and $\tilde{\D} \coloneqq \D_p (\I - \D_p)^{-1}$, then we have
\begin{align*}
\RDPP_p(\A, \lambda) = \DPP(\tilde{\D} + \lambda^{-1}\tilde{\D}^{1/2}\A\A^\top\tilde{\D}^{1/2})
\end{align*}
which means that the R-DPP is equivalent to a DPP from Definition \ref{d:dpp}. 
\end{lemma}
According to Lemma~\ref{lem:equivalence_dpp}, by choosing $p = (\frac{k}{m}, \ldots, \frac{k}{m})$ we have $\DPP(\frac{m}{\bar{\lambda}(m-k)}\A\A^\top + \frac{k}{m-k}\I) = \RDPP_p(\A, \frac{k}{m}\bar{\lambda})$, which shows that this DPP can be equivalently viewed as an R-DPP. For an R-DPP we have the following bound.

\begin{lemma}[Lemma 11 in \cite{derezinski2020bayesian}]\label{lem:dlm11}
For $\setS \sim \RDPP_p(\A,\lambda)$ and $p \in [0,1)^m$, 
\begin{align*}
\E\left[ \left(\A_{\setS}^\top\A_{\setS} + \lambda\I\right)^{-1}\right] \preceq \left(\sum_i p_i \a_i\a_i^\top + \lambda\I \right)^{-1}.
\end{align*}
\end{lemma}
By applying Lemma~\ref{lem:dlm11} to $\setS_{\DPP} \sim \RDPP_p(\A, \frac{k}{m}\bar{\lambda})$ where $p = (\frac{k}{m}, \ldots, \frac{k}{m})$ we have 
\begin{align*}
\E\left[\left(\I + \frac{m}{k\bar{\lambda}}\A_{\setS_{\DPP}}^\top\A_{\setS_{\DPP}} \right)^{-1}\right] \preceq \frac{k\bar{\lambda}}{m} \left(\sum_i p_i \a_i\a_i^\top + \frac{k\bar{\lambda}}{m}\I \right)^{-1} = \bar{\lambda} \left(\A^\top \A + \bar{\lambda}\I\right)^{-1}.
\end{align*}
\end{proof}

\section{Over-determined Systems: Proof of Lemma~\ref{l:nu_misc}}
We first state two lemmas which will be used in the proof. Here, the first lemma is adapted from Theorem 7.2.1 in \cite{tropp2015matrix}.

\begin{lemma}[Matrix Chernoff]
\label{l:chernoff}
Let $\Z_1, \ldots, \Z_s$ be independent random $n \times n$ psd matrices  with $\mu_{\max} \coloneqq \lambda_{\max}(\sum_t \E[\Z_t])$. Suppose that $\max_t \lambda_{\max}(\Z_t) \leq R$. Then, for any $\epsilon\geq 0$ we have
\begin{align*}
\Pr\left\{\lambda_{\max} \left(\sum_{t=1}^s \Z_t\right) \geq (1+\epsilon) \mu_{\max} \right\} \leq n \cdot \exp\left(- \frac{\epsilon^2 \mu_{\max}}{(2+\epsilon) R}\right).
\end{align*}
\end{lemma}

\begin{lemma}[Lemma 6 in \cite{derezinski2024fine}, Lemma 3.3 in \cite{tropp2011improved}]\label{l:ls_uniform}
Suppose matrix $\U\in\R^{m \times r}$ such that $\U^\top\U = \I$ is transformed by RHT. Then, $\tilde{\U}=\Q\U$ with probability $1-\delta$ has nearly uniform leverage scores, i.e., the norm of the rows satisfies $\|\tilde{\u}_i\| \leq \sqrt{r/m} + \sqrt{8 \log(m/\delta) /m}$ for all $i \in [m]$.
\end{lemma}
\begin{proof}[Proof of Lemma~\ref{l:nu_misc}]
For any $t \in [s]$, denote $\Z_t \coloneqq \frac{1}{p_{i_t}} (\A^\top\A)^{\dagger/2} \a_{i_t} \a_{i_t}^\top (\A^\top\A)^{\dagger/2}$, where $p_{i_t} = \frac{1}{m}$ is the uniform sampling probability, and $\a_{i_t}$ is the $i_t$-th row of $\A$. Then, we have
\begin{align*}
\left\|\frac{1}{m}\sum_{t=1}^s \Z_t\right\| = \left\|(\A^\top\A)^{\dagger/2} \A_S^\top \A_S (\A^\top\A)^{\dagger/2}\right\| =
\left\|\A_S(\A^\top\A)^\dagger \A_S^\top \right\|.
\end{align*}
Let $\mu_{\max} \coloneqq \lambda_{\max}(\sum_t \E[\Z_t])$.
Notice that for each $t \in [s]$, since $\E[\Z_t] = (\A^\top\A)^{\dagger/2} \cdot \E[\frac{\a_{i_t} \a_{i_t}^\top}{p_{i_t}}] \cdot (\A^\top \A)^{\dagger/2} = (\A^\top\A)^{\dagger/2} (\A^\top\A) (\A^\top \A)^{\dagger/2}$, we have $\mu_{\max} = s\cdot \lambda_{\max}(\E[\Z_t]) = s$.
By applying matrix Chernoff bound (Lemma~\ref{l:chernoff}) to $\{\Z_t\}_{t=1}^s$ we have
\begin{align*}
\Pr\left\{ \left\|\sum_{t=1}^s \Z_t\right\| \geq (1+\epsilon)s \right\} \leq n \cdot \exp\left(-\frac{\epsilon^2 s}{(2+\epsilon)R} \right),
\end{align*}
where, $R>0$ is a parameter that satisfies
\begin{align*}
\max_t \lambda_{\max}(\Z_t) \leq \max_i \frac{1}{p_i} \left\|(\A^\top\A)^{\dagger/2} \a_i\a_i^\top (\A^\top\A)^{\dagger/2} \right\| = m \cdot \max_i \|\a_i\|_{(\A^\top\A)^\dagger}^2 \leq R.
\end{align*}
Let $\A = \U \Sig\V^\top$ be the thin SVD of $\A$ where $\U \in\R^{m \times r}$ satisfies $\U^\top\U = \I$, then we have
\begin{align*}
\|\a_i\|_{(\A^\top\A)^\dagger}^2 = (\u_i^\top\Sig\V^\top) (\V\Sig^{-2}\V^\top) (\V\Sig\u_i) = \u_i^\top \u_i = \|\u_i\|^2
\end{align*}
where $\u_i^\top$ is the $i$-th row of $\U$. Since we assume that $\A$ is transformed by RHT, so is matrix $\U$. According to Lemma~\ref{l:ls_uniform} we have 
\begin{align*}
\max_i \|\u_i\|^2 \leq \left(\sqrt{\frac{r}{m}} + \sqrt{\frac{8\log(m/\delta)}{m}}\right)^2 \leq \frac{2r}{m} + \frac{16 \log(m/\delta)}{m}
\end{align*}
holds with probability $1-\delta$. Conditioned on this event, we let $R = 2r+ 16\log(m/\delta) \geq m \cdot \max_i \|\u_i\|^2$. Given $0<\delta'<1$, by choosing $\epsilon = \frac{2R}{s}\log(n/\delta')-1 > \frac{4r}{s} -1 > 3$ we have
\begin{align*}
& ~ \Pr\left\{ \left\|\A_S(\A^\top\A)^\dagger \A_S^\top \right\| \geq \frac{4r + 32 \log(m/\delta)}{m} \cdot \log(n/\delta')\right\} \\
= & ~ \Pr\left\{ \left\|\sum_{t=1}^s \Z_t\right\| \geq \left(4r + 32 \log(m/\delta) \right) \cdot \log(n/\delta')\right\} \\
\leq & ~ n \cdot \exp\left(-\frac{\epsilon^2 s}{(2+\epsilon)R} \right) \leq n \cdot \exp\left(-\frac{3\epsilon s}{5R} \right) \leq n \cdot \exp\left(-\frac{5R \log(n/\delta')}{5R} \right) = \delta'.
\end{align*}
Thus we conclude that conditioned on an event that happens with probability $1-\delta$ (and only depends on RHT), with probability $1-\delta'$ we have $\left\|\A_S(\A^\top\A)^\dagger \A_S^\top \right\| \leq \frac{4r + 32 \log(m/\delta)}{m} \cdot \log(n/\delta') = \alpha$. Finally, this implies that $\|(\A^\top\A)^{\dagger/2} \A_S^\top\A_S (\A^\top\A)^{\dagger/2}\| \leq \alpha$, which gives that
\begin{align*}
\A_S^\top\A_S \preceq \alpha \A^\top\A = \frac{4r + 32 \log(m/\delta)}{m} \log(n/\delta') \cdot \A^\top\A.
\end{align*}
\end{proof}

\section{Block Memoization: Proofs of Theorem \ref{thm:block_memo} and Corollary~\ref{cor:block_memo}}\label{s:block_memo-proof}
To prove Theorem \ref{thm:block_memo}, we rely on the following Bernstein's inequality for the concentration of symmetric random matrices, adapted from \cite{tropp2015matrix}.

\begin{lemma}[Matrix Bernstein]\label{l:bernstein}
Let $\Z_1, \ldots, \Z_B$ be independent random symmetric $n \times n$ matrices such that $\frac{1}{B}\sum_j \E[\Z_j] = \bar{\Z}$ and $\|\frac{1}{B} \sum_j \E[(\Z_j - \bar{\Z})^2]\| \leq \sigma^2$. Suppose that $\|\Z_j - \E[\Z_j]\| \leq R$ holds for all $j \in [B]$. Then for any $\epsilon \geq 0$,
\begin{align*}
\Pr\bigg(\big\|\frac{1}{B}\sum_{j=1}^B \Z_j - \bar{\Z}\big\| \geq \epsilon \bigg) \leq 2n \cdot \exp\left(-\frac{\epsilon^2 B / 2}{\sigma^2 + \epsilon R / 3}\right).
\end{align*}
\end{lemma}
\begin{proof}[Proof of Theorem~\ref{thm:block_memo}]
Let $\Z_j \coloneqq \bar{\P}_{\reg}^{\dagger/2}\P_{\reg, S_j}\bar{\P}_{\reg}^{\dagger/2}$ be a normalized form of the regularized projection matrix $\P_{\reg, S_j}$ associated with subset $S_j$ from $\mathcal{B}$. Then we have $\bar{\Z} \coloneqq \E_{S_j \sim \mathcal{D}}[\Z_j] =  \bar{\P}_{\reg}^{\dagger/2}\bar{\P}_{\reg} \bar{\P}_{\reg}^{\dagger/2}$. Denote $\mathcal{E}_j$ as the event that $\A_{S_j}^\top \A_{S_j} \preceq \alpha \A^\top\A$, and let $\widetilde{\Z}_j \coloneqq \Z_j \cdot \mathbf{1}_{\mathcal{E}_j}$. Then, $\{\widetilde{\Z}_j\}_{j=1}^B$ are a series of independent PSD matrices. We have\footnote{For convenience we shorten the subscript $S_j \sim \mathcal{D}$ when taking expectation.}
\begin{align}\label{eq:tilde_Z}
\E[\widetilde{\Z}_j] 
= (1-\delta')\cdot \E[\Z_j \mid \mathcal{E}_j] \preceq \E[\Z_j]
\end{align}
where $\delta' = \Pr(\mathcal{E}_j)$ and in the last step we use the fact that $\Z_j \succeq \mathbf{0}$. In order to apply matrix Bernstein, we need to bound two terms $\|\widetilde{\Z}_j - \E[\widetilde{\Z}_j]\|$ and $\|\E[(\widetilde{\Z}_j - \E[\widetilde{\Z}_j])^2]\|$. For the first term, notice that according to our assumption we have $\bar\P_{\reg}^\dagger \preceq \frac{1}{c}(\A^\top\A + \bar{\lambda}\I)(\A^\top\A)^\dagger \preceq \frac{1}{c}(\I+\bar{\lambda}(\A^\top\A)^\dagger)$, which gives
\begin{align*}
\left\|\widetilde{\Z}_j \right\| = & ~ \left\|\bar\P_{\reg}^{\dagger/2} \P_{\reg, S_j} \bar\P_{\reg}^{\dagger/2} \cdot \mathbf{1}_{\mathcal{E}_j} \right\| = \left\|\P_{\reg, S_j}^{1/2} \bar\P_{\reg}^\dagger \P_{\reg, S_j}^{1/2} \cdot \mathbf{1}_{\mathcal{E}_j}\right\| \\
\leq & ~ \frac{1}{c}\left\|\P_{\reg, S_j}^{1/2}\left(\I + \bar\lambda(\A^\top\A)^\dagger\right)\P_{\reg, S_j}^{1/2} \cdot \mathbf{1}_{\mathcal{E}_j}\right\| \\
\leq & ~ \frac{1}{c}\left\|\P_{\reg, S_j}\right\| + \frac{\bar\lambda}{c} \left\|\P_{\reg, S_j}^{1/2}(\A^\top\A)^\dagger\P_{\reg, S_j}^{1/2} \cdot \mathbf{1}_{\mathcal{E}_j}\right\| \\
\leq & ~ \frac{1}{c} + \frac{\bar\lambda}{c} \left\|(\A^\top\A)^{\dagger/2}\P_{\reg, S_j}(\A^\top\A)^{\dagger/2} \cdot \mathbf{1}_{\mathcal{E}_j}\right\|.
\end{align*}
Notice that we can expand the last term as follows:
\begin{align*}
& ~ \left\|(\A^\top\A)^{\dagger/2}\P_{\reg, S_j}(\A^\top\A)^{\dagger/2} \cdot \mathbf{1}_{\mathcal{E}_j}\right\| \\
= & ~ \left\|(\A^\top\A)^{\dagger/2}\A_S^\top(\A_S\A_S^\top+\lambda\I)^{-1}\A_S(\A^\top\A)^{\dagger/2} \cdot \mathbf{1}_{\mathcal{E}_j}\right\| \\
= & ~ \left\|(\A_S\A_S^\top+\lambda\I)^{-1/2}\A_S(\A^\top\A)^\dagger\A_S^\top(\A_S\A_S^\top+\lambda\I)^{-1/2} \cdot \mathbf{1}_{\mathcal{E}_j}\right\|.
\end{align*}
If $\mathcal{E}_j$ happens, then by definition $\|\A_S(\A^\top\A)^\dagger \A_S^\top\| = \|(\A^\top\A)^{\dagger/2} \A_S^\top\A_S (\A^\top\A)^{\dagger/2}\| \leq \alpha$, and the above quantity can be bounded by $\alpha \left\|(\A_S\A_S^\top+\lambda\I)^{-1}\right\| \leq \frac{\alpha}{\lambda}$; otherwise it equals to $0$. This gives $\|\widetilde{\Z}_j\| \leq \frac{1}{c}+\frac{\alpha\bar\lambda}{c\lambda}$. We can similarly bound $\|\Z_j\|$ (which will be used later), since the only difference is the $\mathbf{1}_{\mathcal{E}_j}$ term: by using $\|\A_S(\A^\top\A)^\dagger \A_S^\top\| \leq 1$, we have $\|\Z_j\| \leq \frac{1}{c} + \frac{\bar\lambda}{c \lambda}$. Combining this with \eqref{eq:tilde_Z} we obtain the $R$ term in Bernstein's inequality:
\begin{align*}
\big\|\widetilde{\Z}_j - \E[\widetilde{\Z}_j] \big\| \leq & ~ \big\|\widetilde{\Z}_j\big\| + \big\|\E[\widetilde{\Z}_j]\big\| 
\leq \big\|\widetilde{\Z}_j\big\| + \big\|\bar{\Z}\big\| \\
= & ~ \big\|\widetilde{\Z}_j\big\| + \big\|\bar{\P}_{\reg}^{\dagger/2}\bar{\P}_{\reg} \bar{\P}_{\reg}^{\dagger/2}\big\| \leq \frac{1}{c}\big(1+\alpha\frac{\bar\lambda}{\lambda}\big)+1 =: R.
\end{align*}
To obtain the $\sigma^2$ term, since $\E[\widetilde{\Z}_j] \preceq \E[\Z_j] \preceq \I$, observe that:
\begin{align*}
& ~ \big\|\E\big[(\widetilde{\Z}_j - \E[\widetilde{\Z}_j])^2 \big] \big\| = \big\|\E\big[\widetilde{\Z}_j^2\big] - \E[\widetilde{\Z}_j]^2 \big\| \leq \big\|\E\big[\widetilde{\Z}_j^2\big]\big\| + \big\|\E[\widetilde{\Z}_j]^2 \big\| \leq \big\|\E\big[\widetilde{\Z}_j^2\big] \big\| + 1 \\
& ~ \leq \big\|\E[(\bar{\P}_{\reg}^{\dagger/2} \P_{\reg, S_j} \bar{\P}_{\reg}^{\dagger/2})^2] \big\| + 1 \leq  \frac{1}{c}\Big(1 + \frac{\bar\lambda}{\lambda}\big(\alpha + \delta'\|\bar\P_\lambda^\dagger\|\big)\Big) + 1 =: \sigma^2
\end{align*}
where the last step follows from Theorem~\ref{l:nu-reg}.
Applying Bernstein (Lemma~\ref{l:bernstein}) to matrices $\{\widetilde{\Z}_j\}_{j=1}^B$ with parameters $\sigma^2$ and $R$ computed above, by setting $\epsilon=\frac{1}{3}$ and $B \geq \frac{20}{c}(c+1 + \frac{\bar\lambda}{\lambda}(\alpha+\frac{9}{10}\delta' \|\bar{\P}_{\lambda}^\dagger\|))\cdot\log(2n/\delta)$ we have
\begin{align*}
\Pr\bigg(\Big\|\frac{1}{B}\sum_{j=1}^B \widetilde{\Z}_j - \E[\widetilde{\Z}_j]\Big\| \geq \frac{1}{3} \bigg) \leq 2n \cdot \exp\bigg(-\frac{B / 18}{\sigma^2 + R/9}\bigg) \leq \delta.
\end{align*}
Taking a union bound, with probability $1-B\delta' - \delta$ we have the following:
\begin{align*}
\bigg\|\frac{1}{B}\sum_{j=1}^B \Z_j - \bar{\Z} \bigg\| \leq & ~ \bigg\|\frac{1}{B}\sum_{j=1}^B \big(\Z_j - \widetilde{\Z}_j\big) \bigg\| + \bigg\|\frac{1}{B}\sum_{j=1}^B \widetilde{\Z}_j - \E[\widetilde{\Z}_j] \bigg\| + \bigg\|\E[\widetilde{\Z}_j] - \bar{\Z} \bigg\| \\
\leq & ~ 0 + \frac{1}{3} + \delta' \cdot \left\| \Z_j\right\| \leq \frac{1}{3} + \delta' \cdot\left(\frac{1}{c}+\frac{\bar\lambda}{c\lambda}\right) \leq \frac{1}{2}
\end{align*}
where the last step follows by taking $\delta' \leq \frac{c\lambda}{12\bar\lambda}$. Notice that since $0 < c \leq 1$, by further assuming that $\delta' \leq \alpha / \|\bar{\P}_{\lambda}^\dagger\|$, the requirement on $B$ becomes $B \geq \frac{40}{c}(1 + \alpha\frac{\bar\lambda}{\lambda})\cdot\log(2n/\delta)$.
So, the average $\hat\Z :=\frac1B\sum_{j=1}^B\Z_j$ satisfies $\hat\Z\succeq \bar\Z-\frac12\I$ with probability $1-B\delta'-\delta$. Note that $\bar\Z$ is a projection onto a subspace that contains the range of $\hat\Z$, applying $\bar\Z$ on both sides of that inequality we get $\hat\Z = \bar\Z\hat\Z\bar\Z\succeq \bar\Z\bar\Z\bar\Z - \frac12\bar\Z\bar\Z = \frac12\bar\Z$. This gives
\begin{align*}
\frac{1}{B}\sum_{j=1}^B \P_{\reg, S_j} \succeq \bar{\P}_{\reg}^{1/2}\hat\Z\bar{\P}_{\reg}^{1/2} \succeq 
\frac12\bar{\P}_{\reg}^{1/2}\bar\Z\bar{\P}_{\reg}^{1/2}
=\frac{1}{2} \bar{\P}_{\reg},
\end{align*}
which concludes the proof.
\end{proof}

\begin{proof}[Proof of Corollary~\ref{cor:block_memo}]
By Theorem~\ref{l:mu-reg}, under these assumptions we have
\begin{align*}
\|\bar{\P}_{\lambda}^\dagger\| = \frac{1}{\lambda_{\min}(\bar\P_{\lambda})} \leq \frac{2r \bar{\kappa}_k^2}{k}.
\end{align*}
Thus according to Lemma~\ref{l:nu_misc} with choice $\delta' = \frac{k}{2m\bar{\kappa}_k^2} \leq \frac{r}{m\|\bar{\P}_{\lambda}^\dagger\|}$, we have
\begin{align*}
\alpha \leq \frac{c'r \log(n/\delta')}{2m} \leq \frac{c'r \log(mn\bar{\kappa}_k)}{m}
\end{align*}
where $c' = 8+\frac{64}{C} \leq 9$ from the assumption that $C\log(m/\delta)< r$ for some $C\geq 64$. We can also easily ensure that $\alpha\leq 1$.
By plugging these bounds in Theorem~\ref{thm:block_memo} and noticing that $c=\frac{1}{2}$, we obtain the desired requirement on $B$.
\end{proof}

\section{Analysis of SymFHT: Proof of Theorem \ref{t:symFHT}}
\label{s:symfht}
Before we analyze the proposed SymFHT algorithm, we first define the Hadamard matrix and state the standard FHT (Fast Hadamard Transform) algorithm as follows.
\begin{definition}[Hadamard matrix]
For $n = 2^m$, we define Hadamard matrix as: $\H_1 = 1$, and for $m\geq 1$, 
\begin{align*}
\H_n = \begin{bmatrix}
\H_{n/2} & \H_{n/2} \\
\H_{n/2} & -\H_{n/2}
\end{bmatrix} \in \R^{n \times n}.
\end{align*}
\end{definition}

\begin{algorithm}[!ht]
\caption{Fast Hadamard Transform (FHT)}
\label{alg:fht}
\begin{algorithmic}[1]
\Function{FHT}{$\A$} \Comment{Input: $n \times d$ matrix $\A=[\a_1, \ldots, \a_d]$.}
\For{$i = 1, 2, \ldots, d$}
\State Compute $\x_i \leftarrow \text{FHT}(\a_i[n/2:n] + \a_i[1:n/2])$ \Comment{Recursive call.}
\State Compute $\y_i \leftarrow \text{FHT}(\a_i[n/2:n] - \a_i[1:n/2])$ \Comment{Recursive call.}
\State Compute $\tilde{\a}_i \leftarrow [\x_i^\top, \y_i^\top]^\top$
\EndFor
\State \textbf{return} $[\tilde{\a}_1, \ldots, \tilde{\a}_d]$ \Comment{Computes $\H_n \A$.}
\EndFunction
\end{algorithmic}
\end{algorithm}

Denote $\mathcal{T}_{\text{FHT}}(n,d)$ as the FLOPs it takes to compute $\text{FHT}(\A)$ for $\A\in\R^{n \times d}$. Notice that if we set $d = 1$ in Algorithm~\ref{alg:fht}, then it recovers the vector version of FHT, in this case by recursion we have $\mathcal{T}_{\text{FHT}}(n, 1) = 2 \mathcal{T}_{\text{FHT}}(n/2, 1) + n = n \log n$. For standard matrix cases, we have $\mathcal{T}_{\text{FHT}}(n, d) = d \cdot \mathcal{T}_{\text{FHT}}(n, 1) = nd \log n$.

In \algpd\ (Algorithm~\ref{alg:bcd}), we need to compute $\Q \A \Q^\top$ for a symmetric matrix $\A$, where $\Q = \H_n\D$, for $\D = \frac1{\sqrt n}\diag(d_1,...,d_n)$ and $d_i$ are independent random $\pm 1$ signs (Rademacher variables). In order to construct this transformation, we can naively apply FHT twice to matrix $\D\A\D$ and have $\Q\A\Q^\top = \text{FHT}(\text{FHT}(\D\A\D)^\top)$. Notice that the cost of applying FHT to an $n \times n$ matrix is $n^2 \log n$, thus this naive method takes $2n^2 \log n$ FLOPs. However, this method does not take the symmetry of $\A$ into account. We improve on this with our proposed SymFHT (Algorithm~\ref{alg:symfht}).
\begin{proof}[Proof of Theorem~\ref{t:symFHT}]
Denote $\H \coloneqq \H_{n/2}$. Then, for symmetric matrix $\A$ we have the following:
\begin{align*}
\H_n \A \H_n = 
\begin{bmatrix}
\H & \H \\
\H & -\H
\end{bmatrix}
\begin{bmatrix}
\A_{11} & \A_{12} \\
\A_{12}^\top & \A_{22}
\end{bmatrix}
\begin{bmatrix}
\H & \H \\
\H & -\H
\end{bmatrix} = 
\begin{bmatrix}
\C_{11} + \C_{21} & \C_{12} + \C_{22} \\
\C_{12}^\top + \C_{22}^\top & \C_{12} - \C_{22}
\end{bmatrix}
\end{align*}
where matrices $\C_{11}, \C_{12}$ and $\C_{22}$ are given by 
\begin{align*}
\C_{11} = & ~ \H\big(\A_{11} + \A_{12}^\top \big)\H \quad\quad
\C_{12} = \H\big(\A_{11} - \A_{12} \big)\H\\
\C_{21} = & ~ \H\big(\A_{12} + \A_{22} \big)\H \quad\quad
\C_{22} = \H\big(\A_{12}^\top - \A_{22} \big)\H
\end{align*}
Thus if we pre-compute $\B_{11} \coloneqq \H\A_{11}\H$, $\B_{12} \coloneqq \H \A_{12}\H$, $\B_{22} = \H\A_{22}\H$, then
\begin{align*}
\C_{11} = & ~ \B_{11} + \B_{12}^\top \quad\quad
\C_{12} = \B_{11} - \B_{12} \\
\C_{21} = & ~ \B_{12} + \B_{22} \quad\quad
\C_{22} = \B_{12}^\top - \B_{22}
\end{align*}
We note that due to symmetry, we do not need to compute $\H\A_{21}\H = \H\A_{12}^\top \H$. When we compute $\B_{11}$ and $\B_{22}$, since both $\A_{11}$ and $\A_{22}$ are also symmetric, we can compute them recursively, i.e., $\B_{11} = \text{SymFHT}(\A_{11})$, $\B_{22} = \text{SymFHT}(\A_{22})$. However when we compute $\B_{12}$, since $\A_{12}$ is no longer symmetric, we can no longer use the same scheme; instead, we can apply standard FHT twice to this smaller matrix, i.e., $\B_{12} = \text{FHT}(\text{FHT}(\A_{12}^\top)^\top)$.
Notice that the costs for computing $\C_{11}, \C_{12}, \C_{21}$ and $\C_{22}$ are all $(n/2)^2 = n^2 / 4$. In addition, we need to compute $\C_{11} + \C_{21}, \C_{12} +\C_{22}$ and $\C_{12} - \C_{22}$. These sum up to $7n^2 / 4$ FLOPs. Thus, the cost of $\text{SymFHT}$ is governed by the following recursive inequality:
\begin{align*}
\mathcal{T}_{\text{SymFHT}}(n) \leq & ~ 2 \mathcal{T}_{\text{SymFHT}}(n/2) + 2 \mathcal{T}_{\text{FHT}}(n/2, n/2) + 7n^2 / 4 \\
\leq & ~ 2 \mathcal{T}_{\text{SymFHT}}(n/2) + 2(n/2)^2 \log(n/2) + 7n^2 / 4\\
= & ~ 2 \mathcal{T}_{\text{SymFHT}}(n/2) + 2(n/2)^2 (\log n + 2.5) \\
\leq & ~ \frac{1}{2} n^2 (2.5 + \log n) \cdot \left(1 + \frac{1}{2} + \frac{1}{2^2} + \cdots\right) \leq n^2 (2.5 + \log n)
\end{align*}
where we use the fact that $\mathcal{T}_{\text{FHT}}(n, n) = n^2 \log n$ and that $n \geq 4$. Compared to the na\"ive method, our algorithm $\text{SymFHT}$ reduces the FLOPs by about a half.
\end{proof}

\section{Further Numerical Experiments}
\label{s:further-experiments}
In this section we provide the details for our experimental setup, and we give the results for the test matrices and experiments not included in Section \ref{s:experiments}. We also carry out additional experiments evaluating the effect of Tikhonov regularization on the Kaczmarz projection steps. The code for our experiments is available at \url{https://github.com/EdwinYang7/kaczmarz-plusplus}.

First, we discuss the specifics of the construction of our test matrices. We consider two classes of test matrices. 

\paragraph{\textbf{1. Synthetic Low-Rank Matrices.}}
To validate our theories on the effect of the number of large outlying eigenvalues, we carry out experiments on synthetic benchmark matrices using the function \verb~make_low_rank_matrix~ from Scikit-learn \cite{scikit-learn}, which provides random matrices with a bell-shaped spectrum, motivated by data in computer vision and natural language processing. As specified in Scikit-learn, the singular value profile of a matrix $\A\in\R^{n_{\text{samples}}\times n_{\text{features}}}$ generated this way is: $(1-\verb~tail_strength~)\cdot\exp(-(i/\verb~effective_rank~)^2)$ for the top \verb~effective_rank~ singular values, and $\verb~tail_strength~\cdot\exp(i/\verb~effective_rank~)$ for the remaining ones.
We  choose parameter \verb~effective_rank~ among the four values $\{25, 50, 100, 200\}$. Note that parameter \verb~effective_rank~ is approximately the number of singular vectors required to explain most of the data by linear combinations. It can be viewed as ``the number of large singular values $k$'' in our theory, and is also roughly the number of steps needed for Krylov-type methods to construct a good Krylov subspace. Parameter \verb~tail_strength~, which captures the relative importance of the fat noisy tail of the spectrum, is set to $0.01$. 
 
For the task of testing \alg (Algorithm~\ref{alg:kzpp}), each of the test matrices is an $m \times n$ rectangular matrix $\A$ with $m=4096$ and $n=1024$ generated as above, and our task is solving a linear system $\A\x=\b$ where $\b$ is generated from the standard normal distribution. On the other hand, for the task of evaluating \algpd, we use \verb~make_low_rank_matrix~ to construct a $4096\times 4096$ matrix $\mPhi$, and then compute $\A=\mPhi\mPhi^\top$. We then solve a PSD linear system $(\A+\phi\I)\x=\b$, with standard normal $\b$ and $\phi=0.001$. Note that we choose $m$ and $n$ to be powers of $2$ simply for the convenience of implementing randomized Hadamard transform (RHT). In general this is not necessary, since we can still implement it by finding the closest power of $2$ larger than $m$ or $n$, enlarging the matrix to that dimension by padding with $0$ entries, and truncating back to the original dimension at the end.

\paragraph{\textbf{2. Kernel Matrices from Machine Learning.}}
To evaluate our algorithm in a practical setting that naturally exhibits large outlying eigenvalues, we consider applying a kernel transformation to four real-world datasets (Abalone and Phoneme from OpenML \cite{OpenML2013}, California\_housing and Covtype from Scikit-learn \cite{scikit-learn}). We truncate each dataset to its first $n=4096$ rows to get matrix $\mPhi \in \R^{n \times m}$. For $(i,j) \in [n] \times [n]$, we define the kernel matrix $\A \in \mathcal{S}_n^{+}$ so that $\A_{ij} = \mathcal{K}(\mPhi_i, \mPhi_j)$ where $\mPhi_i, \mPhi_j$ are the $i$-th and $j$-th rows of $\mPhi$, respectively, and $\mathcal{K}$ is a kernel function. We consider two types of kernel functions: Gaussian, $\mathcal{K}(\mPhi_i,\mPhi_j)=\exp(-\gamma\|\mPhi_i-\mPhi_j\|^2 )$, and Laplacian, $\mathcal{K}(\mPhi_i,\mPhi_j)=\exp(-\gamma\|\mPhi_i-\mPhi_j\|)$. For both choices, we set the width parameter $\gamma$  among two values, $\{0.1,0.01\}$. This leads to four different test matrices for each of the four datasets, giving a total of 16 matrices. For each matrix, we solve a PSD linear system of the form $(\A+\phi\I)\x=\b$ with standard normal $\b$ and $\phi=0.001$.

\subsection{Testing Projection, Acceleration and Memoization}
\label{s:experiment_acc}
In this section we test the effect of (i) computing the inner projection steps inexactly, (ii) the adaptive acceleration scheme and (iii) the block memoization technique used in our methods (see Figure~\ref{fig:lsqr_steps} and \ref{fig:accel-sample} in Section \ref{s:experiments}, as well as Figure \ref{fig:accel-part1} below). 
For all tasks, we used variants of \alg (Algorithm~\ref{alg:kzpp}) for solving rectangular linear systems. 

To test the the effect of computing the inner projection steps inexactly, we consider two variants of our method:
\begin{itemize}
    \item \texttt{K++(LSQR-X)} for \texttt{X}\,$\in\!\{2,4,8\}$: Algorithm~\ref{alg:kzpp} as given in the pseudocode, with the number of inner LSQR iterations $t_{\max}$ set to \texttt{X};
    \item \texttt{K++(Cholesky)}: \alg\ with exact projections, i.e., Algorithm~\ref{alg:kzpp} with function \texttt{Proj} replaced by a direct solve involving $\mathrm{chol}(\A_S\A_S^\top+\lambda\I)$.
\end{itemize}

We observed that $8$ inner iterations consistently suffices in attaining convergence that nearly matches \alg\ with exact projections. Thus, for the remaining experiments, we use $t_{\max}=8$ as the default. Next, we consider four variants of our method, as shown in Table~\ref{tab:kacz-variants}, to identify the effect of its different components individually. For clarity, we explain their differences below.

\begin{itemize}
    \item \texttt{Kaczmarz}: The classical randomized block Kaczmarz method, without acceleration or block memoization, but still preprocessed with RHT;
    \item \algmemo: Randomized Kaczmarz with block memoization, but without adaptive acceleration, i.e., Algorithm~\ref{alg:kzpp} without lines \ref{kzpp:update_start}-\ref{kzpp:update_end}, and with $\eta = 0$ (as opposed to $\eta = \frac{s}{2n}$), meaning that we no longer maintain the adaptive momentum term $\m_t$ from line \ref{kzpp:momentum};
    \item \algaccel: Randomized Kaczmarz with adaptive acceleration, but without block memoization, i.e.,  Algorithm~\ref{alg:kzpp} modified so that  in line \ref{kzpp:sampling} we always choose to sample a new block $S\sim {[m]\choose s}$, and hence, do not save the Cholesky factors;
    \item \texttt{Full K++}: Algorithm~\ref{alg:kzpp} as given in the pseudocode, equipped with both adaptive acceleration and block memoization.
\end{itemize}

\begin{algorithm}[!ht]
\caption{\alg: Kaczmarz type solver for general systems}
\label{alg:kzpp}
\begin{algorithmic}[1]
\State \textbf{Input:} $\A\in\R^{m\times n}$, $\b \in \R^m$, block size $s$, iterate $\x_0$, regularization $\reg$, inner iterations $t_{\max}$, tolerance $\epsilon$;
\State Sample $\D\leftarrow \frac1{\sqrt m}\diag(d_1,...,d_m)$\text{ for }$d_i\sim\mathrm{Rademacher}$;
\State $\A\leftarrow \FHT(\D\A)$, $\b\leftarrow \FHT(\D\b)$;
\Comment{{\footnotesize Preprocessing with RHT.}}
\State Initialize $\m_0 \leftarrow \mathbf{0}$, $\rho\leftarrow 0$, $\eta\leftarrow \frac s{2n}$, $\Bc\leftarrow \emptyset$, $\zeta\leftarrow \lceil m/s\rceil$, $\Ec_0,\Ec_1\leftarrow 0$, $\tau\leftarrow 2s$;
\For{$t=0,1,...$}
\State \textbf{if} Bernoulli$\big(\min\{\,1,\ \frac1t\cdot \frac ms\log m, \frac1t\cdot \frac ns\log m\,\}\big)$ \label{kzpp:sampling}
\textbf{then}
\State \quad\ $\Bc \leftarrow \Bc \cup\{S\}$ \text{for} $S\sim {[m]\choose s}$; \Comment{{\footnotesize Sample new subset.}}
\State \textbf{else} $S\sim \Bc$; \textbf{end if} 
\State $\r_t \leftarrow \A_S\x_t-\b_S$;\Comment{{\footnotesize Use for error estimation.}}
\State $\w_t \leftarrow \mathsf{Proj}(\A, S, \Bc, \r_t, \tau, \lambda, t_{\max})$; \Comment{{\footnotesize Inner LSQR solver.}}
\State $\m_{t+1} \leftarrow \frac{1-\rho}{1+\rho}\big(\m_t - \w_t\big)$;\Comment{{\footnotesize Adaptive momentum.}} \label{kzpp:momentum}
\State $\x_{t+1} \leftarrow \x_t - \w_t + \eta\,\m_{t+1}$;
\State \textbf{if} $t<\zeta \mod 2\zeta$ \textbf{then} $\Ec_0 \leftarrow \Ec_0+\|\r_t\|^2$; \textbf{else} $\Ec_1 \leftarrow \Ec_1+\|\r_t\|^2$; \label{kzpp:update_start}
\If{$t = 2\zeta-1 \mod 2\zeta$}
\State \textbf{if} $\Ec_1\leq \epsilon^2\|\b\|^2$ \textbf{then} \textbf{return $\x_{t+1}$}; \Comment{{\footnotesize Stopping criterion.}}
\State $r\leftarrow r a_t + (\Ec_1/\Ec_0) (1-a_t)$; \Comment{{\footnotesize Weighted average \eqref{eq:weighted}.}}
\State $\rho\leftarrow 1 - r^{1/\zeta}$;
\Comment{{\footnotesize Convergence rate estimate.}}
\State $\Ec_0,\Ec_1\leftarrow 0$;
\EndIf  \label{kzpp:update_end}
\EndFor
\end{algorithmic}
\end{algorithm}

\begin{algorithm}[!ht]
\caption{Preconditioned LSQR solver $\mathsf{Proj}(\A, S, \Bc, \r, \tau, \lambda, t_{\max})$}
\begin{algorithmic}[1]
\State \textbf{Input: }$\A_S \in \R^{s \times n}, \r \in \R^s$, sketch size $\tau$, regularization $\lambda$, max iterations $t_{\max}$;
\State \textbf{if} $S\not\in \Bc$ \textbf{then}
\State \quad $\hat{\A} \leftarrow\A_S \mPi^\top$; \Comment{{\footnotesize $\mPi \in \R^{\tau \times n}$ is a sketching matrix.}}
\State \quad $\Rb[S] \leftarrow \textrm{chol}(\hat{\A} \hat{\A}^\top+\lambda\I)$; 
\Comment{{\footnotesize Save Cholesky factor.}}
\State \textbf{end if}
\State $\M \leftarrow \Rb[S]^{-\top}\big[\A_S\ \sqrt\lambda\I\big],\ \ \b \leftarrow \Rb[S]^{-\top}\r_t$;
\Comment{{\footnotesize Preconditioned system (implicit).}}
\State $\x \leftarrow \textrm{LSQR}(\M, \b, t_{\max})$;
\State \textbf{Return} $\w \leftarrow \x_{1:n}$
\end{algorithmic}
\end{algorithm}

Throughout this section, we set the block size $s$ to be $\{100, 200\}$, and measure the convergence through the residual error defined as $\epsilon_t \coloneqq\|\A \x_t - \b\| / \|\b\|$. For each plot we run all methods $10$ times and take an average.

By comparing the curves, we can see that in all cases, adding block memoization slightly worsens the convergence rate (by comparing \texttt{Kaczmarz} and \algmemo, or \algaccel and \texttt{Full K++}). This is actually suggested by theory, since with block memoization we are essentially not sampling all possible blocks of coordinates, thus reducing the quality of the block sampling distribution. However, this phenomenon is very insignificant especially when we compare \algaccel with \texttt{Full K++}, which suggests that our online block selection scheme (Section~\ref{s:block_memoization}) works well, and that $\tilde O(\min\{m,n\}/s)$ blocks are sufficient for fast convergence, matching our theory. The computational benefits of block memoization are observed in the FLOPs experiments (Figure \ref{fig:flops-rest}), once we plot the convergence in terms of arithmetic operations, taking advantage of the saved Cholesky factors.

From the plots, we can also see that the adaptive acceleration plays an important role for both Kaczmarz and Kaczmarz with block memoization, showing a comparable improvement for both cases. We also observe that the effect of adaptive acceleration is more significant when the sketch size $s$ is smaller - this also aligns with the theory, since for smaller $s$ the tail Demmel condition number $\bar{\kappa}_s$ is larger, thus the effect of acceleration reducing the dependence on $\bar{\kappa}_s$ from second to first power is also more significant.

\begin{figure}
\begin{center}
\includegraphics[width=0.94\linewidth]{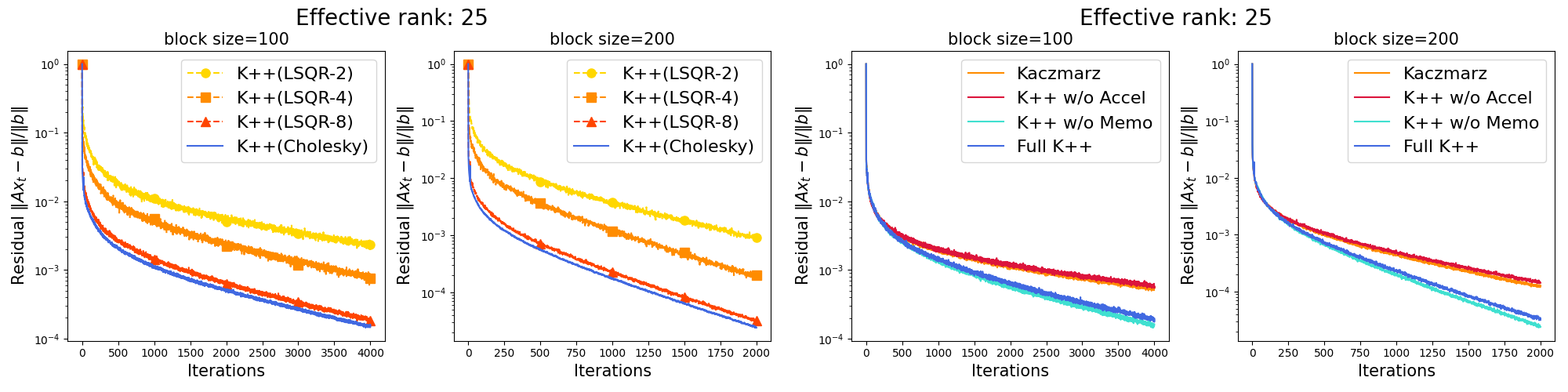}
\includegraphics[width=0.94\linewidth]{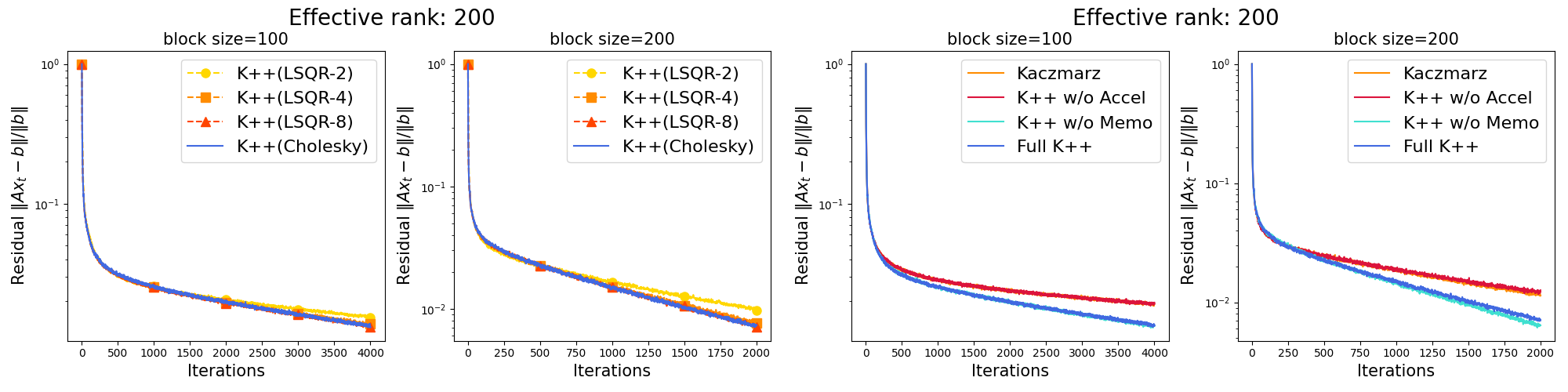}
\end{center}
\vspace{-5mm}
\caption{Convergence plots testing inexact projection steps and different variants of \textnormal{\algshort} (see Table \ref{tab:kacz-variants}), using two block sizes (continuation of Figure \ref{fig:lsqr_steps} and \ref{fig:accel-sample} with effective ranks 25 and 200).}
\label{fig:accel-part1}
\end{figure}

\subsection{Comparison with Krylov Subspace Methods}
\label{s:experiment_flops}

Next, we evaluate the computational cost of our algorithms on kernel matrices based on benchmark machine learning data. Since these are PSD linear systems, we evaluate the convergence of \texttt{Full CD++}, alongside \texttt{CD++ w/o Memo} and \texttt{CD++ w/o RHT}, comparing against Krylov-type methods, conjugate gradient (CG) and GMRES. In this experiment we count the FLOPs it takes to converge for each method, to showcase the advantages of our methods compared to Krylov-type methods.

\begin{figure}[ht]
\centering
\includegraphics[width=0.94\linewidth]{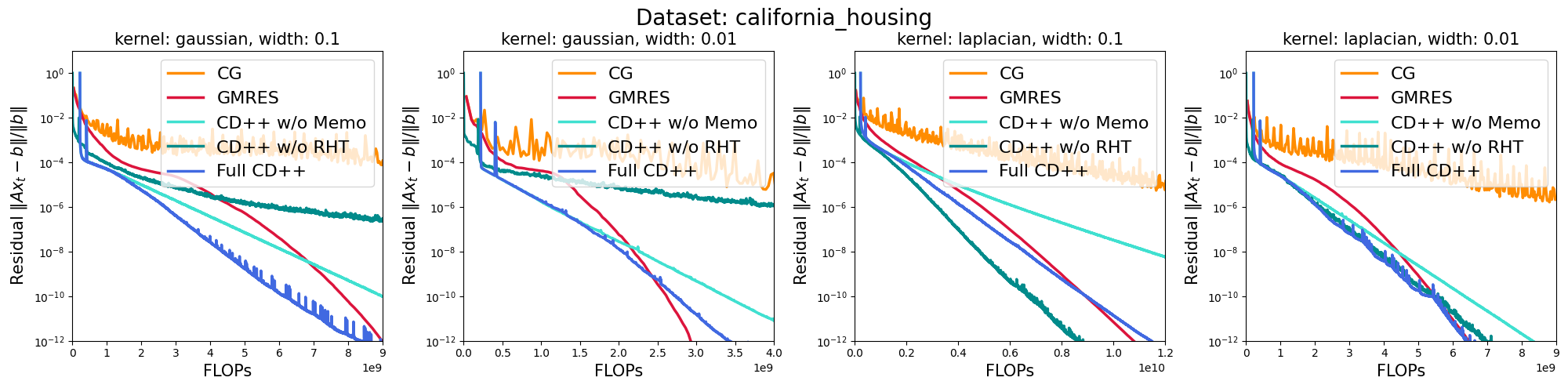}
\includegraphics[width=0.94\linewidth]{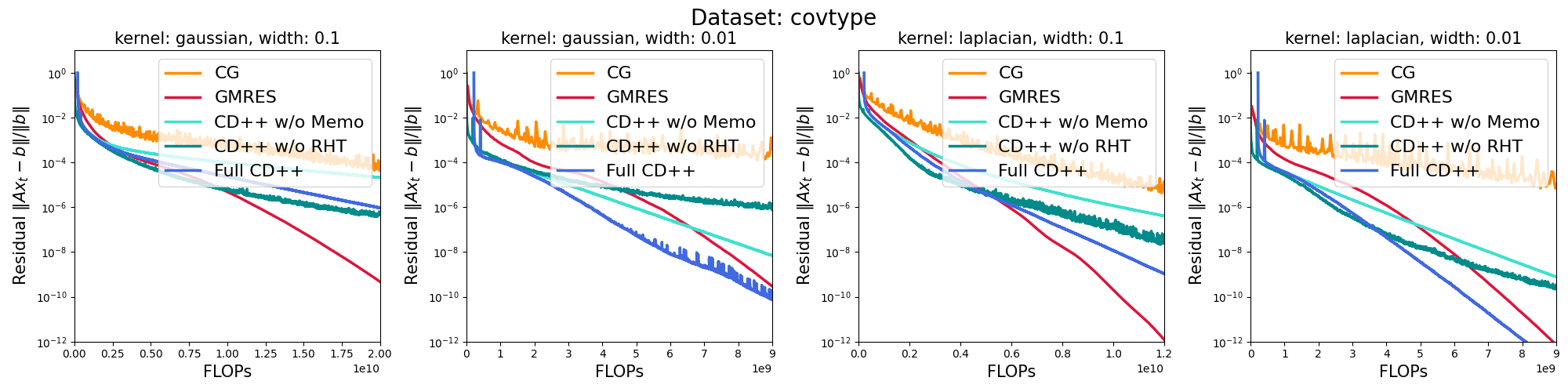}
\includegraphics[width=0.94\linewidth]{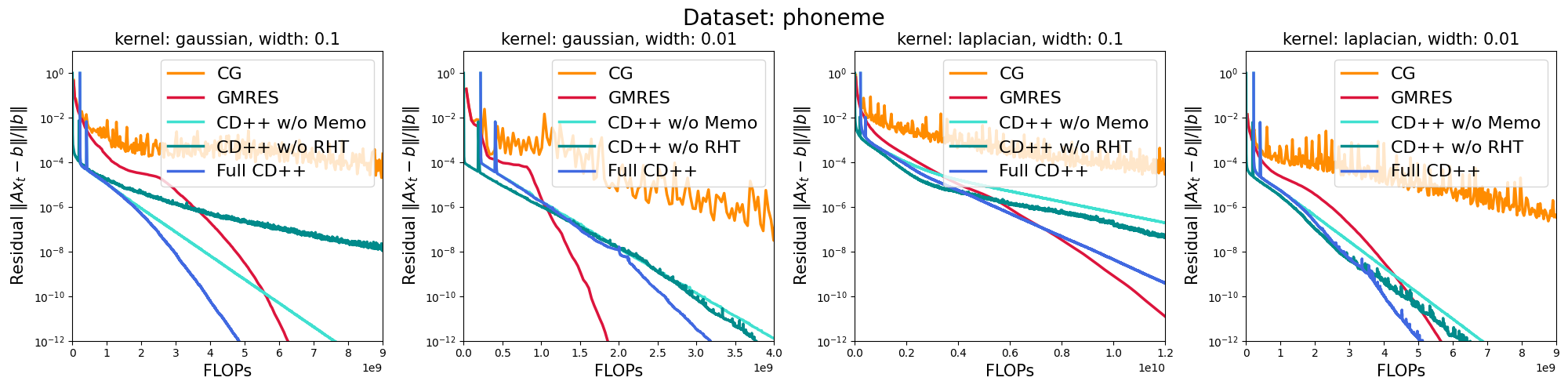}
\includegraphics[width=0.94\linewidth]{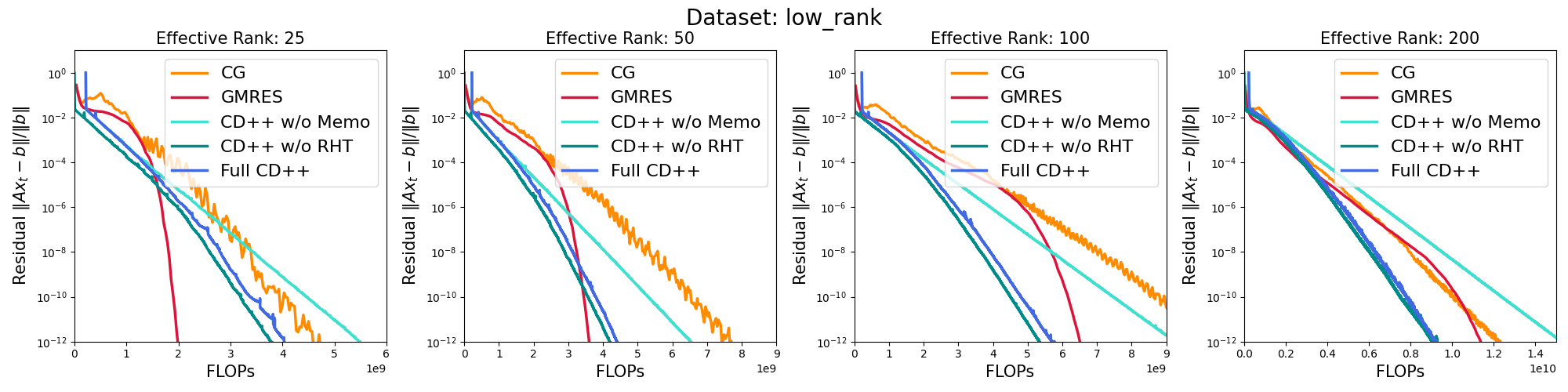}
\vspace{-2mm}
\caption{Computational cost comparison, measuring floating point operations (FLOPs) against the normalized residual error \eqref{eq:residual-error} for \texttt{\emph{Full CD++}}, alongside baselines \texttt{\emph{CD++ w/o Memo}}, CG, and GMRES, and also \texttt{\emph{CD++ w/o RHT}} (continuation of Figure \ref{fig:flops-sample}).}
\label{fig:flops-rest}
\end{figure}

For the FLOPs of CG, we look into the source code from \href{https://github.com/scipy/scipy/blob/v1.14.1/scipy/sparse/linalg/_isolve/iterative.py#L283-L388}{scipy.sparse.linalg} and approximate it by $2n^2 + 11n$ per iteration. For the FLOPs counting of GMRES, we look into the source code from \href{https://pyamg.readthedocs.io/en/latest/_modules/pyamg/krylov/_gmres.html#gmres}{pyamg.krylov} and approximate it by $2n^2T + 4n T (T + 1)$, where $T$ is the number of iterations. For the FLOPs counting of our methods (\texttt{Full CD++}, \texttt{CD++ w/o Memo}, \texttt{CD++ w/o RHT}), we maintain a counter for FLOPs for each iteration.  which includes both the cases of using block memoization (\texttt{Full CD++}, \texttt{CD++ w/o RHT}) and not using it (\texttt{CD++ w/o Memo}). Specifically, for the algorithms with block memoization, if a new block is sampled, then we do the Cholesky factorization and increase the FLOPs by $s^3/3$, where $s$ is the block size; if we sample from the already sampled set of blocks $\Bc$, then Cholesky factorization is not needed and this cost is omitted. Moreover, for \texttt{CD++ w/o Memo} and \texttt{Full CD++}, there is a pre-processing step of applying the RHT which takes extra FLOPs at the beginning (this is omitted from \texttt{CD++ w/o RHT}). We count them by following Appendix~\ref{s:symfht}, and this is reflected in the plots, since the convergence curves of our methods are shifted by the cost of the RHT. Here, we take advantage of our fast SymFHT implementation (Algorithm \ref{alg:symfht}), which reduces that cost by half. Throughout the section, we choose the block size $s = 200$, which tends to work well for the size of our test matrices. Further optimizing the block size, or choosing it dynamically, is an interesting direction for future work. We run each algorithm $5$ times and take average to reduce the noise.

From the experiments (Figure \ref{fig:flops-rest}), we can see that CG converges very slowly and is not comparable with other methods. 
By comparing \texttt{CD++ w/o Memo} and \texttt{Full CD++}, we can see that the block memoization technique gives a significant improvement in FLOPs, since \texttt{Full CD++} successfully reduces the expensive Cholesky factorization step. This improvement is more significant if we are running more iterations (i.e., aiming for higher accuracy). By comparing \texttt{Full CD++} and GMRES, we can see that in most cases \texttt{Full CD++} performs better in the ``low to medium accuracy level'', while GMRES sometimes beats \texttt{Full CD++} in the ``high accuracy level''. 

For example, for the abalone dataset with Gaussian kernel with width=$0.01$, GMRES starts to perform better after the residual reaches $10^{-7}$. However, we note that for abalone, phoneme and california\_housing dataset with Gaussian kernel with width=$0.1$, or with Laplacian kernel with width=$0.01$, our \texttt{Full CD++} outperforms GMRES even when the residual reaches $10^{-12}$. These phenomena are dependent on the spectrum (especially the top eigenvalues) of matrix $\A + \phi \I$, which is reflected in the different choices of kernel and width. 

\begin{table}[]
\centering
\footnotesize
\resizebox{\textwidth}{!}{
\begin{tabular}{|c|c|c|cc|cc|cc|cc|}
\hline
\multicolumn{3}{|c|}{\textbf{Problem}} & \multicolumn{2}{c|}{\textbf{CG}} & \multicolumn{2}{c|}{\textbf{GMRES}} & \multicolumn{2}{c|}{\textbf{\texttt{Full CD++}}} & \multicolumn{2}{c|}{\textbf{\texttt{CD++ w/o RHT}}} \\
\hline
\textbf{Dataset} & \textbf{Kernel} & \textbf{Width} & \textbf{1e-4} & \textbf{1e-8} & \textbf{1e-4} & \textbf{1e-8} & \textbf{1e-4} & \textbf{1e-8} & \textbf{1e-4} & \textbf{1e-8} \\
\hline\hline

\multirow{4}{*}{\textbf{Abalone}}
& \multirow{2}{*}{\textbf{Gaussian}} & 0.1 & 7.63e9 & 3.06e10 & 1.47e9 & 5.17e9 & \textbf{4.64e8} & \textbf{3.26e9} & 6.68e8 & 8.97e9 \\
\cline{3-11}
& & 0.01 & 1.88e9 & 6.89e9 & 8.50e8 & \textbf{1.75e9} & \textbf{2.97e8} & 2.11e9 & 4.04e8 & 2.71e9 \\
\cline{2-11}
& \multirow{2}{*}{\textbf{Laplacian}} & 0.1 & 8.64e9 & 3.22e10 & 2.72e9 & 8.14e9 & {2.22e9} & {8.13e9} & {\textbf{1.74e9}} & {\textbf{6.06e9}} \\
\cline{3-11}
& & 0.01 & 1.11e9 & 1.37e10 & 5.41e8 & 4.11e9 & {2.40e8} & {3.09e9} & {\textbf{2.20e7}} & {\textbf{3.06e9}} \\
\hline

\multirow{4}{*}{\textbf{Phoneme}}
& \multirow{2}{*}{\textbf{Gaussian}} & 0.1 & 7.80e9 & 3.29e10 & 1.79e9 & 4.98e9 & {4.86e8} & \textbf{3.23e9} & {\textbf{4.61e8}} & 1.00e10 \\
\cline{3-11}
& & 0.01 & 3.70e8 & 4.23e9 & 3.37e8 & \textbf{1.33e9} & {2.23e8} & 1.80e9 & {\textbf{4.39e6}} & 1.79e9 \\
\cline{2-11}
& \multirow{2}{*}{\textbf{Laplacian}} & 0.1 & 7.22e9 & 3.73e10 & 2.29e9 & \textbf{8.46e9} & {1.65e9} & 8.96e9 & {\textbf{1.35e9}} & 1.31e10 \\
\cline{3-11}
& & 0.01 & 2.49e9 & 1.41e10 & 7.81e8 & 3.97e9 & {2.80e8} & {3.10e9} & {\textbf{7.47e7}} & {\textbf{3.01e9}} \\
\hline

\multirow{4}{*}{\textbf{\shortstack{California \\ Housing}}}
& \multirow{2}{*}{\textbf{Gaussian}} & 0.1 & 2.45e9 & 3.27e10 & 1.02e9 & 6.14e9 & \textbf{2.40e8} & \textbf{3.88e9} & {3.69e8} & 9.74e9 \\
\cline{3-11}
& & 0.01 & 2.15e9 & 9.51e9 & 5.75e8 & 2.18e9 & \textbf{2.27e8} & \textbf{1.99e9} & {3.64e8} & 7.13e9 \\
\cline{2-11}
& \multirow{2}{*}{\textbf{Laplacian}} & 0.1 & 5.44e9 & 2.34e10 & 2.04e9 & 6.85e9 & {1.52e9} & {6.39e9} & {\textbf{1.24e9}} & {\textbf{4.56e9}} \\
\cline{3-11}
& & 0.01 & 1.31e9 & 1.36e10 & 5.75e8 & 4.15e9 & {2.40e8} & {3.06e9} & {\textbf{4.39e7}} & \textbf{3.45e9} \\
\hline

\multirow{4}{*}{\textbf{Covtype}}
& \multirow{2}{*}{\textbf{Gaussian}} & 0.1 & 1.65e10 & 5.55e10 & {4.87e9} & \textbf{1.70e10} & 5.71e9 & 3.64e10 & {\textbf{3.63e9}} & 3.11e10 \\
\cline{3-11}
& & 0.01 & 9.31e9 & 3.87e10 & 2.00e9 & 7.69e9 & \textbf{9.78e8} & \textbf{5.66e9} & {1.26e9} & 1.22e10 \\
\cline{2-11}
& \multirow{2}{*}{\textbf{Laplacian}} & 0.1 & 6.79e9 & 2.24e10 & 3.23e9 & \textbf{8.34e9} & {2.85e9} & 1.03e10 & {\textbf{1.88e9}} & 8.61e9 \\
\cline{3-11}
& & 0.01 & 3.60e9 & 2.37e10 & 1.23e9 & 5.95e9 & {4.60e8} & \textbf{4.56e9} & {\textbf{3.43e8}} & {5.18e9} \\
\hline\hline

\multirow{4}{*}{\textbf{\shortstack{Synthetic \\ Low-Rank}}}
& \multicolumn{2}{c|}{Effective rank $= 25$} & 1.68e9 & 3.26e9 & 1.44e9 & \textbf{1.83e9} & {1.31e9} & 2.79e9 & {\textbf{1.11e9}} & 2.44e9\\
\cline{2-11}
& \multicolumn{2}{c|}{Effective rank $= 50$} & 2.59e9 & 5.17e9 & 2.43e9 & 3.26e9 & {1.53e9} & {3.21e9} & {\textbf{1.34e9}} & {\textbf{2.92e9}}\\
\cline{2-11}
& \multicolumn{2}{c|}{Effective rank $= 100$} & 3.16e9 & 6.89e9 & 2.65e9 & 5.67e9 & \textbf{1.91e9} & \textbf{3.89e9} & {1.94e9} & {4.16e9} \\
\cline{2-11}
& \multicolumn{2}{c|}{Effective rank $= 200$} & 3.26e9 & 8.00e9 & {2.75e9} & 8.34e9 & 2.92e9 & {6.10e9} & {\textbf{2.69e9}} & {\textbf{5.77e9}} \\
\hline
\end{tabular}
}
\vspace{1mm}
\caption{Comparison of the FLOPs needed to achieve given error threshold, $\epsilon \in \{10^{-4},10^{-8}\}$, for different algorithms. Bold values indicate the best performance for a given error threshold. The last column is included to test the effect of RHT preprocessing.}
\vspace{-2mm}
\label{tab:flops}
\end{table}

In Table~\ref{tab:flops} we show the FLOPs it takes for different methods (CG, GMRES, \texttt{Full CD++}, and \texttt{CD++ w/o RHT}) to achieve the given accuracy in detail. Here we set $\epsilon=10^{-4}$ as the mid-level accuracy and $\epsilon=10^{-8}$ as the high-level accuracy. For mid-level accuracy, we can see that \texttt{Full CD++} outperforms GMRES in 18 out of 20 tasks, showing that our method converges very fast at early stages. For high-level accuracy, we can see that \texttt{Full CD++} still outperforms GMRES at 14 out of 20 tasks. 

\paragraph{\textbf{Effect of RHT.}}
We also evaluate the effect of RHT as a preprocessing step (suggested by theory) on the total cost of our algorithm by comparing \texttt{CD++ w/o RHT} and \texttt{Full CD++}. Note that the only difference here is whether or not to use RHT, while the uniform sampling scheme remains the same. By comparing these two curves in Figure~\ref{fig:flops-rest}, we can see that the RHT step itself is very cheap, and can be measured by how much the beginning of the convergence curve of \texttt{Full CD++} is shifted away from 0 (given that the size of the matrices is $4096\times 4096$, the cost of the RHT step is approximately 2.4e8 FLOPs). This cost is almost negligible in most cases, however, it is noteworthy that in some cases (e.g., Abalone dataset with Laplacian kernel with width 0.01), the cost of RHT dominates the cost of the algorithm when converging to mid-level accuracy. Among all the $20$ test matrices (in Figure~\ref{fig:flops-sample} and \ref{fig:flops-rest}), RHT provides a significant gain in overall computational performance in 6 cases, and a smaller gain in 3 others. Despite this, even without RHT preprocessing, \texttt{CD++ w/o RHT} still exhibits good convergence for most of the problems and is comparable with GMRES, indicating that our algorithm can be effective even for sparse linear systems, where RHT is not desirable.

To conclude, the experiments in this section show that in the measurement of FLOPs, our \algpd outperforms Krylov methods including CG and GMRES in almost all mid-level accuracy tasks, as well as most high-level accuracy tasks. We also show the feasibility of discarding RHT preprocessing step in our algorithm when the input matrix is sparse.

\subsection{Testing Regularization in Projection}\label{s:experiments-regularization}
In this section, we test the effect of explicitly adding regularization to the projection step in \alg/\algpd\ (parameter $\lambda$ in Algorithm \ref{alg:bcd}). We test this experiments on our \texttt{Full CD++} method with $\lambda = \{\text{1e-2}, \text{1e-4}, \text{1e-6}, \text{1e-8}, \text{1e-10}, \text{0}\}$, where in the case of $\lambda = \text{1e-8}$, this recovers the \texttt{Full CD++} used in the remaining experiments. As we can see in Figure \ref{fig:regularization} adding this regularization term does not have a significant effect on the convergence rate of \texttt{Full CD++} (we do not include the plots for the remaining 3 real-world datasets, since they all show the same phenomena). This shows that adding regularization does not sacrifice the convergence rate. Recall that regularization has the benefit of making the computation of the Cholesky factors more stable: in the case where we solve a positive semidefinite (i.e., $\phi=0$) linear system the Cholesky factorization step can be potentially numerically unstable if the block matrix $\A_{S,S}$ is singular. Thus, we recommend using \algpd with a small but positive $\lambda$ to ensure numerical stability.

\begin{figure}[ht]
\includegraphics[width=0.94\linewidth]{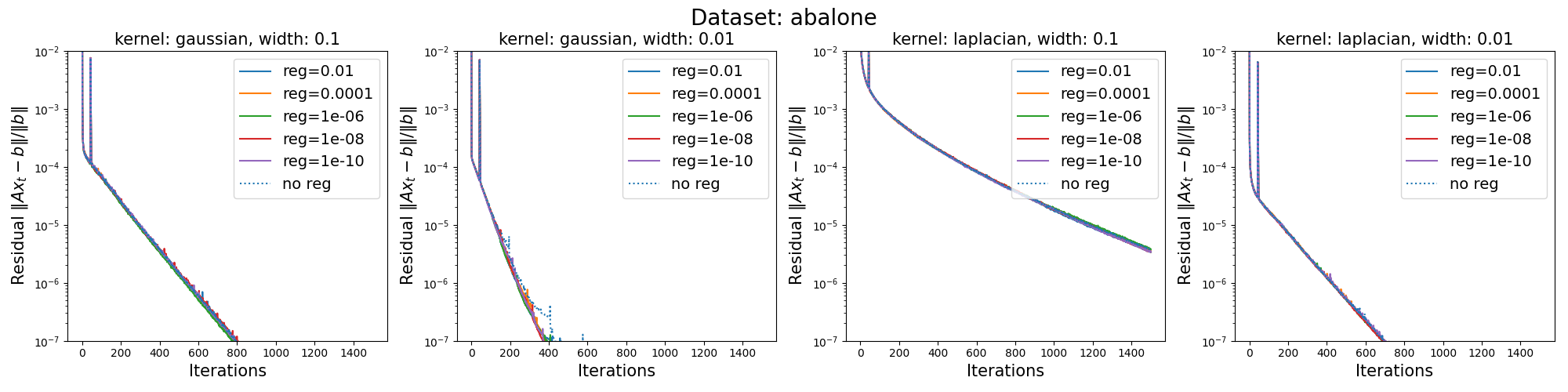}
\includegraphics[width=0.94\linewidth]{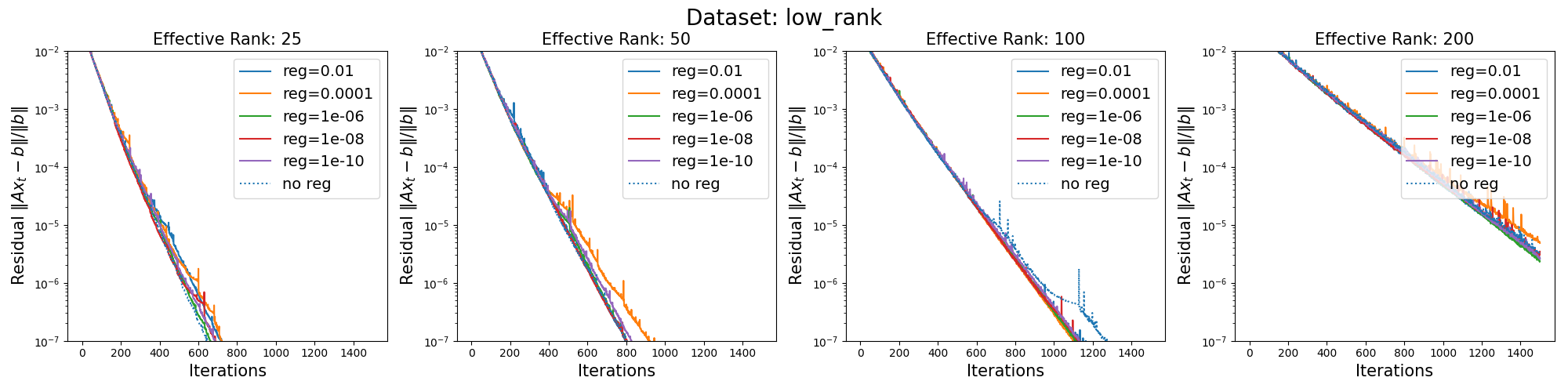}
\caption{Convergence plots showing the stability of \algpd\ with respect to the choice of Tikhonov regularization parameter $\lambda$ in the inner step of \algpd. Check \url{https://github.com/EdwinYang7/kaczmarz-plusplus} for plots on remaining real-world datasets.}
\label{fig:regularization}
\end{figure}

\end{document}